\documentclass{amsart}
\usepackage{amssymb,amsmath,amsfonts,amsthm,epsfig,amscd,hyperref}
\usepackage{stmaryrd,comment}
\usepackage[all,cmtip,poly]{xy}
\usepackage{svn, url}
\usepackage{mathrsfs}
\usepackage[T1]{fontenc}

\newtheorem{theorem}[subsection]{Theorem}
\newtheorem{thm}[subsubsection]{Theorem}

\newtheorem{lem}[subsubsection]{Lemma}

\newtheorem{cor}[subsubsection]{Corollary}

\newtheorem{prop}[subsubsection]{Proposition}

\newtheorem{aprop}[subsection]{Proposition}

\newtheorem{defn}[subsubsection]{Definition}

\theoremstyle{remark}
\newtheorem{remark}[subsubsection]{Remark}
\newtheorem{rem}[subsubsection]{Remark}
\newtheorem{aremark}[subsection]{Remark}

\numberwithin{equation}{subsection}
\def\nummultline{\addtocounter{subsubsubsection}{1}\begin{multline}}
\def\anumequation{\addtocounter{subsection}{1}\begin{equation}}

\newif\iffinalrun
\iffinalrun
\else
\fi

\iffinalrun
  \newcommand{\need}[1]{}
  \newcommand{\mar}[1]{}
\else
  \newcommand{\need}[1]{{\tiny *** #1}}
  \newcommand{\mar}[1]{\marginpar{\raggedright\tiny #1}}
\fi

\newcommand{\A}{\AA}
\newcommand{\C}{\CC}
\newcommand{\F}{\FF}

\newcommand{\Q}{\QQ}
\newcommand{\R}{\RR}
\newcommand{\Z}{\ZZ}

\newcommand{\m}{\frakm}

\renewcommand{\AA}{{\mathbb A}}

\newcommand{\CC}{{\mathbb C}}

\newcommand{\FF}{{\mathbb F}}

\newcommand{\QQ}{{\mathbb Q}}
\newcommand{\RR}{{\mathbb R}}

\newcommand{\ZZ}{{\mathbb Z}}
\newcommand{\tor}{\mathrm{tor}}

\renewcommand{\bf}{\ensuremath{\mathbf{f}}}

\newcommand{\cA}{{\mathcal A}}
\newcommand{\cB}{{\mathcal B}}

\newcommand{\cF}{{\mathcal F}}
\newcommand{\cG}{{\mathcal G}}
\newcommand{\cH}{{\mathcal H}}

\newcommand{\cK}{{\mathcal K}}

\newcommand{\cM}{{\mathcal M}}

\newcommand{\cO}{{\mathcal O}}
\newcommand{\cP}{{\mathcal P}}

\newcommand{\cS}{{\mathcal S}}
\newcommand{\cT}{{\mathcal T}}
\newcommand{\cU}{{\mathcal U}}
\newcommand{\cV}{{\mathcal V}}

\newcommand{\frakm}{\mathfrak{m}}

\DeclareMathOperator{\Aut}{Aut}

\DeclareMathOperator{\Gal}{Gal}
\DeclareMathOperator{\GL}{GL}

\DeclareMathOperator{\Hom}{Hom}

\DeclareMathOperator{\SL}{SL}

\DeclareMathOperator{\Spec}{Spec}
\DeclareMathOperator{\Spa}{Spa}

\DeclareMathOperator{\Spf}{Spf}

\DeclareMathOperator{\Lie}{Lie}

\newcommand{\Frob}{\mathrm{Frob}}

\newcommand{\HT}{\mathrm{HT}}

\newcommand{\ur}{\mathrm{ur}}
\newcommand{\tr}{\mathrm{tr}}

\newcommand{\et}{\mathrm{\acute{e}t}}

\newcommand{\toisom}{\buildrel\sim\over\to}

\newcommand{\Fl}{\mathscr{F}\!\ell}
\newcommand{\Spd}{\mathrm{Spd}}

\newcommand{\Spl}{\mathrm{Spl}}
\newcommand{\Ib}{\mathrm{Ig}^b}

\newcommand{\perf}{\mathrm{perf}}
\newcommand{\qproet}{\mathrm{qpro\acute{e}t}}
\newcommand{\cont}{\mathrm{cont}}
\newcommand{\rk}{\mathrm{rk}}

\begin{document}
\title[Generic cohomology of non-compact unitary Shimura varieties]
{On the generic part of the cohomology of non-compact unitary Shimura varieties}

\author{Ana Caraiani} \email{caraiani.ana@gmail.com}
\address{Department of
  Mathematics, Imperial College London,
  London SW7 2AZ, UK}
  
 \author{Peter Scholze}\email{scholze@mpim-bonn.mpg.de}
 \address{Max-Planck Institut f\"ur Mathematik, Bonn, Germany} 

\begin{abstract}We prove that the generic part of the mod l cohomology of Shimura
varieties associated to quasi-split unitary groups of even dimension is concentrated
above the middle degree, extending our previous work to a non-compact case.
The result applies even to Eisenstein cohomology classes coming from the locally
symmetric space of the general linear group, and has been used in joint work with 
Allen, Calegari, Gee, Helm, Le Hung, Newton, Taylor and Thorne
to get good control on these classes and deduce potential automorphy theorems
without any self-duality hypothesis.

Our main geometric result is a computation of the fibers of the Hodge–Tate
period map on compactified Shimura varieties, in terms of similarly compactified
Igusa varieties.
\end{abstract}

\maketitle

\tableofcontents

\section{Introduction}

In this paper, we aim to control the torsion in the cohomology of some non-compact unitary Shimura varieties, extending previous results \cite{caraiani-scholze} to the non-compact case. Before explaining our results, let us explain our motivation.

The original Taylor--Wiles method for proving automorphy lifting theorems, used for example in proving the modularity of elliptic curves over $\Q$, is restricted to settings where a certain numerical criterion holds: these are roughly the settings where the objects on the automorphic side arise from the middle degree cohomology of a Shimura variety. Recently, Calegari--Geraghty~\cite{calegari-geraghty} outlined a strategy for extending the Taylor--Wiles method to $\GL_n$ over a general CM field $F$. Their method requires a detailed understanding of the cohomology of the locally symmetric spaces associated with $\GL_n/F$. Part of their insight was to realize the central role played by torsion classes in the cohomology of these locally symmetric spaces, and another part of their insight was to reinterpret the failure of the Taylor--Wiles numerical criterion in terms of invariants $q_0$, $l_0$ seen on the automorphic side. 

More precisely, for a connected reductive group $G/\Q$, let $l_0:= \rk(G(\R))-\rk(K_\infty)-\rk(A_\infty)$ and $q_0:=\frac{1}{2}(d-l_0)$.\footnote{Here $K_\infty \subseteq G(\R)$ is a maximal compact subgroup, $A_{\infty}$ is the identity component of the $\R$-points of the maximal $\Q$-split torus in the center of $G$, and $d$ is the dimension over $\R$ of the symmetric space for $G$.} The Calegari--Geraghty method works for $\GL_n$ over an arbitrary number field $F$ as long as the following prerequisites are in place:
\begin{enumerate}
\item The construction of Galois representations associated to classes in the cohomology with $\Z_{\ell}$-coefficients of the locally symmetric space for $\GL_n /F$.
\item Local-global compatibility for these Galois representations at all primes of $F$, including at primes above $\ell$.
\item A folklore conjecture that predicts that, under an appropriate non-degeneracy condition, the cohomology with $\Z_{\ell}$-coefficients of the locally symmetric space for $\GL_n /F$ vanishes outside the range of degrees $[q_0, q_0 + l_0]$. 
\end{enumerate}
\noindent From now on, assume that $F$ is a CM field. Then the first problem was solved in~\cite{scholze}, strengthening previous work~\cite{hltt} that applies to $\Q_{\ell}$-coefficients. (See also~\cite{newton-thorne} and~\cite{arizona} for some further refinements.) 

The remaining two problems were within reach with $\Q_{\ell}$-coefficients: see~\cite{varma} for local-global compatibility and~\cite[Theorem 2.4.9]{10Authors}, which builds on~\cite{franke}, for vanishing results. However, both problems remained largely open for torsion classes. In the second problem, local-global compatibility at primes dividing $\ell$ is particularly subtle, because the construction of Galois representations in~\cite{scholze} uses congruences to automorphic forms with arbitrarily deep level at $\ell$. The third problem is known only in low-dimensional cases, such as arithmetic hyperbolic three-manifolds (arising from $\GL_2$ over an imaginary quadratic field $F$).

When the locally symmetric spaces for the group $G$ arise from Shimura varieties, $l_0 = 0$ and $q_0$ is equal to the complex dimension of the Shimura variety. Originally, we were trying to understand the third problem, by approaching it in the easier case of Shimura varieties first; this is a result we obtained in \cite{caraiani-scholze} for compact unitary Shimura varieties. Our hope was that a suitable adaptation of these results to the case of non-compact unitary Shimura varieties, whose compactification contains the locally symmetric spaces for $\GL_n/F$, could give new information in that situation. This turns out to be the case: however, so far not about the third problem, but about the second problem.

This has been taken up in \cite{10Authors}, where the main result of the present paper is used to obtain local-global compatibility in the ordinary and (many) Fontaine--Laffaille cases. The third problem was overcome in the particular setting of~\cite{10Authors} via an alternative argument that reduces it to the case of $\Q_{\ell}$-coefficients, allowing us to implement the Calegari--Geraghty method unconditionally in arbitrary dimensions, and prove the meromorphic continuation of the $L$-function and the Sato--Tate conjecture for elliptic curves over CM fields (among other results).

In conclusion, the results of the present paper are tailored to give interesting information about $\GL_n/F$. For this reason, we restrict attention to the following specific case at hand, although our methods should extend to somewhat more general Shimura varieties. We discuss further (possible) generalizations in Remark~\ref{rem:koshikawa simplification}.

\subsubsection{Locally symmetric spaces} Fix a CM field $F$ and an integer $n\geq 1$. We assume that $F$ contains an imaginary quadratic field $F_0\subset F$; if we let $F^+\subset F$ be the maximal totally real subfield, we have $F=F^+\cdot F_0$. Let $V=F^{2n}$ be equipped with the skew-hermitian form
\[
\langle (x_1,\ldots,x_{2n}),(y_1,\ldots,y_{2n})\rangle = \sum_{i=1}^n (x_i\overline{y}_{2n+1-i}- x_{2n+1-i} \overline{y}_i)
\]
where $\overline{y}$ denotes the complex conjugate of $y$, and consider the associated alternating form
\[
(\cdot, \cdot): V\times V\to \mathbb Q: (x,y) = \tr_{F/\mathbb Q} \langle x,y\rangle\ .
\]
Then $V$ admits $\cO_F$-lattices $L\subset V$ that are self-dual with respect to $(\cdot,\cdot)$; fixing one, we get an alternating perfect pairing
\[
(\cdot, \cdot): L\times L\to \mathbb Z\ .
\]
Let $G$ be the algebraic group over $\mathbb Z$ defined by
\[
G(R):=\left\{g\in \GL_{\mathcal O_F}(L)(R) \mid (gv,gw) = (v,w)\ \forall v,w\in L \right\}\ .
\]
The generic fibre of $G$ is then a quasi-split unitary group.\footnote{In the main text, we will denote by $G$ a unitary similitude group. The translation is explained in \S \ref{sec:Shimura intro}.}

Over $\mathbb R$, we have $G_{\mathbb R} =\mathrm{U}(n,n)^{[F^+:\mathbb Q]}$. 
The associated symmetric space $X=G(\mathbb R)/K_\infty$ for $G(\mathbb R)$ is given by
\[
X=\prod_{\tau: F^+\hookrightarrow \R} X_{\tau,+}
\]
where $X_{\tau,+}$ is the space of positive definite $n$-dimensional subspaces in the split hermitian space $V\otimes_{F^+,\tau} \R\cong \C^{2n}$.

For any neat compact open subgroup $K\subset G(\mathbb A_f)$, we are interested in the double quotient
\[
X_K = G(\mathbb Q)\backslash (X\times G(\mathbb A_f)/K)\ .
\]
As detailed below, this has the structure of a complex manifold of complex dimension $d=[F^+:\mathbb Q]n^2$. The Borel--Serre compactification of $X_K$ includes strata related to the locally symmetric spaces for $\GL_n/F$; this explains our interest in $X_K$.

Assume that $K=\prod_p K_p$ is a product of compact open subgroups $K_p\subset G(\mathbb Q_p)$, and fix a finite set of primes $S$ containing all primes that ramify in $F$ or at which $K_p\neq G(\mathbb Z_p)$. The abstract unramified Hecke algebra
\[
\mathbb T = \mathbb T^S = \bigotimes_{p\not\in S,p\in \Spl_{F_0/\Q}} \mathbb Z[G(\mathbb Q_p)//G(\mathbb Z_p)]
\]
acts naturally on $H^i(X_K,A)$ and $H_c^i(X_K,A)$ for any coefficient module $A$; here we restrict to the primes $\Spl_{F_0/\Q}$ of $\Q$ that split in $F_0$. (We remark that including fewer Hecke operators makes our main theorem stronger.)

If $v$ is a prime of $F$ that divides a prime $p\not\in S$ that splits in the imaginary quadratic field $F_0$, then we get a lift $\mathfrak p|p$ of $p$ in $F_0$, and an isomorphism
\[
G(\mathbb Q_p)\cong \prod_{w|\mathfrak p} \GL_{2n}(F_w) = \GL_{2n}(F_v)\times \prod_{w|\mathfrak p,w\neq v} \GL_{2n}(F_w)\ .
\]
Here $w$ runs over primes of $F$ dividing $\mathfrak p$. For $i=1,\ldots,2n$, we let
\[
T_{i,v}\in \mathbb Z[G(\mathbb Q_p)//G(\mathbb Z_p)]
\]
be the Hecke operator given by the double coset of
\[
\GL_{2n}(\cO_{F_v})\mathrm{diag}(\underbrace{\varpi_v,\ldots,\varpi_v}_i,1,\ldots,1)\GL_{2n}(\cO_{F_v}) 
\times\prod_{w\mid \mathfrak{p}, w\neq v}\GL_{2n}(\cO_{F_w})
\]
inside $G(\Q_p)$. 

Now fix a prime $\ell$ and let $\mathfrak{m}\subset \mathbb{T}$ be a maximal ideal occurring in the support of $H^*(X_K,\F_\ell)$
(i.e.~a system of Hecke eigenvalues occurring in the cohomology of $X_K$), and fix an embedding $\mathbb T/\mathfrak m\to \overline{\F}_\ell$. Enlarge $S$ to include $\ell$. It follows from \cite[Theorem 4.3.1]{scholze} and~\cite[Theorem 2.3.3]{10Authors} (which relies on \cite{shin-basechange}) that there exists a continuous semisimple $2n$-dimensional Galois representation  
\[\overline\rho_{\mathfrak{m}}: \mathrm{Gal}(\overline F/F)\to \GL_{2n}(\overline{\F}_\ell)\]
unramified at all places not dividing a prime of $S$ and such that for every prime $v$ of $F$ as above, the characteristic polynomial of $\overline{\rho}_{\m}(\Frob_v)$ is the reduction of 
\[
X^{2n} - T_{1,v}X^{2n-1}+\dots+ (-1)^iq_v^{i(i-1)/2}T_{i,v}X^{2n-i}+\dots + q_v^{n(2n-1)}T_{2n,v}. 
\]
modulo $\m$, where $q_v$ is the cardinality of the residue field at $v$.

Our main theorem is the following.

\begin{theorem}\label{thm:main} Assume the following conditions.
\begin{enumerate}
\item[{\rm (i)}] $F^+\neq \mathbb Q$;
\item[{\rm (ii)}] $\overline\rho_{\mathfrak m}$ is of length at most $2$;
\item[{\rm (iii)}] there is a prime $p\not = \ell$ that splits completely in $F$ and such that for all primes $v|p$ of $F$, the representation $\overline\rho_{\mathfrak m}$ is unramified at $v$ with eigenvalues $\{\alpha_{1,v},\ldots,\alpha_{2n,v}\}$ of $\overline\rho_{\mathfrak m}(\Frob_v)\in \GL_{2n}(\overline{\F}_\ell)$ satisfying $\alpha_{i,v}\neq p\alpha_{j,v}$ for all $i\neq j$.
\end{enumerate}

Then
\begin{enumerate}
\item If $H^i(X_{K},\mathbb{F}_\ell)_{\mathfrak{m}}\neq 0$ then $i\geq d$.  
\item If $H^i_c(X_{K},\mathbb{F}_\ell)_{\mathfrak{m}}\neq 0$ then $i\leq d$. 
\end{enumerate}
\end{theorem}

Let us discuss the relevance of assumptions (i) -- (iii) in this paper. 

\begin{aremark} Assumption (i) simplifies some trace formula computations (essentially, it implies that certain boundary terms vanish). Roughly, the point is that the geometric side of the trace formula simplifies whenever the test function at two auxiliary places is of a simple form, and here we would like to use two infinite places. We would expect that one could remove this assumption with additional work.
\end{aremark}

\begin{aremark} We use assumption (ii) to ensure that only one boundary stratum (the one of interest, giving rise to $\GL_n/F$) can possibly contribute to the cohomology localized at $\mathfrak m$. This is required to ensure a tight relation between cohomology and compactly supported cohomology after localization at $\mathfrak m$. In order to ensure this tight relation, we also need the existence of Galois representations associated to torsion classes in the cohomology of locally symmetric spaces for $\GL_m/F$ for $m\leq n$, as proved in~\cite[Theorem 5.4.3]{scholze}. This is in contrast to the compact case~\cite{caraiani-scholze}, where we give an alternate proof to some related results that are implicit in~\cite{scholze}. 
(Assumption (ii) can probably be removed using an inductive argument and additional work. However, in light of the new method more recently introduced by Koshikawa~\cite{koshikawa-simplification}, which is discussed further in Remark~\ref{rem:koshikawa simplification}, we have decided not to pursue this here.)

\end{aremark}

\begin{aremark} Assumption (iii) is slightly weaker than a corresponding assumption in \cite[Theorem 1.1]{caraiani-scholze}, which would in addition ask that $\alpha_{i,v}\neq \alpha_{j,v}$. As observed by Koshikawa, this extra assumption is in fact superfluous: The critical \cite[Lemma 6.2.2]{caraiani-scholze} holds (with very minor modifications) in this more general setup, cf.~proof of Corollary~\ref{cor:existence Galois single degree} below, so in fact \cite[Theorem 1.1]{caraiani-scholze} holds in this more general setup. (The condition $\alpha_i\neq \alpha_j$ was also not necessary in Boyer's work, \cite{boyergeneric}.)
\end{aremark}

\begin{aremark} The theorem implies formally that the same conclusion holds for $H^i(X_K,\mathbb Z_\ell)_{\mathfrak m}$ and $H_c^i(X_K,\mathbb Z_\ell)_{\mathfrak m}$, and in addition that $H^d(X_K,\mathbb Z_\ell)_{\mathfrak m}$ is torsion-free. There is an excision long exact sequence relating $H^i_c(X_{K},\mathbb{Z}_\ell)_{\mathfrak{m}}$, $H^i(X_{K},\mathbb{Z}_\ell)_{\mathfrak{m}}$ (which agrees with the cohomology of the Borel--Serre compactification of $X_K$) and the cohomology of the boundary $H^i(\partial X_K,\mathbb{Z}_\ell)_{\mathfrak{m}}$. The cohomology of the boundary is related to the cohomology of $\GL_n/F$, and can contribute in many different degrees, both below and above the middle degree. By the long exact sequence, this cohomology has to be split between $H^i$ and $H_c^i$. The theorem asserts that this is done in the most transparent way:\footnote{It surprises the second author that such a clean picture can possibly be true: in some vague sense the Shimura variety is able to geometrically realize a truncation functor $\tau^{\leq d}$ on the cohomology of the boundary (while generally such truncation functors do not arise geometrically).} Below the middle degree, everything maps into $H_c^i$; above middle degree, everything comes from $H^i$; and one has an exact sequence
\[
0\to H^{d-1}(\partial X_K,\mathbb Z_\ell)_{\mathfrak m}\to H_c^d(X_K,\mathbb Z_\ell)_{\mathfrak m}\to H^d(X_K,\mathbb Z_\ell)_{\mathfrak m}\to H^d(\partial X_K,\mathbb Z_\ell)_{\mathfrak m}\to 0
\]
around the middle degree. In particular, the torsion-free group $H^d(X_K,\mathbb Z_\ell)_{\mathfrak m}$ surjects onto the cohomology $H^d(\partial X_K,\mathbb Z_\ell)_{\mathfrak m}$ of the boundary, yielding control over the various torsion classes there. Using the Hochschild--Serre spectral sequence, we can deduce the theorem and the same consequences also for non-trivial $\Z_{\ell}$-coefficient systems on $X_K$, such as $\cV_{\lambda}$, where $\lambda$ is a highest weight for $G$. See~\cite{10Authors} for applications.
\end{aremark}

\begin{aremark} If $\m$ is non-Eisenstein, i.e. if $\overline{\rho}_{\m}$ is absolutely irreducible, the theorem and the excision long exact sequence with $\F_{\ell}$-coefficients imply that 
\[
H^i_c(X_K, \F_{\ell})_{\m} \toisom H^i(X_K, \F_{\ell})_{\m}
\]
and that these are non-zero only in degree $i=d$. Indeed, if $\m$ is non-Eisenstein,
then $H^i(\partial X_K,\mathbb{F}_\ell)_{\mathfrak{m}} = 0$, which follows from the proof of~\cite[Theorem 2.4.2]{10Authors}. This also implies the same result with $\Z_{\ell}$-coefficients
and that $H^d_c(X_K, \Z_{\ell})_{\m} \toisom H^d(X_K, \Z_{\ell})_{\m}$ is torsion-free. This matches the folklore conjecture stated in (3) above, in the case of the locally symmetric space $X_K$.  
\end{aremark}

\begin{aremark}\label{rem:koshikawa simplification}
Recently, Koshikawa \cite{koshikawa-simplification} has found a different approach to establishing (a strengthening of) Theorem~\ref{thm:main}, and in fact he does not need conditions (i) and (ii). He still makes use of our geometric results, in particular of the semiperversity theorem, Theorem~\ref{thm:semiperversity}, but he does not use any computation of the cohomology of Igusa varieties, thus bypassing any use of trace formulas. Instead Koshikawa relies on the recent work \cite{fargues-scholze} on the geometrization of the local Langlands correspondence, in particular on the excursion operators.

The new method introduced by Koshikawa makes it more feasible to generalize Theorem~\ref{thm:main}, as long as certain prerequisites are in place, including the results on the geometry of the Hodge--Tate period morphism established in this paper. We expect most of these geometric results, e.g. those established in Sections~\ref{sec:compactifications of Igusa varieties} and~\ref{fibers of HT}, to extend to general Shimura varieties of PEL type, at the expense of working with more combinatorially involved objects as in \cite{lan-thesis}, and a more delicate structure of the toroidal compactification (which will involve torsors under abelian schemes in place of abelian schemes). This is the subject of the upcoming PhD thesis of Mafalda Santos~\cite{santos}.\footnote{One could ask about generalizing Theorem~\ref{thm:main} to the Hilbert--Siegel case, for example. In this case, the necessary geometric results can be established with very minor modifications. However, one also needs a way to control the cohomology of the relevant Igusa varieties or Rapoport--Zink spaces and this would require more work.}   
\end{aremark}

\subsubsection{Shimura varieties} The space $X_K$ can be identified, Hecke-equivariantly, with a union of connected components of $\mathscr S_{K(N)}(\mathbb C)$, for a Shimura variety $\mathscr S_{K(N)}$ of PEL type over $\mathbb Z$; see \S~\ref{sec:Shimura intro} for the precise relationship. As in \cite{caraiani-scholze}, we will prove the theorem by using the $p$-adic geometry of $\mathscr S_{K(N)}$ for a prime $p$ as guaranteed by Assumption (iii). More precisely, we will use the Shimura variety with infinite level at $p$ regarded as a perfectoid space, and analyze the fibers of the Hodge--Tate period map in terms of Igusa varieties, including compactifications.

Our geometric results hold for any imaginary CM field $F$ (not necessarily containing an imaginary quadratic field). Let $\Delta_F$ be the discriminant of $F$. From now on, set $K:=K(N)$
to be a principal congruence subgroup in the unitary similitude group corresponding to $G$, for some integer $N\geq 3$. We can define an algebraic variety $\mathscr S_K$ over $\mathbb Z[\tfrac 1{N\Delta_F}]$ whose $S$-valued points parametrize quadruples $(A,\iota,\lambda,\eta)$ where $A$ is an abelian scheme of dimension $[F:\mathbb Q]n$ over $S$ with an action $\iota: \mathcal O_F\to \mathrm{End}(A)$ such that $\Lie A$ is free of rank $n$ over $\mathcal O_F\otimes_{\mathbb Z} \mathcal O_S$, $\lambda: A\cong A^\vee$ is a principal polarization of $A$ whose associated Rosati involution is compatible with complex conjugation on $\mathcal O_F$ via $\iota$, and $\eta$ is an isomorphism $A[N]\cong L/NL$ compatible with the $\mathcal O_F$-action and polarization. Moreover, $\mathscr S_K$ admits a minimal compactification $\mathscr S_K^*$ and a toroidal compactification $\mathscr S_K^\tor$, the latter depending on the choice of a certain family of cone decompositions $\Sigma$, as usual.

Now fix a prime $p$ that is unramified in $F$ and prime to $N$, and an algebraically closed field $k$ of characteristic $p$. Moreover, fix a $p$-divisible group $\mathbb X$ over $k$, with an $\mathcal O_F$-action $\iota$ and a principal polarization $\lambda$, satisfying the same assumptions as the abelian variety above. Then the subset
\[
\{x\in \mathscr S_K\times k\mid A[p^\infty]\times k(\overline{x})\cong \mathbb X\times_k k(\overline{x})\}
\]
of all points $x$ such that $A[p^\infty]$ is isomorphic to $\mathbb X$, compatibly with the extra structures, defines a leaf $\mathscr C^{\mathbb X}\subset \mathscr S_K\times k$, a locally closed smooth subscheme. In \cite{lan-stroh}, Lan--Stroh prove that the leaf $\mathscr C^{\mathbb X}$ is \emph{well-positioned}, which implies that one can define partial minimal and toroidal compactifications $\mathscr C^{\mathbb X,\ast}$, $\mathscr C^{\mathbb X,\tor}$, which have many of the same properties as the ambient Shimura varieties.

Over the leaf $\mathscr C^{\mathbb X}$, we can look at the scheme parametrizing isomorphisms $A[p^\infty]\cong \mathbb X$ compatible with extra structures. This defines an $\underline{\Aut}(\mathbb X)$-torsor
\[
\mathfrak{Ig}^{\mathbb X}\to \mathscr C^{\mathbb X},
\]
where $\mathfrak{Ig}^{\mathbb X}$ is a perfect scheme. Note that in general when $\mathbb X$ is not isoclinic, $\underline{\Aut}(\mathbb X)$ is a highly non-reduced group scheme. Its group of connected components is the profinite group $\Gamma_{\mathbb X} := \underline{\Aut}(\mathbb X)(k)$. Then the map to the perfection
\[
\mathfrak{Ig}^{\mathbb X}\to \mathscr C^{\mathbb X}_\perf
\]
is a $\Gamma_{\mathbb X}$-torsor. One important property of $\mathfrak{Ig}^{\mathbb X}$ is that it depends on $\mathbb X$ only up to isogeny, i.e.~an isogeny $\phi: \mathbb X\to \mathbb X'$ (compatible with the extra structures) induces an isomorphism $\mathfrak{Ig}^{\mathbb X}\cong \mathfrak{Ig}^{\mathbb X'}$.

Our first result is that there are good partial toroidal compactifications of Igusa varieties.

\begin{aprop} The $\Gamma_{\mathbb X}$-torsor $\mathfrak{Ig}^{\mathbb X}\to \mathscr C^{\mathbb X}_\perf$ extends uniquely to a $\Gamma_{\mathbb X}$-torsor
\[
\mathfrak{Ig}^{\mathbb X,\tor}\to \mathscr C^{\mathbb X,\tor}_\perf.
\]
\end{aprop}

Moreover, one can describe the local structure of $\mathfrak{Ig}^{\mathbb X,\tor}$ at the toroidal boundary as in the case of Shimura varieties, cf.~e.g.~Theorem~\ref{thm:formal completions perfect Igusa}. As a consequence, we prove that an isogeny $\phi: \mathbb X\to \mathbb X'$ that induces isomorphisms on \'etale quotients induces an isomorphism $\mathfrak{Ig}^{\mathbb X,\tor}\cong \mathfrak{Ig}^{\mathbb X',\tor}$.\footnote{The condition on \'etale quotients is required to ensure that the choice of cone decomposition does not cause trouble.}

One can also define partial minimal compactifications of $\mathfrak{Ig}^{\mathbb X,*}$, for example via normalization. We can describe the boundary components of the minimal compactification explicitly in terms of \emph{Igusa cusp labels}, cf.~Theorem~\ref{thm:min Igusa strata}. Another important result is that partial minimal compactifications are affine:

\begin{aprop} For any $\mathbb X$, the leaf $\mathscr C^{\mathbb X,\ast}$ and the Igusa variety $\mathfrak{Ig}^{\mathbb X,\ast}$ are affine.
\end{aprop}

To prove this, we first prove it for a special leaf in a given isogeny class (one that is contained in a fundamental Ekedahl--Oort stratum). Such leaves exist in general by a result of Nie, \cite{nie-fundamental}, and in their case we appeal to a result of Boxer, \cite{boxer}, that partial minimal compactifications of Ekedahl--Oort strata are affine. (The latter result was also proved independently by Goldring--Koskivirta~\cite{goldring-koskivirta}.) Then we deduce the general case by using isogenies and the invariance of the partial toroidal compactifications and minimal compactifications under isogenies.

\begin{aremark} When $\mathbb X$ is completely slope divisible the Igusa variety and its compactifications admit non-perfect versions that are sometimes useful to keep the situation more geometric. In particular, one can naturally define versions with finite level structure. For this reason, some of the analysis is carried out only for completely slope divisible $\mathbb X$. As any $\mathbb X$ is isogenous to a completely slope divisible $\mathbb X$ this results in no essential loss of generality.
\end{aremark}

Our main geometric result is a description of the fibres of the Hodge--Tate period map on both minimal and toroidal compactifications. Let
\[
\mathcal{S}_{K(p^\infty N)}^\ast = \varprojlim_m \mathscr{S}_{K(p^m N),\mathbb Q_p}^\ast
\]
and similarly
\[
\mathcal{S}_{K(p^\infty N)}^\tor = \varprojlim_m \mathscr{S}_{K(p^m N),\mathbb Q_p}^\tor,
\]
as ``$p$-adic analytic spaces''.\footnote{Technically, as diamonds, \cite{scholze-diamonds}. It is known that both limits are representable by perfectoid spaces.} There is a Hodge--Tate period map
\[
\pi_\HT^\ast: \mathcal{S}_{K(p^\infty N)}^\ast\to \Fl
\]
where the flag variety $\Fl$ parametrizes totally isotropic $F$-stable subspaces of $V$; by pre-composition with the projection from the toroidal to the minimal compactification, we also get a map
\[
\pi_\HT^\tor: \mathcal{S}_{K(p^\infty N)}^\tor\to \Fl.
\]
Let $C$ be some complete algebraically closed nonarchimedean extension of $\Q_p$, and let $x\in \Fl(C)$ be a point. By \cite{scholze-weinstein}, this is equivalent to a $p$-divisible group $X$ over $\cO_C$ with extra structure as above, and an isomorphism $T_p(X)\cong L$. The special fibre $\mathbb X$ of $X$ gives rise to $\mathfrak{Ig}^{\mathbb X}$ with its partial minimal and toroidal compactification $\mathfrak{Ig}^{\mathbb X,\ast}$ and $\mathfrak{Ig}^{\mathbb X,\tor}$. As these are perfect schemes, they admit canonical lifts to $p$-adic formal schemes over $\cO_C$; let us simply denote these by a subscript $_{\cO_C}$, and then their generic fibres, which are perfectoid spaces over $C$, by a subscript $_C$.

\begin{theorem} There are canonical maps
\[\begin{aligned}
\mathfrak{Ig}^{\mathbb X,\tor}_C&\to (\pi_\HT^\tor)^{-1}(x),\\
\mathfrak{Ig}^{\mathbb X,\ast}_C&\to (\pi_\HT^\ast)^{-1}(x).
\end{aligned}
\]
They are open immersions of perfectoid spaces with the same rank-$1$-points; in fact, the target is the canonical compactification of the source, in the sense of~\cite[Proposition 18.6]{scholze-diamonds}.
\end{theorem}

In particular, the cohomology of the fibres of the Hodge--Tate period map agrees with the cohomology of Igusa varieties, even on compactifications. This extends the result of \cite{caraiani-scholze} to the non-compact case.

Finally, let us briefly explain how to deduce Theorem~\ref{thm:main}; we refer to Section~\ref{subsec:mainargument} for more details. The basic strategy is the same as in \cite{caraiani-scholze}: One analyzes the cohomology of the Shimura variety in terms of a Leray spectral sequence along $\pi_\HT$. The goal is to show that, after localization at $\mathfrak m$, the sheaf $(R\pi_{\HT\ast} \mathbb F_\ell)_\m$ is concentrated at the $0$-dimensional Newton stratum $\Fl(\Q_p)\subset \Fl$, and sits in the correct degrees. To prove this, one first proves a semiperversity result for one variant of the sheaf $R\pi_{\HT\ast} \mathbb F_\ell$, see Theorem~\ref{thm:semiperversity}. This is essentially the same argument as in \cite{caraiani-scholze}, using that $\pi_\HT^\ast$ is both partially proper and affine, but the boundary causes some small technical problems. Now one looks at the largest Newton stratum where the sheaf is nonzero. Its stalks there are computed in terms of the cohomology of Igusa varieties by the previous theorem; and semiperversity gives a bound on the range of degrees where cohomology can appear. This bound is in tension with another bound coming directly from Artin vanishing. Analyzing the boundary of Igusa varieties and using assumption (ii), one finds that the boundary does not contribute to the cohomology, and then the two estimates (from semiperversity and Artin vanishing) show that the cohomology would have to be concentrated in the middle degree, and thus the $\mathbb Z_\ell$-cohomology is torsion-free. Now a computation of the $\overline{\mathbb Q}_\ell$-cohomology in terms of automorphic forms shows that it actually has to be zero.

\subsubsection{Organization} Let us now summarize the contents of the different sections, highlighting also some further results not mentioned so far.

In Section~\ref{preparation}, we set up the basic formalism of our Shimura varieties and recall some foundational results, especially about the construction of Igusa varieties. At the end of this section, we also explain the main argument, referring to the key results of the later sections.

In Section~\ref{sec:compactifications of Igusa varieties}, we construct the compactifications of Igusa varieties and prove basic results about their geometry, as stated above.

In Section~\ref{fibers of HT}, we describe the fibres of the compactified Hodge--Tate period maps $\pi_{\HT}^{\tor}$ and $\pi_{\HT}^{\ast}$ in terms of compactified Igusa varieties. We use toroidal compactifications first to write down an explicit map in terms of boundary strata (even integrally). Once we have a map, checking that it is an isomorphism can be done on geometric points by general properties of perfectoid spaces (or even diamonds). The case of minimal compactifications follows rather formally by using that both sides are affine. Finally, we also use these results to deduce a semiperversity result for $R\pi_{\HT\ast}^\circ \mathbb F_\ell$.

In Section~\ref{igusa}, we compute the cohomology of the Igusa varieties $\mathrm{Ig}^b$, as a virtual $\overline{\mathbb Q}_\ell$-representation. This builds on the previous work of Shin, \cite{shin-igusa}, \cite{shin-stable}, \cite{shin-galois}; a similar analysis also appeared in \cite{caraiani-scholze}. The present situation requires us to understand more precisely the boundary terms. It is here that we have to assume that $F^+\neq \Q$, which ensures that $G$ admits no \emph{cuspidal subgroups}, cf. Definition~\ref{cuspidal subgroups}. Essentially, this ensures that we have a special test function at at least two places (the infinite places) which simplifies the geometric of the side trace formula, and allows us to compare it to Shin's trace formula for Igusa varieties. If $F^+=\Q$, then the work of Morel, \cite{morel-trace}, shows that the situation is more complicated even in the case of Shimura varieties, and one can expect a clean answer only for intersection cohomology.

In Section~\ref{pink}, we analyze the cohomology of the boundary of the Igusa varieties in terms of the Igusa varieties for smaller unitary groups and the cohomology of general linear groups. What we prove is essentially a version for Igusa varieties of Pink's formula~\cite{pink}. It is actually somewhat tricky to show that these computations are Hecke-equivariant, and we use some new ideas involving adic spaces to get a transparent argument. As an application, we also construct Galois representations associated to mod $\ell$ systems of Hecke eigenvalues that occur in the cohomology of Igusa varieties (in the setting of Theorem~\ref{thm:main}).  

\subsubsection{Acknowledgments}
We thank Teruhisa Koshikawa for pointing out that the ``decomposed generic'' condition of~\cite{caraiani-scholze} can be slightly weakened. We also thank K\k{e}stutis \v{C}esnavi\v{c}ius, Sophie Morel, and James Newton for helpful discussions. We thank Mingjia Zhang and the anonymous referees for their comments on an earlier version of this paper. 

The first author was supported by a Royal Society University Research Fellowship, a Leverhulme Prize, and by ERC Starting Grant 804176. The second author was supported by a DFG Leibniz Grant and the Hausdorff Center for Mathematics (GZ 2047/1, Projekt-ID 390685813).

\newpage

\section{Quasisplit unitary Shimura varieties}\label{preparation}

In this section, we define the relevant Shimura varieties associated to quasi-split unitary similitude groups and prove some basic related results. Our notation in this section will be slightly different from the notation we used in the introduction. 

\subsection{Shimura varieties}\label{sec:Shimura intro}

As in the introduction, fix an imaginary CM field $F$ with totally real subfield $F^+\subset F$, and an integer $n\geq 1$. Consider the $F$-vector space $V=F^{2n}$ with the skew-hermitian form
\[
\langle (x_1,\ldots,x_{2n}),(y_1,\ldots,y_{2n})\rangle = \sum_{i=1}^n(x_i\overline{y}_{2n+1-i} - x_{2n+1-i}\overline{y}_i)
\]
and consider the induced alternating form
\[
(\cdot,\cdot): V\times V\to \Q: (x,y) = \tr_{F/\Q} \langle x,y\rangle.
\]
Then $V$ admits self-dual $\cO_F$-lattices $L\subset V$ with respect to the pairing $(\cdot,\cdot)$: For example, if $I\subset F$ is the inverse different, then $L=\cO_F^n \oplus \overline{I}^n$ is such a self-dual lattice. Fix any self-dual $\cO_F$-lattice $L\subset V$, and define the group $G$ over $\mathbb Z$ by
\[
G(R) = \{(g,c)\in \GL_{\cO_F}(L)(R)\times \mathbb G_m(R)\mid (gv,gw) = c(v,w)\ \forall v,w\in L\}\ ,
\]
which is a unitary similitude group. 
We let
\[
X=\prod_{\tau: F^+\hookrightarrow \R} X_{\tau,+}\sqcup \prod_{\tau: F^+\hookrightarrow \R} X_{\tau,-}
\]
be the symmetric space for $G(\R)$, where $X_{\tau,+}$ (resp.~$X_{\tau,-}$) is the space of positive (resp.~negative) definite $n$-dimensional subspaces of $V\otimes_{F^+} \R\cong \C^{2n}$. For any neat compact open subgroup $K\subset G(\A_f)$, we consider the double quotient
\[
X_K = G(\Q)\backslash (X\times G(\A_f)/K)\ .
\]

To compare with Theorem~\ref{thm:main}, we will also need to consider the corresponding unitary group $G_0$ over $\Z$:
\[
G_0(R) = \{g\in \GL_{\cO_F}(L)(R) \mid (gv,gw) = (v,w)\ \forall v,w\in L\}\ .
\] 
Then $G_0$ is naturally a subgroup of $G$ (and is the group denoted by $G$ in the introduction). We also consider $X^0= \prod_{\tau: F^+\hookrightarrow \R} X_{\tau,+}$,
which is the symmetric space for $G_0(\R)$ and, for a neat compact open subgroup $K_0\subset G_0(\A_f)$, the corresponding double quotient $X^0_{K_0}$ (this is the space
denoted $X_K$ in the introduction). 

Moreover, for a prime $v$ of $F$ above a prime away from $S$ and that splits in $F_0$, let $T^0_{i,v}$ be the double coset operator
\[
\GL_{2n}(\cO_{F_v})\mathrm{diag}(\underbrace{\varpi_v,\ldots,\varpi_v}_i,1,\ldots,1)\GL_{2n}(\cO_{F_v}) 
\times\prod_{w\mid \mathfrak{p}, w\neq v}\GL_{2n}(\cO_{F_w})
\]
for $G_0$, and $T_{i,v}$ be the double coset operator
\[
\GL_{2n}(\cO_{F_v})\mathrm{diag}(\underbrace{\varpi_v,\ldots,\varpi_v}_i,1,\ldots,1)\GL_{2n}(\cO_{F_v}) 
\times\prod_{w\mid \mathfrak{p}, w\neq v}\GL_{2n}(\cO_{F_w})\times \Z_{p}^\times
\]
for $G$. We let $\mathbb{T}^{0,S}$ denote the abstract unramified Hecke algebra for $G_0$ defined by 
\[
\mathbb{T}^{0,S} = \bigotimes_{p\not \in S, p\in \mathrm{Spl}_{F_0/\Q}} \Z[G_0(\Q_p)//G_0(\Z_p)]. 
\]
This is generated by the $T^0_{i,v}$, as $v$ runs over primes of $F$ as above. We let $\mathbb{T}^S$ denote the subalgebra of
\[
\bigotimes_{p\not \in S, p\in \mathrm{Spl}_{F_0/\Q}} \Z[G(\Q_p)//G(\Z_p)] 
\]
generated by the $T_{i,v}$, as $v$ runs over primes of $F$ as above. 
The inclusion $G_0\hookrightarrow G$ induces a restriction map $\mathbb{T}^S\to \mathbb{T}^{0,S}$ on the level of Hecke algebras  
that identifies $T_{i,v}$ with $T^0_{i,v}$. 

\begin{lem}\label{lem:similitude factor} Assume that $K_0 = K\cap G_0(\A_f)$. The inclusion $G_0\hookrightarrow G$
induces a natural map $X^0_{K_0}\to X_K$, which is an open and closed immersion. The induced map on cohomology and compactly supported cohomology is Hecke-equivariant 
for the restriction map $\mathbb{T}^S\to \mathbb{T}^{0,S}$. 
\end{lem}

\begin{proof} 
We have a short exact sequence of algebraic groups over $\Z$:
\[
1\to G_0\to G\stackrel{c}{\to} \mathbb{G}_m\to 1. 
\]
The inclusion $G_0\hookrightarrow G$ identifies $X^0$ with a connected component of $X$. The locally symmetric space $X_K$ is a disjoint union of quotients of $X$ by congruence subgroups of $G(\Q)$; the analogous statement also holds true for $X^0_{K_0}$. To prove the lemma, it is enough to check that the natural map $X^0_{K_0}\to X_K$ is injective. Assume that we have pairs $(h_i,g_i)\in X^0\times G_0(\A_f)$ for $i=1,2$ and elements $\gamma\in G(\Q)$ and $k \in K$ such that $(h_1, g_1) = \gamma (h_2,g_2)k$. From $g_1 = \gamma g_2 k$, we deduce that 
\[
c(\gamma) = c(k)^{-1}\in \Q^\times \cap \left(\prod_{p}\Z_p^\times\right) = \{\pm 1\}.
\] 
On the other hand, $h_1 = \gamma h_2$ implies that $c(\gamma)>0$. We deduce that $\gamma\in G_0(\Q)$, $k\in K_0$. The last statement is clear. 
\end{proof}

It will be convenient to know that $G$ satisfies the Hasse principle:

\begin{prop}\label{hasse principle} The group $G$ satisfies the Hasse principle, i.e.~the map
\[
H^1(\Q,G)\to \prod_v H^1(\Q_v,G)
\]
is injective, where $v$ runs over all places of $F$. Equivalently, if $(V',(\cdot,\cdot)')$ is any other $2n$-dimensional $F$-vector space with an alternating form $(\cdot,\cdot)'$ such that $(xv,w)' = (v,\overline{x}w)'$ for all $v,w\in V'$ and $x\in F$, such that there are isomorphisms $(V_v,(\cdot,\cdot)_v)\cong (V'_v,(\cdot,\cdot)'_v)$, identifying the forms up to a scalar, after base change to any local field $\Q_v$ of $\Q$, then there is an isomorphism $(V,(\cdot,\cdot))\cong (V',(\cdot,\cdot)')$, identifying the forms up to a scalar.
\end{prop}

\begin{proof} (cf.~\cite[Section 7]{kottwitzpoints}) The derived group $G_{\mathrm{der}}$ of $G$ is simply connected, so $H^1(\Q,G_{\mathrm{der}}) = 0$. It follows that it suffices to see that $D=G/G_{\mathrm{der}}$ satisfies the Hasse principle. One can identify $D$ with the subtorus of $\mathrm{Res}_{F/\Q} \mathbb G_m\times \mathbb G_m$ of all pairs $(z,t)$ such that $\mathrm{Nm}_{F/F^+}(z)=t^{2n}$. Via $(z,t)\mapsto (z/t^n,t)$, this is isomorphic to $\mathrm{Res}_{F^+/\Q} T\times\mathbb G_m$, where $T=\ker(\mathrm{Nm}: \mathrm{Res}_{F/F^+} \mathbb G_m\to \mathbb G_m)$. Both of these tori satisfy the Hasse principle.
\end{proof}

We can now interpret $X_K$ as the $\C$-points of a moduli scheme of abelian varieties with certain extra structures. As we will have much occasion to consider similar kinds of extra structure, let us fix some terminology for this paper.

\begin{defn}\label{def:extra structure}\leavevmode
\begin{enumerate}
\item Let $S$ be a scheme over $\Z[\tfrac 1{\Delta_F}]$. An abelian variety with $G$-structure over $S$ is a triple $(A,\iota,\lambda)$ where $A$ is an abelian scheme of dimension $[F:\Q]n$ over $S$, $\iota: \cO_F\to \mathrm{End}(A)$ is an $\cO_F$-action such that $\Lie A$ is free of rank $n$ over $\cO_F\otimes_{\Z} \mathcal O_S$, and $\lambda: A\cong A^\vee$ is a principal polarization on $A$ whose Rosati involution is compatible with complex conjugation on $\mathcal O_F$ via $\iota$.
\item Let $p$ be a prime that is unramified in $F$ and let $S$ be a scheme on which $p$ is locally nilpotent. A $p$-divisible group with $G$-structure over $S$ is a triple $(X,\iota,\lambda)$ where $X$ is a $p$-divisible group of height $2[F:\Q]n$ and dimension $[F:\Q]n$, $\iota: \cO_F\to \mathrm{End}(X)$ is an $\cO_F$-action such that $\Lie X$ is free of rank $n$ over $\cO_F\otimes_{\Z} \mathcal O_S$, and $\lambda: X\cong X^\vee$ is a principal polarization on $X$ whose Rosati involution is compatible with complex conjugation on $\mathcal O_F$ via $\iota$.
\end{enumerate}
\end{defn} 

Let $N\geq 3$ and let
\[
K=K(N)=\{g\in G(\widehat{\Z})\mid g\equiv 1\mod N\}
\]
be a principal congruence subgroup, automatically neat. In the following definition, we do not invert $N$; we do not claim any nice geometric properties of the resulting scheme at primes dividing $N$. (We need some integral models at ``bad'' places in our discussion of toroidal compactifications.)

\begin{defn}\label{def:shimura variety} Let $\mathscr S_K^{\mathrm{pre}}$ over $\Spec \Z[\tfrac 1{\Delta_F}]$ parametrize over a test scheme $S$ an abelian variety with $G$-structure $\underline{A} = (A,\iota,\lambda)$ together with an $\cO_F$-linear map $L/N\to A[N]$ and a primitive $N$-th root of unity $\zeta_N\in \mu_N(\mathcal O_S)$ such that the diagram
\[\xymatrix{
L/N\times L/N\ar[r]^-{(\cdot,\cdot)}\ar[d] & \Z/N\ar[d]^{\zeta_N}\\
A[N]\times A[N]\ar[r] & \mu_N
}\]
commutes, where the lower map is the Weil pairing on $A[N]$ induced by $\lambda$, and such that the map $L/N\to A[N]$ extends to similar $\cO_F$-linear maps $L/\Delta_F^m N\to A[\Delta_F^m N]$ compatible with the pairing for all $m\geq 1$.
\end{defn}

It is standard that $\mathscr S_K^{\mathrm{pre}}$ is a Deligne--Mumford stack and that $\mathscr S_K^{\mathrm{pre}}\times \Z[\tfrac 1{\Delta_F N}]$ is representable, as it is relatively representable over the Siegel modular variety of principal level $N\geq 3$. In fact, even if $p$ divides $N$, $\mathscr S_K^{\mathrm{pre}}\times \mathbb Z_{(p)}$ is representable as soon as the part of $N$ prime to $p$ is at least $3$; we will always be in such a situation. We will need to normalize $\mathscr S_K^{\mathrm{pre}}$ (because at primes dividing $N$, we have formulated a very bad moduli problem):

\begin{defn} Let $\mathscr S_K$ be the normalization of $\mathscr S_K^{\mathrm{pre}}$ in $\mathscr S_K^{\mathrm{pre}}\times \Z[\tfrac 1{\Delta_F N}]$.
\end{defn}

The reason for the final condition on lifting the level structure will become apparent in the following proof:

\begin{prop}\label{prop:complex unif} There is a natural isomorphism of manifolds
\[
X_K\cong \mathscr S_K(\C).
\]
\end{prop}

\begin{proof} If $\underline{A}=(A,\iota,\lambda)$ is an abelian variety with $G$-structure over $\C$, then the first homology $L'=H_1(A,\Z)$ is a finite projective $\cO_F$-module of rank $2n$ equipped with a perfect alternating form $(\cdot,\cdot)': L'\times L'\to \Z$ (up to sign). We want to see that when $\underline{A}$ comes from $\mathscr S_K$, there is an $F$-linear isomorphism $L_\Q\cong L'_\Q$ compatible with the forms up to scalar. By the Hasse principle (Proposition~\ref{hasse principle}), it is enough to show that such isomorphisms exist locally. At the archimedean places, both forms are of the same signature by the condition on the dimension of the Lie algebra. At the primes that are unramified in $F$, both lattices are self-dual, which determines the isomorphism class of the form. Finally, at the ramified primes we use the final condition in Definition~\ref{def:shimura variety}.

Now the choice of an isomorphism $L_\Q\cong L'_\Q$ compatible with the form, up to scalar, gives a $G(\Q)$-torsor over $S_K(\C)$. We want to identify this with $X\times G(\A_f)/K$. The element of $X$ comes from Hodge theory -- the (positive or negative) definite $n$-dimensional $F\otimes_{\Q} \C$-subspace is the Hodge filtration -- and the element of $G(\A_f)/K$ is exactly the given level structure. Note again that at ramified primes, we need to use the assumption that the isomorphism lifts to ensure that the isomorphism modulo $N$ comes from an element of $G(\A_f)$.
\end{proof}

\subsection{$p$-divisible groups}

Next, we give some results about $p$-divisible groups with $G$-structure.

Let $(\mathbb X,\iota,\lambda)$ be a $p$-divisible group with $G$-structure over an algebraically closed field $k$ of characteristic $p$. We get a filtration
\[
\mathbb X^{\mu}\subset \mathbb X^\circ\subset \mathbb X
\]
into the multiplicative and connected part. This filtration is $\mathcal O_F$-stable and symplectic with respect to $\lambda$. We let
\[
\mathbb X^{(0,1)} = \mathbb X^\circ/\mathbb X^{\mu}, \mathbb X^{\mathrm{\acute{e}t}} = \mathbb X/\mathbb X^\circ
\]
be the graded pieces for the filtration. As we work over a perfect field, this filtration splits uniquely, so
\[
\mathbb X\cong \mathbb X^{\mu}\oplus \mathbb X^{(0,1)}\oplus \mathbb X^{\mathrm{\acute{e}t}}
\]
under which $\lambda$ decomposes similarly into a direct sum. In particular, up to isomorphism, $\mathbb X$ is determined by $\mathbb X^{(0,1)}$ with its $\mathcal O_F$-action and principal polarization and the finite projective $\cO_F\otimes_{\Z} \Z_p$-module $T_p(\mathbb X^{\mathrm{\acute{e}t}})$. Slightly more generally, we have the following observation.\footnote{The motivation for the result are the Igusa cusp labels introduced later; this also explains the notation.}

\begin{prop}\label{classification pdivisible} Assume that
\[
\mathrm{Z}_{-2}\subset \mathrm{Z}_{-1}\subset \mathbb X
\]
is an $\mathcal O_F$-stable filtration by sub-$p$-divisible groups such that $\mathrm{Z}_{-2}$ is multiplicative, $\mathbb X/\mathrm{Z}_{-1}$ is \'etale, and the polarization identifies $\mathrm{Z}_{-2}$ with $\mathbb X/\mathrm{Z}_{-1}$. Then there is an $\mathcal O_F$-linear splitting
\[
\mathbb X\cong \mathrm{Z}_{-2}\oplus \mathrm{Z}_{-1}/\mathrm{Z}_{-2}\oplus \mathbb X/\mathrm{Z}_{-1}
\]
under which $\lambda$ decomposes similarly into a direct sum.
\end{prop}

\begin{proof} We can decompose into multiplicative, biconnected, and \'etale parts, which gives us only two splitting problems, on the \'etale and on the multiplicative parts. We only have to arrange that they are dual, so we can make an arbitrary choice at one side and then arrange the other one to be dual.
\end{proof}

\subsubsection{Construction of Igusa covers}

Let $k$ be an algebraically closed field of characteristic $p$ and let $\mathbb{X}/k$
be an isoclinic $p$-divisible group equipped with extra structures of EL type or of PEL type. 
For every $m\in \Z_{\geq 1}$, let $\Gamma_m$ denote the group of automorphisms of $\mathbb{X}[p^m]$
which commute with the extra structures and which lift to automorphisms of $\mathbb{X}[p^{m'}]$ 
for every $m'\geq m$. We let $\Gamma_{m,k}$ denote the finite \'etale group
scheme over $k$ corresponding to $\Gamma_m$. The following
is a consequence of~\cite[Corollary 4.1.10]{caraiani-scholze}. 

\begin{prop}\label{prop:finite etale gp scheme} The finite \'etale group scheme $\Gamma_{m,k}$ represents
the functor 
\[
\mathscr{F}_{\Gamma_m}: \Spec k-\mathrm{Schemes}\to \mathrm{Sets}
\]
for which $\mathscr{F}_{\Gamma_m}(\mathscr{T})$ is the set of $\mathscr{T}$-automorphisms
of $\mathbb{X}_{\mathscr{T}}[p^m]$ which commute with the extra structures and which 
lift (fppf locally on $\mathscr{T}$) to $\mathscr{T}$-automorphisms of $\mathbb{X}_{\mathscr{T}}[p^{m'}]$ for all $m'\geq m$. In fact, for any $m$ there is some $m'\geq m$ so that it is enough to ask for a lifting to $\mathbb X_{\mathscr{T}}[p^{m'}]$ (fppf locally on $\mathscr T$).
\end{prop}

\begin{proof} We consider the functor $\mathscr{F}_{H_m}$ sending a $\Spec k$-scheme $\mathscr{T}$ 
to the set of $\mathscr{T}$-endomorphisms of $\mathrm{X}_{\mathscr{T}}[p^m]$ which lift (fppf locally on $\mathscr{T}$) 
to $\mathscr{T}$-endomorphisms of $\mathrm{X}_{\mathscr{T}}[p^{m'}]$ for all $m'\geq m$ (we do not
require these to be invertible or to commute with the extra structures). Then $\mathscr{F}_{\Gamma_m}$ is a subfunctor of $\mathscr{F}_{H_m}$.
It is enough to see that $\mathscr{F}_{H_m}$ is representable by a finite \'etale scheme $H_{m,k}$. 
This is proved in the second part of~\cite[Corollary 4.1.10]{caraiani-scholze}, which applies because $\mathbb{X}$ is isoclinic.

To see that one $m'$ is enough, use that everything is finitely presented.
\end{proof}

Let $X/k$ be a seminormal scheme. Assume that $\mathscr{G}/X$ is a $p$-divisible group equipped with extra structures of EL type or of PEL type as in \cite[Section 4.2]{caraiani-scholze}, such that for any point $x\in X$ and for any geometric point $\bar{x}$ above $x$ there exist isomorphisms 
\[
\mathscr{G}\times_{X} k(\bar{x}) \toisom \mathbb{X}\times_{k}k(\bar{x})
\]
compatible with the extra structures. The following result is proved in~\cite{mantovan} in the
special case when $X$ is an Oort central leaf in a PEL-type Shimura variety, but the result holds more generally and we will need to appeal to the general result repeatedly later on. 

\begin{thm}\label{thm:construction of Igusa covers} Let $m\in \Z_{\geq 1}$. Consider the functor from $X$-schemes
to sets which sends a scheme $\mathscr{T}/X$ to the set of isomorphisms
\[
\rho_m: \mathscr{G}[p^m]\times_{X}\mathscr{T} \toisom \mathbb{X}[p^m]\times_{k} \mathscr{T}
\]
compatible with the extra structures and which lift (fppf locally on $\mathscr{T}$)
to isomorphisms $\rho_{m'}$ for all $m'\geq m$ (or just for one sufficiently large $m'$ as in the previous proposition). This functor is representable by a $\Gamma_m$-torsor
\[
J_m(\mathscr{G}/X)\to X. 
\]
\end{thm}

\begin{proof} We may assume that $X=\Spec R$ is affine. We first make a reduction to the case that $R$ is perfect, so assume that the result holds true when $R$ is perfect. Our goal is to show that after some faithfully flat cover, $\mathscr G$ is isomorphic to $\mathbb X$; this clearly implies the theorem by faithfully flat descent. Assuming the result for $X_\perf$, we get for any $m\geq 1$ a $\Gamma_m$-torsor $J_m(\mathscr G\times_X X_\perf/X_\perf)\to X_\perf$. As $X\to X_\perf$ is a universal homeomorphism, this $\Gamma_m$-torsor descends uniquely to a $\Gamma_m$-torsor $J_m'\to X$. Let $J' = \varprojlim_m J_m'$, which is faithfully flat over $X$. We claim that the isomorphism $\mathscr G\cong \mathbb X$ that exists tautologically over $J'\times_X X_\perf = \varprojlim_m J_m(\mathscr G\times_X X_\perf/X_\perf)$ is in fact defined over $J'$. Note that this isomorphism is given by a series of isomorphisms between finite locally free group schemes $\mathscr G[p^m]\cong \mathbb X[p^m]$, so this amounts to a countable system of elements of $R_\perf$ being elements of $R$. As $R$ is seminormal, an element $f\in R_\perf$ lies in $R$ if and only if for all $x\in X=\Spec R$, the element $f(x)\in k(x)_\perf$ lies in $k(x)$, cf.~\cite[Theorem 2.6]{swan-seminormal}. Thus, we can reduce to the case that $R$ is a field. In that case $R\to R_\perf$ is faithfully flat, so the theorem follows by faithfully flat descent for $R$; and then we also see the claim, as necessarily $J_m(\mathscr G/X) = J_m'$.

We therefore assume that $X$ is perfect. We further want to reduce to the case that $X$ is strictly henselian local. For this, consider for any $m$ the scheme $J_m''(\mathscr G/X)\to X$ parametrizing isomorphisms $\mathscr G[p^m]\cong \mathbb X[p^m]$ compatible with extra structures; note that this functor is representable by a finitely presented affine scheme 
over $X$. (For this, we think of group scheme homomorphisms $\mathscr{G}[p^m]\times_{X}\mathscr{T}\to \mathbb{X}[p^m]\times_{k} \mathscr{T}$ 
as maps of Hopf algebras over $\mathscr{T}$.) Let $J_m''(\mathscr G/X) = \Spec A_m$. For $m'\geq m$, we get maps $A_m\to A_{m'}$; let $\overline{A}_m$ denote the image of $A_m$ in $A_{m'}$ where $m'$ is chosen large enough as in the previous proposition. Then the formation of $\overline{A}_m$ commutes with flat base change, and $\overline{A}_m$ is still a finitely presented $R$-algebra, as it only depends on $\mathscr G[p^{m'}]$ which is finitely presented. It is enough to show that $\overline{A}_m$ is faithfully finite flat over $R$: Indeed, if this holds for all $m$, then $\varinjlim_m A_m=\varinjlim_m \overline{A}_m$ is faithfully flat over $R$, and we can by flat base change reduce to the case that $\mathscr{G}\cong \mathbb X$, in which case the claim is clear. But checking that $\overline{A}_m$ is faithfully finite flat can be done over the strictly henselian local rings of $R$, so from now on we assume $X=\Spec R$, with $R$ a strictly henselian local ring
that is also perfect. It is then enough to find an isomorphism $\mathscr{G}\toisom \mathbb{X}\times_{k} R$ compatible with extra structures.

From~\cite[Lemma 4.3.15]{caraiani-scholze}, we see that we can construct an isogeny $\mathscr{G}\to \mathbb{X}\times_{k} R$, a priori not compatible with the extra structures. This induces a morphism from $\Spec R$ to the reduced special fiber $\overline{\cM}^{0,d}_{\mathbb{X}}$ of a truncated Rapoport--Zink space for $\mathbb{X}$. Lemma~\ref{lem:finite subset of RZ} below shows that the subset of points of $\overline{\cM}^{0,d}_{\mathbb{X}}$ over which the $p$-divisible group parametrised by the Rapoport--Zink space is isomorphic to $\mathbb{X}$ is finite. As $\Spec R$ is connected, the map from $\Spec R$ to the Rapoport--Zink space factors through one such (closed) point. Modifying the isogeny, we can then find an isomorphism $\rho: \mathscr{G}\toisom \mathbb{X}\times_{k} R$, a priori only compatible with the extra structures over the closed point. But \cite[Lemma 4.3.15]{caraiani-scholze} ensures that it must automatically be compatible with the extra structure over all of $\Spec R$.
\end{proof}

We used the following general lemma: Let $k$ be an algebraically closed field of characteristic $p$ and let $\mathbb{Y}/k$ be a $p$-divisible group. Let $\cM_{\mathbb{Y}}$ be the Rapoport--Zink space for $\mathbb{Y}$ parametrizing quasi-isogenies from $\mathbb Y$ to varying $p$-divisible groups. Consider the truncated Rapoport--Zink space $\cM_{\mathbb{Y}}^{0,d}$, which parametrises isogenies with kernel contained in the $p^d$-torsion subgroup. Let $\overline{\cM}^{0,d}_{\mathbb{Y}}$ be its reduced special fiber. The following lemma is a slight generalisation of~\cite[Lemma 3.4]{mantovan-thesis},
which applies to the special case of an isoclinic, completely slope divisible $p$-divisible group. 

\begin{lem}\label{lem:finite subset of RZ} Let $\mathscr{H}$ be the universal
$p$-divisible group over $\overline{\cM}^{0,d}_{\mathbb{Y}}$. The subset 
\[
Z :=\{x\in \overline{\cM}^{0,d}_{\mathbb{Y}}\mid 
\mathscr{H}\times k(\bar{x})\cong \mathbb{Y}\times_{k} k(\bar{x}) \} \subseteq \overline{\cM}^{0,d}_{\mathbb{Y}}
\]
consists of finitely many reduced points.
\end{lem}

\begin{proof} We know that the subset $Z$ is constructible by~\cite[Corollary 2.5]{oort}. To show that it is finite, it is enough to see that all points are defined over $k$. Consider the scheme $\tilde{Z}\to \overline{\cM}^{0,d}_{\mathbb{Y}}$ parametrizing isomorphisms $\mathscr{H}\cong \mathbb Y$; this surjects onto $Z$. Then for any algebraically closed field $k'$ over $k$, $\tilde{Z}(k')$ is the set of self-isogenies $\mathbb Y_{k'}\to \mathbb Y_{k'}$ of degree $\leq d$. But $\mathrm{End}(\mathbb Y)=\mathrm{End}(\mathbb Y_{k'})$, so we see that $\tilde{Z}(k') = \tilde{Z}(k)$, and so it follows that all points of $Z$ are $k$-valued, and hence $Z$ is finite.
\end{proof}

Note that in the previous lemma, we did not ask that $\mathbb Y$ be isoclinic. In fact, we will now use it to prove a variant of Theorem~\ref{thm:construction of Igusa covers} for non-isoclinic groups.

Let $\mathbb X/k$ be a $p$-divisible group with extra structure of EL or PEL type. Let $\Gamma = \varprojlim_m \Gamma_m = \mathrm{Aut}(\mathbb X)$ be the profinite group of automorphisms of $\mathbb X/k$ compatible with the extra structure. Let $X/k$ be a perfect scheme, and assume that $\mathscr G$ is a $p$-divisible group over $X$ with the same kind of extra structure such that for all geometric points $\bar x$ of $X$, there is an isomorphism $\mathscr G\times k(\bar x)\cong \mathbb X\times_k k(\bar x)$ compatible with extra structures.

\begin{prop}\label{perfect Igusa} The functor on perfect $X$-schemes $\mathscr T$ parametrizing isomorphisms $\mathscr G\times_X \mathscr T\cong \mathbb X\times_k \mathscr T$ is representable by a $\Gamma$-torsor $J(\mathscr G/X)\to X$.
\end{prop}

\begin{proof} First one verifies the assertion when $\mathscr G$ is constant, for which one uses that if $R$ is a strictly henselian perfect ring then all automorphisms of $\mathbb X_R$ are constant, cf.~\cite[Lemma 4.3.15]{caraiani-scholze}. In general, the functor is evidently representable by a scheme affine over $X$; it is enough to show that it is faithfully flat, as one can then by faithfully flat descent reduce to the case that $\mathscr G$ is constant. Now faithful flatness can be checked on strictly henselian local rings, so we can assume that $X=\Spec R$ where $R$ is a strictly henselian perfect ring. In this case, the argument from Theorem~\ref{thm:construction of Igusa covers} shows that $\mathscr G$ is already constant, so the result follows.
\end{proof}

A variant is the following. Let $\underline{\mathrm{Aut}}(\mathbb X)$ denote the group scheme of automorphisms of $\mathbb X$ compatible with extra structures. This is in general highly non-reduced, with the non-reduced structure related to the failure of $\mathbb{X}$ to be isoclinic (see~\cite[Corollary 4.1.11]{caraiani-scholze}). Let $X/k$ be a regular scheme, and assume as above that $\mathscr G$ is a $p$-divisible group over $X$ with the same kind of extra structure such that for all geometric points $\bar x$ of $X$, there is an isomorphism $\mathscr G\times k(\bar x)\cong \mathbb X\times_k k(\bar x)$ compatible with extra structures.

\begin{prop}\label{perfect Igusa two} The functor on all $X$-schemes $\mathscr T$ parametrizing isomorphisms $\mathscr G\times_X \mathscr T\cong \mathbb X\times_k \mathscr T$ is representable by an $\underline{\mathrm{Aut}}(\mathbb X)$-torsor over $X$.
\end{prop}

\begin{proof} It is clearly a quasi-torsor, so it is enough to find a section over a faithfully flat cover. Proposition~\ref{perfect Igusa} gives a section over the faithfully flat $J(\mathscr G\times_X X_\perf/X_\perf)\to X_\perf$, and $X_\perf\to X$ is faithfully flat as $X$ is regular.
\end{proof}

\subsection{Igusa varieties}

Using the results of the previous subsection, we will now define the Igusa varieties related to the Shimura variety $\mathscr S_K$. We assume that $K=K(N)$ is a principal level with $N\geq 3$, and we fix a prime $p$ that is unramified in $F$ and prime to $N$. Let $k$ be an algebraically closed field of characteristic $p$, and let $\mathbb X/k$ be a $p$-divisible group with $G$-structure.

\begin{defn}
The central leaf\footnote{In~\cite{oort}, Oort calls these objects central leaves 
to distinguish them from so-called isogeny leaves. We will
only consider central leaves in this paper, so we will simply call these leaves.} corresponding to
$\mathbb{X}$ is the subset of $\mathscr{S}_{K,k} := \mathscr{S}_K\times_{\F_p}k$ where the fibers of the $p$-divisible group
$\cA[p^\infty]$ at all geometric points are isomorphic to $\mathbb{X}$:
\[
\mathscr{C}^{\mathbb X}:=\left\{x\in \mathscr{S}_{K,k}\mid \cA[p^\infty]\times k(\bar x)\cong \mathbb{X}\times_{k} k(\bar x)\right\}.
\]
(The isomorphisms are understood to be isomorphisms of $p$-divisible groups with $G$-structure.)
\end{defn}
\noindent By the argument in~\cite[Proposition 1]{mantovan}, this is a locally closed subset of $\mathscr{S}_{K,k}$ and becomes a smooth subscheme when endowed with the induced reduced structure.

The results of the previous section imply the following result.

\begin{cor}\label{cor:fpqc torsor} The scheme
\[
\mathfrak{Ig}^{\mathbb X}\to \mathscr C^{\mathbb X}\subset \mathscr S_{K,k}
\]
parametrizing isomorphisms $\cA[p^\infty]\cong \mathbb X$ of $p$-divisible groups with $G$-structure is representable by an $\underline{\mathrm{Aut}}(\mathbb X)$-torsor over $\mathscr C^{\mathbb X}$.
\end{cor}

\begin{proof} Apply Proposition~\ref{perfect Igusa two}.
\end{proof}

Moreover, $\mathfrak{Ig}^{\mathbb X}$ can be reinterpreted in terms of a moduli space of abelian varieties with $G$-structures up to $p$-power isogeny, and isomorphisms of $A[p^\infty]$ with $G$-action up to $p$-power isogeny, cf.~\cite[Lemma 4.3.4]{caraiani-scholze}. This shows in particular that $\mathfrak{Ig}^{\mathbb X}$ is perfect, and that an isogeny $\phi: \mathbb X\to \mathbb X'$ (compatible with extra structures) induces an isomorphism $\mathfrak{Ig}^{\mathbb X}\cong \mathfrak{Ig}^{\mathbb X'}$. In particular, we get pro-finite correspondences
\[
\mathscr C^{\mathbb X}\leftarrow \mathfrak{Ig}^{\mathbb X}\cong \mathfrak{Ig}^{\mathbb X'}\to \mathscr C^{\mathbb X'}
\]
between different leaves in the same isogeny class.

Now we wish to obtain a variant of Igusa varieties that works at finite level. For this, we work with completely slope divisible $p$-divisible groups:

\begin{defn}\leavevmode
Let $\mathscr{T}/\Spec \F_p$ be a scheme and $\mathscr{G}/\mathscr{T}$ a $p$-divisible
group. Let $\mathrm{Frob}_{\mathscr{G}}$ denote the Frobenius morphism relative to $\mathscr{T}$. 
\begin{enumerate} 
\item $\mathscr{G}$ is isoclinic and slope divisible of slope $\lambda\in \Q_{\geq 0}$ if one can write $\lambda=\frac{r}{s}$ so that the quasi-isogeny
\[
p^{-r}\mathrm{Frob}^s_{\mathscr{G}}: \mathscr{G}\to \mathscr{G}^{(p^s)}
\]
is an isomorphism.
\item $\mathscr{G}$ is slope divisible with respect to $\lambda\in \Q_{\geq 0}$ if one can write $\lambda=\frac{r}{s}$ so that the quasi-isogeny
\[
p^{-r}\mathrm{Frob}^s_{\mathscr{G}}: \mathscr{G}\to \mathscr{G}^{(p^s)}
\]
is an isogeny. 
\item $\mathscr{G}$ is completely slope divisible if it has a filtration by closed immersions of $p$-divisible groups
\[
0=\mathscr{G}_{\leq 0}\subset \mathscr{G}_{\leq 1}\subset\dots \subset \mathscr{G}_{\leq r}=\mathscr{G}
\]
such that, for each $i$, $\mathscr{G}_{\leq i}/\mathscr{G}_{\leq i-1}$ is isoclinic and slope divisible
of slope $\lambda_i$ and $\mathscr{G}_{\leq i}$ is slope divisible with respect to $\lambda_i$, where $\lambda_i$ is a strictly decreasing sequence of rational numbers.
\end{enumerate}
\end{defn}

\noindent We note that such slope decompositions are unique if they exist (as when $\mathscr G$ is isoclinic and slope divisible of slope $\lambda$ and $\mathscr G'$ is slope divisible with respect to $\lambda'>\lambda$, then there are no maps $\mathscr G'\to \mathscr G$) and stable under base change. In particular, the property of being completely slope divisible is fpqc local.

We will use repeatedly the following basic result on completely slope 
divisible $p$-divisible groups. 

\begin{lem}\label{lem:generic completely slope divisible} Let $\mathscr{T}/\Spec \F_p$ be a connected 
regular scheme and $\mathscr{G}/\mathscr{T}$ a $p$-divisible
group. Let $\eta$ be the generic point of $\mathscr{T}$, and $\overline{\eta}$ a geometric point above $\eta$. If $\mathscr{G}_{\overline{\eta}}$
is completely slope divisible, then so is $\mathscr{G}$.
\end{lem}

\begin{proof} This is shown in the first part of the proof of~\cite[Theorem 7]{zink}. 
\end{proof}

Choose a $p$-divisible group with $G$-structure $\mathbb X$ over an algebraically closed field $k$
of characteristic $p$. Assume that $\mathbb X$ is completely slope divisible (this is a condition without extra structures).
Note that a $p$-divisible group over $\overline{\F}_p$ is completely slope divisible if and only if it is a direct sum of its isoclinic pieces, 
which are defined over a finite field, cf.~\cite[Corollary 1.5]{oort-zink}. One can check that such a choice of $\mathbb{X}$ exists in any isogeny class.

We have
\[
\mathbb{X} = \oplus_{i=1}^r \mathbb{X}_i,
\]
where the $\mathbb{X}_i$ are isoclinic $p$-divisible groups of strictly decreasing slopes $\lambda_i\in [0,1]$.
The polarization $\lambda$ on $\mathbb{X}$ induces isomorphisms $\lambda^i: \mathbb{X}_i\to (\mathbb{X}_j)^\vee$ 
for all $i,j$ with $\lambda_i+\lambda_j = 1$ which satisfy $(\lambda^i)^\vee= - \lambda^j$.  Let $\mathscr G = \mathscr{G}_{\mathbb X}$ 
be the $p$-divisible group of the universal abelian variety $\cA/\mathscr{S}_{K}$ restricted to $\mathscr{C}^{\mathbb X}$. Then $\mathscr G$ is completely slope
divisible, since it is geometrically fiberwise constant. (Note that $\mathscr{C}^{\mathbb X}$ is a regular scheme and the fibers of $\mathscr G$
over every generic point of $\mathscr{C}^{\mathbb X}$ are completely slope divisible, so Lemma~\ref{lem:generic completely slope divisible} 
implies the result.) It has a slope filtration 
\[
0\subset \mathscr{G}_{\leq 1}\subset \dots\subset \mathscr{G}_{\leq r}=\mathscr{G}
\] 
for which the graded pieces $\mathscr{G}_i:=\mathscr{G}_{\leq i}/\mathscr{G}_{\leq i-1}$ are isoclinic of slope $\lambda_i$. The $\cO_F\otimes_{\Z}\Z_p$-action on $\mathscr{G}$ respects this filtration, 
so that each $\mathscr{G}_i$ is endowed with an $\cO_F\otimes_{\Z}\Z_p$-action. Moreover, the polarization on $\mathscr{G}$ induces isomorphisms
$l^i:\mathscr{G}_i\to (\mathscr{G}_j)^\vee$ for all $i,j$ with $\lambda_i+\lambda_j =1$, which satisfy $(l^{i})^\vee = -l^j$. 

\begin{defn}\label{definition of Igusa varieties} The (pro-)Igusa variety is the map 
\[
\mathrm{Ig}^{\mathbb X}\to \mathscr{C}^{\mathbb X}
\]
which over a $\mathscr{C}^{\mathbb X}$-scheme $\mathscr{T}$ parametrizes tuples $(\rho_i)_{i=1}^r$ of isomorphisms
\[
\rho_i :\mathscr{G}_i\times_{\mathscr{C}^{\mathbb X}}\mathscr{T}\toisom \mathbb{X}_i\times_k \mathscr{T}
\]
which are compatible with the $\cO_F\otimes_{\Z}\Z_p$-actions on $\mathscr{G}_i$ and $\mathbb{X}_i$
and commute with the polarizations on $\mathscr{G}$ and $\mathbb{X}$ up to an element of $\Z^\times_p(\mathscr{T})$
that is independent of $i$. 
\footnote{This element of $\Z_{p}^\times(\mathscr T)$ can be identified with an automorphism
of the multiplicative $p$-divisible group $\mu_{p^\infty,\mathscr T}$.}
\end{defn}

\begin{remark}\label{rem:finite level Igusa} We can also define $\mathrm{Ig}^{\mathbb X}_m$ for any $m\geq 0$ as the moduli 
space of isomorphisms on $\mathscr{C}^{\mathbb X}$-schemes $\mathscr{T}$ 
\[
\rho_{i,m}: \mathscr{G}_i[p^m]\times_{\mathscr{C}^{\mathbb X}}\mathscr{T}\toisom \mathbb{X}_i[p^m]\times_k \mathscr{T}
\] 
which (fppf locally on $\mathscr{T}$) lift to arbitrary $m'\geq m$ and which respect the extra structures. Explicitly,
this means that the isomorphisms $\rho_{i,m}$ commute with the $\cO_F\otimes_{\Z}\Z_p$-actions on 
$\mathscr{G}_i[p^m]$ and $\mathbb{X}_i[p^m]$ and for $i,j$ with $\lambda_i+\lambda_j=1$ they fit into a diagram of isomorphisms
\[
\xymatrix{\mathscr{G}_i[p^m]\times_{\mathscr{C}^{\mathbb X}}\mathscr{T}\ar[d]^{l^i}\ar[r]^{\rho_{i,m}} & \mathbb{X}_i[p^m]\times_k\mathscr{T}
\ar[d]^{\lambda^i}\\ (\mathscr{G}_j)^\vee[p^m]\times_{\mathscr{C}^{\mathbb X}}\mathscr{T} & (\mathbb{X}_j)^\vee[p^m]\times_k\mathscr{T}\ar[l]^{\rho_{j,m}^\vee}},
\]
that commutes up to an element of $(\Z/p^m\Z)^{\times}$ that is independent of $i$. By Theorem~\ref{thm:construction of Igusa covers},
$\mathrm{Ig}^{\mathbb X}_m$ is a finite \'etale cover of $\mathscr{C}^{\mathbb X}$, Galois with Galois group $\Gamma_{m,\mathbb X}$, the group of automorphisms of $\mathbb X[p^m]$ that lift to isomorphisms of $\mathbb X$ compatible with extra structure.
\end{remark}

In particular, passing to the inverse limit, we see that
\[
\mathrm{Ig}^{\mathbb X}\to \mathscr C^{\mathbb X}
\]
is a pro-finite \'etale cover with Galois group $\Gamma_{\mathbb X}=\mathrm{Aut}(\mathbb X)$.

\begin{remark} The scheme $\mathfrak{Ig}^{\mathbb X}$ maps naturally to $\mathrm{Ig}^{\mathbb X}$ (as is evident from the moduli description), and then even to its perfection. The resulting map $\mathfrak{Ig}^{\mathbb X}\to \mathrm{Ig}^{\mathbb X}_\perf$ is an isomorphism, as the slope filtration splits uniquely over a perfect base, cf.~\cite[Proposition 4.3.8]{caraiani-scholze}, or simply because it is a map between $\Gamma_{\mathbb X}$-torsors over $\mathscr C^{\mathbb X}_\perf$.
\end{remark}

\subsection{Serre--Tate theory}
Our goal in this subsection is to prove a version of Serre--Tate theory for semi-abelian schemes that are globally 
extensions of abelian schemes by tori. We follow Drinfeld's original proof, cf.~\cite{katzSerreTate}, but see also~\cite{andre}. Note
that we do not assume that we are working over a Noetherian base.

\begin{thm}\label{thm:rigidity of quasi-isogenies}
Let $S'\twoheadrightarrow S$ be a surjection of rings in which $p$ is nilpotent, with nilpotent kernel $I\subset S'$.  
\begin{enumerate}
\item The functor $\mathscr{G}_{S'}\mapsto \mathscr{G}_S:=\mathscr{G}_{S'}\times_{S'}S$ from $p$-divisible groups up to isogeny over $S'$
to $p$-divisible groups up to isogeny over $S$ is an equivalence of categories. 
\item The functor $A_{S'}\mapsto A_S:= A_{S'} \times_{S'}S$ from abelian varieties up to $p$-power isogeny over $S'$
to abelian varieties up to $p$-power isogeny over $S$ is an equivalence of categories. 
\item We now consider the category $\mathcal{R}_{S'}$ of semi-abelian schemes $A_{S'}$ which are globally over $S'$
an extension 
\[
0\to T_{S'}\to A_{S'}\to B_{S'} \to 0
\]
of an abelian scheme $B_{S'}$ by a split torus $T_{S'}$ (necessarily of constant rank over $S'$), 
with morphisms in $\mathcal{R}_{S'}$ given by $\Hom(A_{S'}, A'_{S'})[1/p]$.
Then the functor $\mathcal{R}_{S'}\to \mathcal{R}_{S}$ given by $A_{S'}\mapsto A_{S}:=A_{S'}\times_{S'} S$ is
an equivalence of categories. 
\end{enumerate}
\end{thm}

\begin{proof} Consider first the case of $p$-divisible groups. We prove that the functor $\mathscr{G}_{S'}\mapsto \mathscr{G}_S$ is faithful. 
Let $\mathscr{G}_{S'}$ be a $p$-divisible group over $S'$. 
Then its formal completion along its identity section $\widehat{\mathscr{G}}_{S'}$ is a formal Lie group. 
Let $T'$ be an affine $S'$-scheme and set $T:=T'\times_{S'}S$. Then
\[
\ker(\mathscr{G}_{S'}(T')\to \mathscr{G}_{S}(T))\toisom \ker(\widehat{\mathscr{G}}_{S'}(T')\to \widehat{\mathscr{G}}_{S}(T))
\]
is killed by $p^n$ for some sufficiently large positive integer $n$ that depends only on the degree of nilpotence of the ideal $(p,I)\subset S'$. This follows from~\cite[Ch. II, \S 4]{messing}. 

Now let $\mathscr{G}_{i,S'}$ be $p$-divisible groups over $S'$ for $i=1,2$ and assume $f: \mathscr G_{1,S'}\to\mathscr G_{2,S'}$ is a map that reduces to $0$ on $S$. Then for all $T'$, the image of $f(T')$ lies in $\ker(\mathscr{G}_{2,S'}(T')\to \mathscr{G}_{2,S}(T))$ which is killed by $p^n$. It follows that $p^n f=0$, which implies that $f=0$ as $\Hom$'s between $p$-divisible groups are torsion-free.

We now prove that the functor for $p$-divisible groups is full. Let $f\in \mathrm{Hom}(\mathscr{G}_{1,S}, \mathscr{G}_{2,S})$. We want to
show that there exists a lifting $\tilde{g}\in \mathrm{Hom}(\mathscr{G}_{1,S'}, \mathscr{G}_{2,S'})$ of $p^n f$,
where $n$ is the sufficiently large positive integer chosen above. Let $x\in \mathscr{G}_{1,S'}(T')$ with image $\bar{x}\in \mathscr{G}_{1,S}(T)$. Since $\mathscr{G}_{2,S'}$ is formally smooth
over $S'$, there exists a lift $\tilde{f}(\bar{x})\in \mathscr{G}_{2,S'}(T')$ of $f(\bar{x})$. Since $p^n$ kills 
\[
\ker(\mathscr{G}_{2,S'}(T')\to \mathscr{G}_{2,S}(T)),
\]
the element $\tilde{g}(x):=p^n \tilde{f}(\bar{x})$ is well-defined and gives rise to an element 
\[
\tilde{g}\in \mathrm{Hom}(\mathscr{G}_{1,S'}(T'), \mathscr{G}_{2,S'}(T'))
\]
lifting $p^n f$.

Essential surjectivity in the case of $p$-divisible groups follows from~\cite[Th\'eor\`eme 4.4]{illusie}. 

We now consider the case of abelian schemes. The full faithfulness is proved in the same way as
for $p$-divisible groups. The key point is that $\ker(\widehat{A}_{S'}(T')\to \widehat{A}_{S}(T))$ is the same as 
$\ker(\widehat{A}_{S'}(T')\to \widehat{A}_{S}(T))$, where $\widehat{A}$ is the formal completion
of $A$ along its identity section and is a formal Lie group. Essential surjectivity for abelian schemes is standard (the deformation theory of abelian schemes is unobstructed).

We now consider the case of semi-abelian schemes. The full faithfulness is also proved in the same way as for $p$-divisible groups: 
note that $A_{S'}$ is still smooth and therefore formally smooth. Moreover, $\ker(A_{S'}(T')\to A_{S}(T))$ can still be identified with 
$\ker(\widehat{A}_{S'}(T')\to \widehat{A}_{S}(T))$, where $\widehat{A}$ is the formal completion
of $A$ along its identity section and is a formal Lie group. For essential surjectivity, note that we can first lift the base abelian scheme $B$; afterwards, lifting the semiabelian scheme is equivalent to lifting a section of the dual abelian scheme, which is possible by (formal) smoothness.
\end{proof}

\begin{thm}\label{thm:Serre-Tate}
Let $S'\twoheadrightarrow S$ be a surjection of rings in which $p$ is nilpotent, with nilpotent kernel $I\subset S'$.  
\begin{enumerate}
\item Consider the category of triples $(A_S,\mathscr{G}_{S'}, \rho)$,
where $A_S$ is an abelian scheme over $S$, $\mathscr{G}_{S'}$ is a $p$-divisible group over $S'$ and
\[
\rho: A_S[p^\infty]\toisom \mathscr{G}_{S'}\times_{S'}S 
\]  
is an isomorphism. The functor $A_{S'}\mapsto (A_{S'}\times_{S'}S, A_{S'}[p^\infty], \mathrm{id})$ induces an equivalence of 
categories between the category of abelian schemes over $S'$ and the category of triples defined above.  

\item Consider the category of triples $(A_S,\mathscr{G}_{S'}, \rho)$,
where $A_S$ is an extension 
\[
1\to T_S\to A_S\to B_S\to 1
\] 
of an abelian scheme $B_S$ by a split torus $T_S$ over $S$ 
(necessarily of constant rank over $S$), $\mathscr{G}_{S'}$ is a $p$-divisible group over $S'$ 
and
\[
\rho: A_S[p^\infty]\toisom \mathscr{G}_S:=\mathscr{G}_{S'}\times_{S'}S 
\]  
is an isomorphism. We let morphisms between these triples be compatible morphisms of triples.  
The functor $A_{S'}\mapsto (A_{S'}\times_{S'}S, A_{S'}[p^\infty], \mathrm{id})$ induces an equivalence of 
categories between the category of extensions 
\[
1\to T_{S'}\to A_{S'}\to B_{S'}\to 1
\]
over $S'$ and the category of triples defined above.  
\end{enumerate}
\end{thm}

\begin{proof}
Consider first the case of abelian schemes. We prove that the functor is fully faithful. 
The faithfulness follows from the fact that we have an injection 
\[
\mathrm{Hom}(A_{1,S'}, A_{2,S'})\hookrightarrow \mathrm{Hom}(A_{1,S}, A_{2,S}),
\]
as proved in Theorem~\ref{thm:rigidity of quasi-isogenies}. To prove that the functor is also full, let $f: A_{1,S}\to A_{2,S}$ and $f^{\infty}: \mathscr{G}_{1,S'}\to \mathscr{G}_{2,S'}$
be compatible with the isomorphisms $\rho_1,\rho_2$. We can lift $p^n f$ to a morphism $\tilde{g}: A_{1,S'}\to A_{2,S'}$,
for $n$ as in the proof of Theorem~\ref{thm:rigidity of quasi-isogenies}. In order to obtain a lift $\tilde{f}: A_{1,S'}\to A_{2,S'}$ of $f$, we need to show that 
$\tilde{g}$ is divisible by $p^n$, in other words that $A_{1,S'}[p^n]\subseteq \ker \tilde{g}$. 
This can be checked on the level of $p$-divisible groups, where it is automatic that $p^n f^\infty$ is divisible by $p^n$.

For essential surjectivity, take a triple $(A_S,\mathscr G_{S'},\rho)$ and pick any abelian scheme $A'_{S'}$ over $S'$ with an isogeny $f: A'_S\to A_S$, which exists by Theorem~\ref{thm:rigidity of quasi-isogenies}. We get an isogeny $\rho f^\infty: A'_S[p^\infty]\to \mathscr G_S$. Up to multiplying $f$ by $p^n$, we may assume that this lifts (necessarily uniquely) to an isogeny $\tilde{\rho}: A'_{S'}[p^\infty]\to \mathscr G_{S'}$. Let $K=\ker \tilde{\rho}\subset A'_{S'}$, which is a finite flat group scheme. Replacing $A'_{S'}$ by $A'_{S'}/K$, we get the result.

Consider now the case of semi-abelian schemes. The fact that the functor is fully faithful is proved in the same way as above. In the proof of essential surjectivity, the only thing to make sure is that the quotient $A'_{S'}/K$ is indeed a semi-abelian scheme of the desired form. But note that this quotient is a flat algebraic space over $S'$ whose base change to $S$ is $A_S$. This implies that it is a scheme, smooth over $S$, with all fibers being semi-abelian schemes, i.e.~it is a semi-abelian scheme. Moreover, by the rigidity of multiplicative groups as in~\cite[Expos\'e IX Th\'eor\`eme 3.6 bis]{SGA3}, the torus $T_S\subset A_S$ deforms uniquely to a torus $T_{S'}\subset A_{S'}$, and then necessarily the quotient is an abelian scheme (as it deforms an abelian scheme), showing that $A_{S'}$ is still an extension of an abelian variety by a split torus $T_{S'}$.
\end{proof}

\subsection{Compactifications}

In this section, we recall the minimal and toroidal compactifications of the Shimura varieties under consideration.

\subsubsection{Degenerations} Before explaining the precise combinatorics of the compactifications, let us quickly describe the $C$-valued points of $\mathscr S_K$ where $C$ is some complete algebraically closed nonarchimedean field with ring of integers $\mathcal O_C$, such that $\Delta_F^{-1}\in \mathcal O_C$.

Let us recall first the structure of principally polarized abelian varieties $(A,\lambda)$ over $C$. By the semistable degeneration theorem, there is a unique semiabelian scheme $\mathcal A$ over $\mathcal O_C$ with generic fibre $A$. The completion $\widehat{\mathcal A}$ of $\cA$ along its special fibre sits in an exact sequence
\[
0\to \widehat{T}\to \widehat{\mathcal A}\to \widehat{\mathcal B}\to 0
\]
where $\widehat{T}$ is the completion of a torus $T$ over $\mathcal O_C$, and $\widehat{\mathcal B}$ is the completion of an abelian variety $\mathcal B$ over $\mathcal O_C$. In fact, this short exact sequence algebraizes uniquely to a short exact sequence
\[
0\to T\to \mathcal G\to \mathcal B\to 0\ ,
\]
where $\mathcal G$ is the so-called Raynaud extension. The polarization induces a line bundle $\mathcal L$ on $A\times A$; its pullback to the generic fibre of $\widehat{\mathcal G}\times \widehat{\mathcal G}$ arises from a line bundle on the generic fibre of $\mathcal B\times \mathcal B$, inducing a principal polarization of the generic fibre of $\mathcal B$, and thus on $\mathcal B$ itself by properness. Thus, $\mathcal B$ carries a canonical principal polarization $\lambda_{\mathcal B}: \mathcal B\cong \mathcal B^\vee$.

Let $X$ be the cocharacter group of $T$ (as $\mathcal O_C$ is strictly henselian, this is simply a finite free abelian group). Then the Raynaud extension is given by a map $X\to \mathcal B^\vee$. Using the principal polarization of $\mathcal B$, this gives a map $f_0: X\to \mathcal B$. On the other hand, the map $\widehat{\mathcal G}\to \widehat{\mathcal A}$ extends to a map between the rigid-analytic generic fibres of $\mathcal G$ and $\mathcal A$; this is a covering map, and identifies $A$ as a rigid-analytic variety with a quotient $\mathcal G_\eta/Y$ for some discrete subgroup $Y\subset \mathcal G(C)$. The principal polarization in fact induces an identification $Y\cong X$ under which the composite $Y\to \mathcal G(C)\to \cB(C)$ is identified with the given map $f_0: X\to \cB(\mathcal O_C) = \cB(C)$. The lift $f: X\to \mathcal G(C)$ of $f_0: X\to \cB(C)$ is equivalent to a section of the pullback of the Poincar\'e bundle over $\cB\times \cB^\vee\cong \cB\times \cB$ to $X\times X$; this section has to be symmetric (recalling that the Poincar\'e bundle is symmetric). Moreover, for all $x\in X$, the section at $(x,x)$ must be topologically nilpotent (this condition is independent of a local integral trivialization of the Poincar\'e bundle).

Conversely, given a symmetric lift $f: X\to \mathcal G(C)$ of the given map $X\to B(C)$ satisfying this topological nilpotence condition, one can form the quotient $\mathcal G_\eta/X$ to get a principally polarized abelian variety $A$ over $C$.

If $(A,\lambda)$ carries in addition an $\mathcal O_F$-action, all objects involved will carry an $\mathcal O_F$-action; note that $X$ will then automatically be a finite projective $\mathcal O_F$-module. Regarding level structures, we note that for any integer $N$ invertible in $C$, the $N$-torsion $A[N]=A(C)[N]$ of $A$ carries a canonical filtration
\[
0\subset T(C)[N]\subset \mathcal G(C)[N]\subset A(C)[N]
\]
with associated gradeds
\[
\mathrm{Gr}_{-2} = T(C)[N], \mathrm{Gr}_{-1} = \mathcal G(C)[N]/T(C)[N] = B(C)[N],
\] 
\[
\mathrm{Gr}_0 = A(C)[N] / \mathcal G(C)[N] = X/NX\ .
\]
Moreover, $T(C)[N]=G(C)[N]^\bot$, and $G(C)[N]=T(C)[N]^\bot$, under the Weil pairing on $A(C)[N]$. Thus, any level structure $\eta: A(C)[N]\cong L/N$ will induce a filtration $\mathrm{Z}_N = \{\mathrm{Z}_{N,-2}\subset \mathrm{Z}_{N,-1}\subset L/N\}$ such that $\mathrm{Z}_{N,-2}=\mathrm{Z}_{N,-1}^\bot$ and $\mathrm{Z}_{N,-1} =\mathrm{Z}_{N,-2}^\bot$ (and the filtration lifts to a similar filtration modulo $M$ for all $N|M$). Let $\mathrm{Gr}^{\mathrm{Z}_N}_i$ denote the associated gradeds for $i=-2,-1,0$. Let us pick an $\mathcal O_F$-linear splitting of this sequence, so

\begin{equation}\label{eq:chosen splitting}
L/N L = \mathrm{Z}_{N,-2}\oplus \mathrm{Z}_{N,-1}/\mathrm{Z}_{N,-2}\oplus (L/N)/\mathrm{Z}_{N,-1} = \bigoplus_{i=-2}^0 \mathrm{Gr}^{\mathrm{Z}_N}_i\ .
\end{equation}

In that case, level structures $\eta: A(C)[N]\cong L/N$ inducing the given filtration $\mathrm{Z}_N$ are in bijection with the following data:
\begin{enumerate}
\item $\mathcal O_F$-linear isomorphisms $\eta_{-2}: T(C)[N]\cong \mathrm{Gr}^{\mathrm{Z}_N}_{-2}$, $\eta_{-1}: B(C)[N]\cong \mathrm{Gr}^{\mathrm{Z}_N}_{-1}$, $\eta_0: X/NX\cong \mathrm{Gr}^{\mathrm{Z}_N}_{0}$ such that $\eta_{-2}$ and $\eta_0$ are dual under the canonical pairing, and $\eta_{-1}$ is compatible with the Weil pairing;
\item an $\mathcal O_F$-linear splitting of the short exact sequence
\[
0\to T(C)[N]\to \mathcal G(C)[N]\to B(C)[N]\to 0\ .
\]
This amounts to giving an extension of the given map $f_0: X\to B(C)$ to a map $\tfrac 1N X\to B(C)$. 
\item An $\mathcal O_F$-linear extension of the map $f:X \to \mathcal G(C)$ to a map $\tfrac 1N X\to \mathcal G(C)$
whose composition with $\mathcal{G}(C)\to B(C)$ gives the extension in (2), which amounts to
a splitting $X/NX \to A(C)[N]$,   
\end{enumerate}
subject to the condition that the induced isomorphism
\begin{equation}\label{eq:compatibility}
L/N\cong \bigoplus_{i=-2}^0 \mathrm{Gr}^{\mathrm{Z}_N}_i\cong T(C)[N]\oplus B(C)[N]\oplus X/NX\cong A(C)[N]
\end{equation}
is compatible with the Weil pairing, and that this data lifts to similar data modulo $M$ for all $N|M$. In~\eqref{eq:compatibility},
the first isomorphism is the splitting chosen in~\eqref{eq:chosen splitting}, the second 
is induced by the $\eta_{i}$ for  $i=-2,-1,0$, and the third is induced by the splittings in (2) and (3). 

\subsubsection{Cusp labels}\label{cusp labels} In our situation (using critically that the alternating form $(\cdot,\cdot)$ on $L$ is perfect), we can define a cusp label to be a pair $(\mathrm{Z},X)$ where
\begin{enumerate}
\item The filtration
\[
\mathrm{Z}=\{\mathrm{Z}_{-2}\subset \mathrm{Z}_{-1}\subset L\otimes_{\Z} \widehat{\Z}\}
\]
is $\mathcal O_F$-stable and symplectic in the sense that $\mathrm{Z}_{-1}=\mathrm{Z}_{-2}^\bot$ and $\mathrm{Z}_{-2} = \mathrm{Z}_{-1}^\bot$ with respect to the perfect alternating form $(\cdot,\cdot)$ on $L$.
\item $X$ is a finite projective $\mathcal O_F$-module with an $\mathcal O_F$-linear isomorphism
\[
X\otimes_{\Z} \widehat{\Z}\to \mathrm{Gr}^Z_0 = L\otimes_{\Z} \widehat{\Z}/\mathrm{Z}_{-1}.
\]
\end{enumerate}

In the following, we write $\mathrm{Gr}^{\mathrm{Z}}_{-2} = \mathrm{Z}_{-2}$, $\mathrm{Gr}^{\mathrm{Z}}_{-1} = \mathrm{Z}_{-1}/\mathrm{Z}_{-2}$ and $\mathrm{Gr}^{\mathrm{Z}}_0 = L\otimes_{\Z} \widehat{\Z}/\mathrm{Z}_{-1}$. Then $\mathrm{Gr}^{\mathrm{Z}}_{-2}$ and $\mathrm{Gr}^{\mathrm{Z}}_0$ are in natural perfect duality, and $\mathrm{Gr}^{\mathrm{Z}}_{-1}$ admits a natural perfect pairing. Note also that in (2), one could equivalently ask for an isomorphism $X^\vee\otimes_{\Z}\widehat{\Z}\cong \mathrm{Z}_{-2}$, where $X^\vee = \Hom(X,\mathbb Z)$ is the dual of $X$.

There is an action of $G(\mathbb A_f)$ on cusp labels $(\mathrm{Z},X)$. Indeed, cusp labels are in bijection with pairs $(\mathrm{Z}_{\mathbb Q},X_{\mathbb Q})$ defined similarly but on the rational level, and there is an obvious action on the latter.

\begin{defn}For $K\subset G(\A_f)$ a compact open subgroup, a cusp label at level $K$ is a $K$-orbit of cusp labels $(\mathrm{Z},X)$. If $H\subset G(\A_f)$ is any closed subgroup, a cusp label at level $H$ is a compatible collection of cusp labels at level $K$ for all open compact subgroups $K\subset G(\A_f)$ containing $H$.
\end{defn}

In case $K=K(N)$ as above, we note that the cusp labels at level $K$ can be identified with pairs $(\mathrm{Z}_N,X)$ where $\mathrm{Z}$ is an $\mathcal O_F$-stable filtration
\[
\mathrm{Z}_{N,-2}\subset \mathrm{Z}_{N,-1}\subset L/N
\]
that admits a lift to a filtration $\mathrm{Z}$ as above (in particular, $\mathrm{Z}_{N,-2} = \mathrm{Z}_{N,-1}^\bot$ and $\mathrm{Z}_{N,-1} = \mathrm{Z}_{N,-2}^\bot$), and $X$ is a finite projective $\mathcal O_F$-module with an $\mathcal O_F$-linear isomorphism $X/N\cong \mathrm{Gr}^{\mathrm{Z}_N}_0$.

Given a cusp label $Z=(\mathrm{Z}_N,X)$ at principal level $K=K(N)$, we define the group
\[
\Gamma_Z=\{g\in \mathrm{GL}_{\mathcal O_F}(X)\mid g\equiv 1\mod N\}\ .
\]
Moreover, we define a smaller Shimura variety $\mathscr S_Z$ associated to $Z$ as follows, via its moduli problem. Let $r$ be the $\mathcal O_F$-rank of $X$. Then $\mathscr S_Z^{\mathrm{pre}}/\mathbb Z[\tfrac 1{\Delta_F}]$ parametrizes over a test scheme $S$ quadruples $(B,\iota,\lambda,\eta)$ where
\begin{enumerate}
\item $B$ is an abelian scheme of dimension $[F:\mathbb Q](n-r)$;
\item $\iota: \mathcal O_F\to \mathrm{End}(B)$ is an action making $\Lie B$ free of rank $n-r$ over $\mathcal O_F\otimes_{\mathbb Z} \mathcal O_S$;
\item $\lambda: B\cong B^\vee$ is a principal polarization whose Rosati involution is compatible with complex conjugation on $\mathcal O_F$ along $\iota$;
\item $\eta: \mathrm{Gr}^{\mathrm{Z}_N}_{-1}\to B[N]$ is a map compatible with $\iota$ and such that the diagram
\[\xymatrix{
\mathrm{Gr}^{\mathrm{Z}_N}_{-1}\times \mathrm{Gr}^{\mathrm{Z}_N}_{-1}\ar[r]\ar[d] & \Z/N\Z\ar[d]^{\zeta_N}\\
B[N]\times B[N]\ar[r] & \mu_N
}\]
commutes for some primitive $N$-th root of unity $\zeta_N\in \mathcal O_S$; here, the upper map is the natural perfect pairing on $\mathrm{Gr}^{\mathrm{Z}_N}_{-1}$ and the lower map is the Weil pairing induced by $\lambda$. Moreover, we demand that \'etale locally there is a lifting to such a level structure modulo $N\Delta_F^m$ for all $m\geq 0$.
\end{enumerate}

\noindent Let $\mathscr S_Z$ be the normalization of $\mathscr S_Z^{\mathrm{pre}}$ in $\mathscr S_Z^{\mathrm{pre}}\times \Z[\tfrac 1{\Delta_F N}]$. Note that this is again one of the unitary Shimura varieties of our setup, with $n$ replaced by $n-r$. We will make use of this observation for some inductive arguments.

\subsubsection{Cone decompositions}\label{cone decompositions} Each cusp label $Z=(\mathrm{Z}_N,X)$ at level $K(N)$ determines an $\R$-vector space $M_Z$ of symmetric pairings 
\[
(\cdot,\cdot): X\otimes \mathbb R\times X\otimes \mathbb R\to \R
\]
such that $(av,w) = (v,\overline{a}w)$ for $v,w\in X\otimes \R$, $a\in F$; these are equivalent to Hermitian pairings
\[
\langle\cdot,\cdot\rangle: X\otimes \mathbb R\times X\otimes \mathbb R\to F\otimes \R\ .
\]
We consider the cones $P^+_Z\subset P_Z\subset M_Z$, where $P^+_Z$ is the cone of positive definite Hermitian pairings in $M_Z$ and $P_Z$ is the cone of positive semidefinite Hermitian pairings with $F$-rational radicals\footnote{In general, one would impose the condition of admissible radicals, but in our case the order $\cO_F$ is maximal and having admissible radicals is equivalent to having rational radicals, cf.~\cite[Remark 6.2.5.5]{lan-thesis}.}. There is an action of $\Gamma_Z$ on $M_Z$ and both $P^+_Z$ and $P_Z$ are stable under this action. 

\begin{remark} In our case we can identify $M_Z$ with $[F^+:\Q]$ copies of the space of Hermitian matrices in $M_r(\C)$, where $r$ is the $\cO_F$-rank of $X$. The subspace $P^+_Z$ is then obtained by taking $[F^+:\Q]$ copies of the space of positive definite Hermitian matrices in $M_r(\C)$, but $P_Z$ does not decompose into a product.
\end{remark}

The cusp label $Z$ also determines the $\Z$-lattice $\mathbb{S}_Z\subset M^\vee_Z$ that is the image of $\frac 1N(X\otimes X)$ under the natural map
\[
\tfrac 1N(X\otimes X)\to M^\vee_Z
\]
sending $v\otimes w\in X\otimes X$ to the map taking the alternating form $(\cdot,\cdot)\in M_Z$ to $(v,w)$. Then $\mathbb S_Z$ is stable under the action of $\Gamma_Z$ on $M_Z^\vee$. A \emph{rational polyhedral cone} $\sigma\subset M_Z$ is a subset of the form 
\[
\sigma = \R_{>0}v_1+\ldots +\R_{>0}v_r
\]
for $v_1,\dots,v_r\in \mathbb{S}^\vee_{Z}$. (In particular, $\{0\}$ is a rational polyhedral cone because it is obtained from the empty sum.) A rational polyhedral cone $\sigma\subset M_Z$ is \emph{non-degenerate} if the closure $\bar{\sigma}$ of $\sigma$ does not contain any non-trivial $\R$-vector subspace of $M_Z$.  A rational polyhedral cone $\sigma\subset M_Z$ is \emph{smooth} if it is of the form   
\[
\sigma = \R_{>0}v_1+\ldots +\R_{>0}v_r
\]  
for $v_1,\dots,v_r\in \mathbb{S}^\vee_{Z}$ which extend to a basis of $\mathbb{S}^\vee_{Z}$. A rational polyhedral cone $\tau$ is a \emph{face} of a rational polyhedral cone $\sigma$ if there exists a linear functional $\lambda: M_Z\to \R$ with $\lambda(\sigma)\subset \R_{\geq 0}$ and $\bar{\tau}=\bar{\sigma}\cap \lambda^{-1}(0)$. (This definition implies that $\sigma$ is always a face of itself.)

A rational polyhedral cone $\sigma\subset M_Z$ determines semigroups 
\[
\sigma^\vee = \left\{l\in \mathbb{S}_Z\mid l(v)\geq 0, \forall v\in \sigma \right\},
\]
\[
\sigma^\vee_0= \left\{l\in \mathbb{S}_Z\mid l(v) >0, \forall v\in \sigma \right\},
\]
\[
\sigma^\perp = \left\{l\in \mathbb{S}_Z\mid l(v) = 0, \forall v\in \sigma \right\}.
\]

A $\Gamma_Z$-\emph{admissible rational polyhedral cone decomposition} is a set $\Sigma_{Z}$ of non-degenerate rational polyhedral cones such that
\begin{enumerate}
\item The cones in $\Sigma_{Z}$ are disjoint and $P_{Z}=\cup_{\sigma \in \Sigma_Z}\sigma$. 
\item For each $\sigma\in \Sigma_{Z}$ and each face $\tau$ of $\sigma$ we have $\tau\in \Sigma_Z$. 
\item The set $\Sigma_Z$ is invariant under $\Gamma_Z$ and the set of orbits $\Sigma_Z/\Gamma_Z$ is finite. 
\end{enumerate}  
Such a $\Sigma_Z$ has a $\Gamma_Z$-stable subset $\Sigma_Z^+$ forming a rational polyhedral cone decomposition of $P^+_Z$. We call $\Sigma_Z$ \emph{smooth} if each cone $\sigma \in \Sigma_Z$ is smooth. 

Let us write $Z\leq Z'$ if $Z=(\mathrm{Z}_N,X),Z'=(\mathrm{Z}_{N'},X')$ are cusp labels at level $K$ that admit lifts $(\mathrm{Z},X)$, $(\mathrm{Z}',X')$ such that $\mathrm{Z}_{-2}\subset \mathrm{Z}'_{-2}$ and $X^\vee\subset (X')^\vee$. In that case, fixing such a lift, the injection $X^\vee\to (X')^\vee$ induces a surjection $X^\prime\to X$ and then inclusions $M_{Z'}\subset M_Z$ and $P_{Z'}\subset P_Z$. We say that two admissible rational polyhedral cone decompositions $\Sigma_Z$ and $\Sigma_{Z'}$ are \emph{compatible} if for each $\sigma\in \Sigma_{Z'}$ we also have $\sigma\in \Sigma_Z$ (via the inclusion $M_{Z'}\subset M_Z$). A \emph{compatible family} of cone decompositions at level $K$ is a collection $\Sigma = \left\{\Sigma_Z\right\}$ of $\Gamma_Z$-admissible rational polyhedral cone decompositions for each cusp label $Z$ at level $K$ that are pairwise compatible.

\begin{remark}\label{trivial stabilizer} We can and do assume that each $\Sigma$ is smooth 
and projective, in the sense of~\cite[Definition 7.3.1.1]{lan-thesis}, and that for each cusp label 
$Z$ and $\sigma\in \Sigma^+_{Z}$, its stabilizer $\Gamma_\sigma$ in $\Gamma_Z$ is trivial. 
The fact that compatible families of cone decompositions satisfying 
these properties exist follows from~\cite[Proposition 7.3.1.4]{lan-thesis}. Moreover, we note that such a cone decomposition at some principal level $K(N)$ induces a cone decomposition with the same properties at principal level $K(M)$ for any $N|M$. For the following, we fix some smooth projective cone decomposition $\Sigma$ (satisfying the assumption that all stabilizers $\Gamma_\sigma$ are trivial) at some auxiliary principal level $K(N)$ in the beginning, and pull it back to any other principal level considered.
\end{remark}

\subsubsection{Compactifications}\label{sec:compactifications}
Under the above assumptions, in particular $K=K(N)$, and away from the primes dividing $N$, the Shimura variety $\mathscr{S}_K$ has good toroidal and minimal compactifications $\mathscr{S}^{\tor}_{K,\Sigma}$ and $\mathscr{S}^*_K$, whose properties we summarize below. Actually, we will define a naive extension of the toroidal compactification to unramified primes dividing $N$.

\begin{thm}\label{existence of minimal compactifications} 
There exists a flat, projective, normal scheme $\mathscr{S}^*_K/ \Spec \Z[\tfrac 1{N\Delta_F}]$ 
together with a dense open embedding 
\[
j: \mathscr{S}_K[\tfrac 1N]\hookrightarrow \mathscr{S}^*_K,
\] 
satisfying the following additional properties:  
\begin{enumerate}
\item For each cusp label $Z$ at level $K$, there is a canonical locally closed immersion 
$\mathscr{S}_Z[\tfrac 1N]\hookrightarrow \mathscr{S}_K^*$, where $\mathscr{S}_Z$ is the Shimura variety defined above.
\item The incidence relations among strata are determined by the partial 
order relation on cusp labels: $\mathscr{S}_Z[\tfrac 1N]$ lies in the closure
of $\mathscr{S}_{Z'}[\tfrac 1N]$ if and only if $Z\leq Z'$. 
Moreover, this property is preserved under pullback to any fiber over $\Spec \Z[\tfrac 1{N\Delta_F}]$. 
\item If $K_1=K(N_1),K_2=K(N_2)\subset G(\widehat{\Z})$ with $N_1,N_2\geq 3$ and 
$g\in G(\A_f)$ satisfies $gK_1g^{-1}\subset K_2$ then there is a finite surjective morphism 
\[
[g]: \mathscr{S}^*_{K_1}[\tfrac 1{N_1N_2}]\to \mathscr{S}^*_{K_2}[\tfrac 1{N_1N_2}]
\]
extending the usual morphism $[g]: \mathscr{S}_{K_1}[\tfrac 1{N_1N_2}]\to \mathscr{S}_{K_2}[\tfrac 1{N_1N_2}]$. 
\end{enumerate}
\end{thm}

\begin{proof} This summarizes the main results in~\cite[\S 7]{lan-thesis}. There is one subtlety, namely the identification 
of the boundary stratum $\mathscr{S}_Z$. See~\cite[Definition 5.4.2.6]{lan-thesis} and~\cite[Remark 4.3.3]{lan-stroh} 
for an explanation of why the description of the boundary stratum simplifies in our case, 
when the level $K=K(N)$ is principal, when $p\nmid N\Delta_F$, and when~\cite[Condition 1.4.3.10]{lan-thesis} is satisfied. 
\end{proof}

In the following statement, we assume that for all primes $p$ unramified in $F$, the part of $N$ prime to $p$ is still $\geq 3$. 
(Otherwise, we would have to talk about Deligne--Mumford stacks.)

\begin{thm}\label{existence of toroidal compactifications} Given a choice $\Sigma$ of 
a compatible family of cone decompositions as in \S\ref{cone decompositions} 
(and Remark~\ref{trivial stabilizer}), there exists a projective scheme 
$\mathscr{S}^{\tor}_{K,\Sigma}/\Spec \Z[\tfrac 1{\Delta_F}]$, together with an open dense embedding 
$j^{\tor}_{\Sigma}: \mathscr{S}_K\hookrightarrow \mathscr{S}^{\tor}_{K,\Sigma}$, 
satisfying the following additional properties:
\begin{enumerate}
\item After inverting $N$, the scheme $\mathscr S^{\tor}_{K,\Sigma}[\tfrac 1N]$ is smooth.
\item There is a set-theoretic decomposition 
\[
\mathscr{S}^{\tor}_{K,\Sigma} = \bigsqcup_{Z}\mathscr{S}^{\tor}_{K,\Sigma,Z},
\]
where the disjoint union runs over the set of cusp labels at level $K$ where each $\mathscr S^{\tor}_{K,\Sigma,Z}$ is reduced, flat over $\Z$, and locally closed.
\item Let $\widehat{\mathscr{S}}^{\tor}_{K,\Sigma,Z}$ denote the formal completion of 
$\mathscr{S}^{\tor}_{K,\Sigma}$ along $\mathscr{S}^{\tor}_{K,\Sigma,Z}$. Moreover, pick a splitting
\[
L/N\cong \mathrm{Gr}^{\mathrm{Z}_N}_{-2}\oplus \mathrm{Gr}^{\mathrm{Z}_N}_{-1}\oplus \mathrm{Gr}^{\mathrm{Z}_N}_{0}
\]
of the filtration $\mathrm{Z}_N$. Then we have the following canonical description of the formal completion $\widehat{\mathscr{S}}^{\tor}_{K,\Sigma,Z}$.
There exists an abelian scheme 
\[
C_Z\to \mathscr{S}_Z
\]
together with a compatible action of $\Gamma_Z$. 
There exists a torsor under the split torus with character group $\mathbb{S}_Z$, 
\[
\Xi_Z\to C_Z,
\]
with a compatible action of $\Gamma_Z$. We have 
\[
\Xi_Z = \underline{\Spec}_{\cO_{C_Z}}\bigoplus_{l\in \mathbb{S}_Z}\Psi_Z(l),
\]
where for each $l\in\mathbb{S}_Z$, $\Psi_Z(l)$ is a line bundle on $C_{Z}$ and for each $l,l'\in \mathbb{S}_Z$
there is an isomorphism $\Psi_Z(l)\otimes\Psi_Z(l')\cong \Psi_Z(l+l')$, giving 
$\bigoplus_{l\in \mathbb{S}_Z}\Psi_Z(l)$ the structure of a sheaf of $\cO_{C_Z}$-algebras.
The choice of $\Sigma_Z$ determines a $\Gamma_Z$-equivariant relative torus embedding 
\[
\partial_{Z,\Sigma_Z}: \Xi_{Z}\hookrightarrow \Xi_{Z,\Sigma_Z}.
\]
We let $\mathfrak{X}_{Z,\Sigma_Z}$ denote the formal completion of $\Xi_{Z,\Sigma_Z}$ along the boundary
$\partial_{Z,\Sigma_Z}$ of the relative torus embedding. We may form the quotient $\mathfrak{X}_{Z,\Sigma_Z}/\Gamma_Z$
as a formal scheme and we have a canonical isomorphism of formal schemes 
\[
\widehat{\mathscr{S}}^{\tor}_{K,\Sigma,Z}\simeq \mathfrak{X}_{Z,\Sigma_Z}/\Gamma_Z.
\]

Moreover, there exists a semi-abelian scheme $\cA$ over $\mathscr{S}^{\tor}_{K}$ endowed with an $\cO_F$-action 
and a Raynaud extension $\cG_Z$ over $C_{Z}$ also endowed with an $\cO_F$-action
such that we have a canonical $\cO_F$-equivariant isomorphism of the formal completion 
of $\cA$ with $\cG_Z$ over $\widehat{\mathscr{S}}^{\tor}_{K,\Sigma,Z}\simeq \mathfrak{X}_{Z,\Sigma_Z}/\Gamma_Z$. 
\item There is a stratification
\[
\mathscr{S}^{\tor}_{K,\Sigma}=\bigsqcup_{(Z,[\sigma])}\mathscr{S}^{\tor}_{K,\Sigma,(Z,[\sigma])},
\]
where $Z$ runs over cusp labels at level $K$ and $[\sigma]$ runs over $\Gamma_Z$-orbits in $\Sigma^+_Z$. If $(Z,[\sigma])$ and $(Z',[\sigma'])$
are two such pairs, then $\mathscr{S}^{\tor}_{K,\Sigma,(Z,[\sigma])}$ lies in the closure of 
$\mathscr{S}^{\tor}_{K,\Sigma,(Z',[\sigma'])}$ if and only if $Z\leq Z'$ and 
there are representatives $\sigma,\sigma'$ such that via the inclusion $M_Z\subseteq M_{Z'}$, $\sigma'$ is a face of $\sigma$. 
These incidence relations among strata are preserved under pullback to fibers.
\item Fix a representative $\sigma$ of an orbit $[\sigma]\in \Sigma^+_Z/\Gamma_Z$.
We have the relatively affine toroidal embedding 
\[
\Xi_Z\hookrightarrow \Xi(\sigma):=\underline{\Spec}_{\cO_{C_Z}}\bigoplus_{l\in \sigma^\vee}\Psi_Z(l).
\]
The scheme $\Xi(\sigma)$ has a closed subscheme $\Xi_{\sigma}$, which is 
defined by the ideal sheaf $\underline{\Spec}_{\cO_{C_Z}}\bigoplus_{l\in \sigma_0^\vee}\Psi_Z(l)$
(so naturally isomorphic to $\underline{\Spec}_{\cO_{C_Z}}\bigoplus_{l\in \sigma^{\perp}}\Psi_Z(l)$). Let $\mathfrak{X}_{\sigma}$ 
denote the formal completion of $\Xi(\sigma)$ along $\Xi_{\sigma}$ and $\widehat{\mathscr{S}}^{\tor}_{K,\Sigma,(Z,[\sigma])}$ denote 
the formal completion of $\mathscr{S}^{\tor}_{K,\Sigma}$ along $\mathscr{S}^{\tor}_{K,\Sigma,(Z,[\sigma])}$. Then there is a canonical isomorphism
\[
\widehat{\mathscr{S}}^{\tor}_{K,\Sigma,(Z,[\sigma])}\simeq \mathfrak{X}_{\sigma}.
\]
\item After inverting $N$, there is a projective morphism 
\[
\pi_{K,\Sigma}: \mathscr{S}^{\tor}_{K,\Sigma}[\tfrac 1N]\to \mathscr{S}^*_K
\]
which is the identity on the open subscheme $\mathscr{S}_K[\tfrac 1N]$. For each cusp label $Z$ we have 
\[
\pi_{K,\Sigma}^{-1}(\mathscr{S}_Z)= \mathscr{S}^{\tor}_{K,\Sigma,Z}[\tfrac 1N]
\]
set-theoretically and $\pi_{K,\Sigma*}\cO_{\mathscr{S}^{\tor}_{K,\Sigma}}[\tfrac 1N] = \cO_{\mathscr{S}^*_K}$.
This final equality holds on any fiber over $\Spec \Z[\tfrac 1{N\Delta_F}]$.
\end{enumerate}
\end{thm}

\begin{proof} Away from the primes dividing $N$, this summarizes results in~\cite[\S 6,\S 7]{lan-thesis}. 
See~\cite[Remark 2.1.8]{lan-stroh} for an explanation why $C_Z\to \mathscr{S}_Z$ is an abelian scheme in our particular case rather 
than an abelian scheme torsor over a finite \'etale cover of $\mathscr{S}_Z$.

More precisely, let us give a construction of the Raynaud extension in part (2). Let $r$ be the $\mathcal O_F$-rank of $X$. Then $\mathscr S_Z$ parametrizes on $S$-valued points quadruples $(B,\iota,\lambda,\eta)$ where $B$ is an abelian scheme of dimension $[F:\mathbb Q](n-r)$, $\iota: \mathcal O_F\to \mathrm{End}(B)$ is an $\mathcal O_F$-action making $\Lie B$ free of rank $n-r$ over $\mathcal O_F\otimes_{\mathbb Z} \mathcal O_S$, $\lambda: B\cong B^\vee$ is a principal polarization whose Rosati involution is compatible with complex conjugation on $\mathcal O_F$ via $\iota$, and $\eta: B[N]\cong \mathrm{Gr}^{\mathrm{Z}_N}_{-1}$ is an isomorphism that \'etale locally lifts to an isomorphism modulo $M$ for all $N|M$.

Let $T$ be the torus with character group $X$. Over $\mathscr S_Z$, the abelian scheme $C_Z\to \mathscr S_Z$ parametrizes the following data:
\begin{enumerate}
\item An $\mathcal O_F$-linear extension
\[
0\to T\to \mathcal G\to \mathcal B\to 0\ .
\]
\item An $\mathcal O_F$-linear splitting of the short exact sequence
\[
0\to T[N]\to \mathcal G[N]\to \mathcal B[N]\to 0\ .
\]
\end{enumerate}
We note that the first piece of data is equivalent to an $\mathcal O_F$-linear map $X\to \mathcal B^\vee$, and then the second piece of data is equivalent to lifting this to a map $\tfrac 1N X\to \mathcal  B^\vee$. As $\tfrac 1N X$ is a finite projective $\mathcal O_F$-module, one sees that
\[
C_Z = \underline{\mathrm{Hom}}_{\mathcal O_F}(\tfrac 1N X,\mathcal  B^\vee)
\]
defines an abelian scheme over $\mathscr S_Z$.

Identifying $\mathcal B^\vee$ with $\mathcal B$ via $\lambda$, we get a map $f_0: \tfrac 1N X\to \mathcal B$ over $C_Z$. The $\mathbb S_Z$-torsor $\Xi_Z\to \mathscr S_Z$ parametrizes ``polarizable'' lifts $f: \tfrac 1N X\to \mathcal G$ of $f_0$. Via the theory of degenerations of abelian varieties, one can after formal completion along the boundary of a corresponding torus compactification construct the quotient
\[
\mathcal A = \mathcal G/X
\]
over $\mathfrak X_{Z,\Sigma_Z}$, as a semiabelian scheme equipped with $\mathcal O_F$-action and a principal polarization away from the boundary. From the construction, $\mathcal A[N]$ admits, away from the boundary, a filtration with graded pieces $T[N]$, $\mathcal B[N]$ and $X/NX$. In fact, the construction makes this filtration split. The similar decomposition
\[
L/N\cong \mathrm{Gr}^{\mathrm{Z}_N}_{-2}\oplus \mathrm{Gr}^{\mathrm{Z}_N}_{-1}\oplus \mathrm{Gr}^{\mathrm{Z}_N}_0
\]
induces a level-$N$-structure on the generic fibre of $\mathcal A$.

At primes dividing $N$, these models have been constructed by Lan, \cite{lan-normalization}. Let $N=p^m N'$ and assume that still $N'\geq 3$. Then one can construct the desired toroidal compactification relatively over the toroidal compactification at level $N'$. On boundary charts, the extra data parametrized is an $\mathcal O_F$-linear splitting of the short exact sequence
\[
0\to T[p^m]\to \mathcal G[p^m]\to \mathcal B[p^m]\to 0
\]
(which is incorporated into the abelian scheme $C_Z\to \mathscr S_Z$) and a polarizable lift of $f_0: \tfrac 1N X\to \mathcal B$ to $f: \tfrac 1N X\to \mathcal G$. Both of these structures extend the corresponding structures present away from characteristic $p$. Away from the toroidal boundary, the given data induce an $\mathcal O_F$-linear map
\[
L/p^m\cong \mathrm{Gr}^{\mathrm{Z}_{p^m}}_{-2}\oplus \mathrm{Gr}^{\mathrm{Z}_{p^m}}_{-1}\oplus \mathrm{Gr}^{\mathrm{Z}_{p^m}}_0\to T[p^m]\oplus \cB[p^m]\oplus X/p^mX\to \mathcal A[p^m]
\]
giving the desired level-$p^m$-structure.
\end{proof}

In the following, we will fix a choice of $\Sigma$ as in Remark~\ref{trivial stabilizer} and leave its choice implicit; in particular, we will simply write $\mathscr S^\tor_K := \mathscr S^\tor_{K,\Sigma}$ for principal level $K=K(N)$.

\subsection{Perfectoid Shimura varieties}\label{sec:perfectoid}

In this section, we recall what is known about perfectoid Shimura varieties in our setup. Fix a prime $p$ that can actually be arbitrary for this section (i.e., even be ramified in $F$). We can take the adic spaces over $\Q_p$ associated with $\mathscr S_{K(N)}^*$ and $\mathscr S_{K(N)}^\tor$; let us denote these simply by a subscript $_{\Q_p}$. We would like to take the inverse limit over levels $K(p^m N)$ as $m\to\infty$. Unfortunately, inverse limits do not exist in the category of adic spaces, but they do in the category of diamonds. As we are mainly interested in \'etale cohomology, which is entirely functorial in diamonds, \cite{scholze-diamonds}, this is good enough for our purposes.

As above, we fix a cone decomposition $\Sigma$ as in Remark~\ref{trivial stabilizer}. The following definition uses that the inverse limits exist in the category of diamonds, see \cite[Lemma 11.22]{scholze-diamonds}.

\begin{defn}\label{def:perfectoid shimura} Let
\[\begin{aligned}
\mathcal S_{K(p^\infty N)}^\ast = \varprojlim_m \mathscr S_{K(p^m N),\Q_p}^{\ast,\diamondsuit},\\
\mathcal S_{K(p^\infty N)}^\tor = \varprojlim_m \mathscr S_{K(p^m N),\Q_p}^{\tor,\diamondsuit}
\end{aligned}\]
in the category of diamonds over $\Spd \Q_p$.
\end{defn}

As both are inverse limits of diamonds associated to qcqs analytic adic spaces, they are spatial diamonds. We have the following theorem. We will not need the part on toroidal compactifications in this paper, but we record it for reassurance.

\begin{thm} The diamonds $\mathcal S_{K(p^\infty N)}^\ast$ and, for a cofinal choice of cone decompositions $\Sigma$, $\mathcal S_{K(p^\infty N)}^\tor$ are representable by perfectoid spaces.
\end{thm}

\begin{proof} In \cite[Theorem 4.1.1]{scholze}, certain spaces $\mathscr S_{K(p^m N)}^{\underline{\ast}}$ finite under $\mathscr S_{K(p^m N)}^\ast$ were defined such that
\[
\mathcal S_{K(p^\infty N)}^{\underline{\ast}} = \varprojlim_m \mathscr S_{K(p^m N),\Q_p}^{\underline{\ast},\diamondsuit}
\]
is representable by a perfectoid space. Now \cite[Theorem 1.17 (1)]{bhatt-scholze-prisms} ensures that $\mathcal S_{K(p^\infty N)}^\ast$ is itself representable by a perfectoid space.

In the case of $\mathcal S_{K(p^\infty N)}^\tor$, the result follows from \cite[Th\'eor\`eme 0.4]{pilloni-stroh}, at least for cone decompositions that are compatible with cone decompositions on the Siegel moduli space, as the induced map of toroidal compactifications is a closed immersion at high enough level, cf. \cite{lantoroidalzariskiclosed}. (Here, we are using that Zariski closed is the same as strongly Zariski closed, cf.~\cite[Remark 7.5]{bhatt-scholze-prisms}, and that strongly Zariski closed subsets of perfectoid spaces are themselves perfectoid by \cite[Lemma 2.2.2]{scholze}.)
\end{proof}

We note that for any locally spatial diamond $D$ over $\Spd\ \Q_p$, there is a well-defined \'etale site $D_\et$ and a sheaf $\mathcal O_D^+/p$ on $D_\et$ whose pullback to any quasi-pro-\'etale perfectoid space $X\to D$ agrees with $\mathcal O_{X^\sharp}^+/p$, where $X^\sharp$ is the untilt of $X$ over $\mathbb Q_p$ corresponding to the map $X\to D\to \Spd\ \Q_p$. Indeed, it is clear that one can define a quasi-pro-\'etale sheaf this way, but this comes via pullback from the \'etale site when $X$ is a perfectoid space; thus, \cite[Theorem 14.12 (ii)]{scholze-diamonds} applies. If $D=X^\diamondsuit$ for an analytic adic space $X$ over $\Q_p$, then $D_\et\cong X_\et$ and $\mathcal O_D^+/p\cong \mathcal O_X^+/p$ naturally (as this can be checked after pullback to a pro-\'etale perfectoid cover). For a cofiltered inverse limit of spatial diamonds $D_i$ over $\Spd\ \Q_p$ along qcqs transition maps, the inverse limit $D=\varprojlim_i D_i$ is a locally spatial diamond by \cite[Lemma 11.22]{scholze-diamonds}, and denoting by $\pi_i: D\to D_i$ the projections, we have
\[
\mathcal O_D^+/p = \varinjlim_i \pi_i^\ast \mathcal O_{D_i}^+/p.
\]
This can be proved by identifying the stalks at geometric points of $D$; these factor over geometric points of $D_i$, and then one is reduced to the case that $D_i=\Spa(C_i,C_i^+)$ are all adic spectra of complete algebraically closed field $C_i$ over $\mathbb Q_p$ with open and bounded valuation subrings $C_i^+\subset C_i$. Then $D=\Spa(C,C^+)$ where $C^+$ is the $p$-adic completion of $\varinjlim_i C_i^+$ is of the same form, and indeed $C^+/p = \varinjlim_i C_i^+/p$.

We see in particular that we have such an identification of \'etale $\mathcal O^+/p$-sheaves in the situation of Definition~\ref{def:perfectoid shimura}.

\begin{remark}\label{strong limit} In \cite[Theorem 4.1.1]{scholze}, a stronger result is proved, namely that the map
\[
\varinjlim_m \mathcal O_{\mathscr S_{K(p^m N),\Q_p}^{\underline{\ast}}}\to \mathcal O_{\mathcal S_{K(p^\infty N)}^{\underline{\ast}}}
\]
has dense image when evaluated on some (explicit) affinoid cover of the limit. A similar result is known for the toroidal compactification by \cite[Th\'eor\`eme 0.4]{pilloni-stroh}, but for the minimal compactification it is still open (it was claimed by the second author in the original version of \cite[Lecture X]{berkeley-notes}, but the argument was incorrect).
\end{remark}

Inside $\mathscr S_{K(N),\Q_p}^\ast$, we have the locus $\mathscr S_{K(N),\Q_p}^\circ\subset \mathscr S_{K(N),\Q_p}$, contained in the adic space associated to the scheme $\mathscr S_{K(N)}\times \Spec \Q_p$, where the universal abelian variety has good reduction. This is a Hecke-equivariant quasicompact open subspace. If $p$ is a prime of good reduction, this can also be defined as the adic generic fibre of the $p$-adic completion $\mathscr S_{K(N),\Z_p}$ of $\mathscr S_{K(N)}$. We also let
\[
\mathcal S_{K(p^\infty N)}^\circ := \varprojlim_m \mathscr S_{K(p^m N),\Q_p}^{\circ,\diamondsuit}
\]
and note that it is itself a perfectoid space; this case follows directly from \cite[Theorem 4.1.1]{scholze} (and so we even have the stronger assertion from Remark~\ref{strong limit}).

We note that the cohomology of the good reduction locus captures the whole cohomology of the Shimura variety by results of Lan-Stroh, \cite{lan-stroh2}:\footnote{We will actually only use the result when $N$ is prime to $p$, in which case it is a standard consequence of the existence of a compactification with a relative normal crossing boundary divisor, here given by the toroidal compactification.}

\begin{prop}\label{prop:cohom good reduction} Let $C$ be a complete algebraically closed extension of $\Q_p$. For any $N\geq 3$ (not necessarily prime to $p$), the natural Hecke-equivariant map
\[
H^i(\mathscr S_{K(N),\overline{\Q}},\mathbb F_\ell)\to H^i(\mathscr S_{K(N),C}^\circ,\mathbb F_\ell)
\]
is an isomorphism.
\end{prop}

\begin{proof} We apply the results of Lan-Stroh, \cite{lan-stroh}, in the case (Nm) of normal integral models defined by normalization over the Siegel moduli problem; let $\mathfrak S_{K(N),\mathbb Z_p}$ be such a model. Then \cite[Corollary 5.20]{lan-stroh} shows that
\[
H^i(\mathscr S_{K(N),\overline{\Q}},\mathbb F_\ell)\cong H^i(\mathfrak S_{K(N),\overline{\mathbb F}_p},R\psi\mathbb F_\ell)
\]
using the nearby cycles $R\psi\mathbb F_\ell$. (This would be clear for proper varieties, and Lan-Stroh prove that even in case of bad reduction one can understand the boundary well enough to justify the assertion.) On the other hand the natural map
\[
H^i(\mathscr S_{K(N),C}^\circ,\mathbb F_\ell)\to H^i(\mathfrak S_{K(N),k},R\psi\mathbb F_\ell)
\]
is always an isomorphism, where $k$ is the residue field, by the comparison of nearby cycles in the algebraic and adic setting, \cite[Theorem 3.5.13]{huber}. We conclude by invariance of nearby cycles and cohomology under the extension of algebraically closed fields $k/\overline{\mathbb F}_p$.
\end{proof}

\subsection{The Hodge--Tate period morphism}\label{sec:hodgetate}

Let $\Fl$ be the adic space over $\mathrm{Spa}\ \Q_p$ associated to the flag variety $\mathrm{Fl}$ parametrizing totally isotropic $F$-linear subspaces of $V$.

\begin{thm}\label{thm:hodgetate} There exists a $G(\A_f)$-equivariant Hodge--Tate period morphism\footnote{The morphisms are equivariant for the natural action of $G(\Q_p)$ 
on $\Fl_{G,\mu}$ and for the trivial action of $G(\A_f^p)$ on $\Fl_{G,\mu}$.} 
fitting in the commutative diagram 
\[\xymatrix{
j: \mathcal{S}^\circ_{K(p^\infty N)} \ar@{^{(}->}[r] \ar[dr]_{\pi^\circ_{\HT}} & \mathcal{S}_{K(p^\infty N)}^* \ar[d]^{\pi_{\HT}^\ast} \\
& \Fl.}
\]
In particular, by projection to the minimal compactification we also get a Hodge--Tate period morphism
\[
\pi_\HT^\tor: \mathcal S_{K(p^\infty N)}^\tor\to \Fl.
\]
\end{thm}

\begin{proof} This relies on~\cite[Theorem 4.1.1]{scholze} and~\cite[Theorem 2.1.3]{caraiani-scholze}, though the precise statement regarding $\pi_{\HT}^\ast$ requires the extra argument in~\cite[Theorem 3.3.1]{arizona}.
\end{proof}

In \cite{caraiani-scholze}, we also described the fibers of the Hodge--Tate period morphism $\pi_\HT^\circ$ on the good reduction locus. More precisely, assume from now on again that $p$ is unramified in $F$ and let $C$ be a complete algebraically closed nonarchimedean field over $\mathbb Q_p$ with ring of integers $\cO_C$ and residue field $k$, and with a section $k\to \cO_C/p$. Take a point $x\in \Fl(C)$. By \cite[Theorem B]{scholze-weinstein}, the point $x$ corresponds to a $p$-divisible group $\mathbb X_{\cO_C}$ with $G$-structure over $\cO_C$ together with an isomorphism $T_p(\mathbb X_{\cO_C})\cong L\otimes_{\Z}\Z_p$ compatible with $G$-structures. Associated to the special fibre $\mathbb X_k$, we get the perfect Igusa variety
\[
\mathfrak{Ig}^{\mathbb X_k}
\]
defined above, where the level prime to $p$ is given by the part of $N$ prime to $p$. As it is perfect a scheme, it admits a canonical lift to $W(k)$, and thus to $\cO_C$; let us denote by subscript $_C$ the adic generic fibre, which is a perfectoid space over $C$.

\begin{thm}\label{thm:open hodge tate fibers} There is a canonical open immersion
\[
\mathfrak{Ig}^{\mathbb X_k}_C\hookrightarrow (\pi_\HT^\circ)^{-1}(x)
\]
whose image contains all points of rank $1$. In particular, for a prime $\ell\neq p$, there is a canonical Hecke-equivariant isomorphism
\[
(R(\pi_\HT^\circ)_\ast \mathbb F_\ell)_x\cong R\Gamma(\mathfrak{Ig}^{\mathbb X_k},\mathbb F_\ell).
\]
\end{thm}

\begin{proof} This is \cite[Theorem 4.4.4]{caraiani-scholze}.
\end{proof}

We also recall the Newton stratification of $\Fl$. Let $B(G_{\Q_p})$ be the Kottwitz set attached to the group $G_{\Q_p}$; see~\cite{kottwitz} and~\cite{rapoport-richartz} for more details on the Kottwitz set. This is equipped with a partial order called the Bruhat order, which we denote by $\leq$. Let $\mu^{-1}$ be a dominant representative of the inverse of $\mu$. Let $B(G_{\Q_p},\mu^{-1})$ denote the set of $\mu^{-1}$-admissible elements in $B(G_{\Q_p})$, cf.~\cite[Definition 3.1.2]{caraiani-scholze}.

For any $b\in B(G_{\Q_p},\mu^{-1})$, let $d_b$ be the dimension of the Igusa variety $\mathrm{Ig}^b$; this is known to be given explicitly by $\langle 2\rho,\nu_b\rangle$, cf.~e.g.~\cite{hamacher}. In particular, $d_{b'}\geq d_b$ whenever $b'\geq b$.

Recall from above that for a complete algebraically closed nonarchimedean extension $C/\Q_p$ and a point $x\in \Fl(C)$, we have a $p$-divisible group $\mathbb X_{\cO_C}$ with $G$-structure over $\cO_C$. In particular, the special fiber $\mathbb X_k$ defines a $p$-divisible group with $G$-structure, and this is classified up to isogeny by an element $b=b(x)\in B(G_{\Q_p},\mu^{-1})$. The following theorem asserts that this defines a reasonable stratification of $\Fl$:

\begin{thm}[\S 3 of~\cite{caraiani-scholze}] There exists a stratification 
\[
\Fl = \bigsqcup_{b\in B(G_{\Q_p},\mu^{-1})} \Fl^b
\]
with locally closed partially proper strata such that $x\in \Fl(C)$ lies in $\Fl^b(C)$ if and only if $b(x)=b$. The dimension of $\Fl^b$ (i.e., the Krull dimension of the locally spectral space $|\Fl^b|$) is given by $d-d_b$.

Moreover, the strata 
\[
\Fl^{\geq b} = \bigsqcup_{b'\geq b}\Fl^{b'}
\]
are closed.
\end{thm}

\noindent Since the reflex field of the Shimura datum is $\Q$, the largest element of $B(G_{\Q_p},\mu^{-1})$ is the ordinary one, cf.~\cite[Theorem 1.6.3]{wedhorn-thesis}. We have $\Fl^{\mathrm{ord}} = \Fl(\Q_p)$ (see, for example,~\cite[Proposition 3.3.8]{arizona}) and in particular this stratum is $0$-dimensional.

\subsection{The main argument}\label{subsec:mainargument}

Let us now give the proof of Theorem~\ref{thm:main}. It uses the following inputs. Assume again that $p$ is unramified in $F$ and fix some level $N\geq 3$ prime to $p$ and cone decomposition $\Sigma$ as in Remark~\ref{trivial stabilizer}. For any $b\in B(G_{\Q_p},\mu^{-1})$, we have the associated Igusa variety $\mathrm{Ig}^b = \mathrm{Ig}^{\mathbb X}_{K^p(N)}$ associated to some choice of completely slope divisible $\mathbb X$ in the isogeny class given by $b$. In Section~\ref{sec:compactifications of Igusa varieties}, we define a partial minimal compactification $j:\mathrm{Ig}^b\hookrightarrow \mathrm{Ig}^{b,\ast}$, and set
\[
H^i_{c-\partial}(\mathrm{Ig}^b,\mathbb F_\ell) = H^i(\mathrm{Ig}^{b,*},j_! \mathbb F_\ell).
\]

The first result we need is that $\mathrm{Ig}^{b,*}$ is affine:

\begin{thm}\label{igusa affine} The partial minimal compactification $\mathrm{Ig}^{b,*}$ is affine.
\end{thm}

\begin{proof} This is Lemma~\ref{basic properties of minimal compactifications}.
\end{proof}

In particular, we get the following result about cohomology:

\begin{prop}\label{concentration compactly supported cohomology} For any $\ell\neq p$, the cohomology group
\[
H^i_{c-\partial}(\mathrm{Ig}^b,\mathbb F_\ell)
\]
is nonzero only for $i\leq d_b=\mathrm{dim}\ \mathrm{Ig}^b$.
\end{prop}

\begin{proof} This is a direct consequence of Artin vanishing and the affineness of $\mathrm{Ig}^{b,\ast}$.
\end{proof}

The next result we need is the following semiperversity result.

\begin{thm}\label{semiperversity} Let $C$ be a complete algebraically closed extension of $\Q_p$ with ring of integers $\cO_C$ and residue field $k$. Consider the Hodge--Tate period map
\[
\pi_\HT^\circ: \mathcal S_{K(p^\infty N),C}^\circ\to \Fl_C .
\]
Any geometric point $\overline{x}$ of $\Fl_C$ has a cofinal collection of affinoid \'etale neighborhoods $U=\Spa(A)\to \Fl_C$ such that denoting $\mathfrak U=\Spf(A^\circ)$, the nearby cycles are semiperverse:
\[
R\psi(R(\pi_\HT^\circ)_\ast \mathbb F_\ell)|_U\in {}^p D^{\geq d}(\mathfrak{U}_k,\mathbb F_\ell).
\]
\end{thm}

\begin{proof} This is Theorem~\ref{thm:semiperversity}. It proceeds by first proving an analogous result for toroidal compactifications, for which one uses an identification of the fibers of the Hodge--Tate period map with partial toroidal compactifications of Igusa varieties.
\end{proof}

Now fix a maximal ideal $\mathfrak m\subset \mathbb T^S$ containing $\ell$, where the finite set $S$ of places of $\Q$ contains $\infty$ and all primes dividing $p\ell N\Delta_F$. Pick a $b$ with $d_b$ minimal such that $H^i(\mathrm{Ig}^b,\mathbb F_\ell)_{\mathfrak m}$ is nonzero; recall that $d_b$ is the dimension of $\mathrm{Ig}^b$. (If several $b$ achieve the same value of $d_b$, pick any of them.) The previous theorem implies the following result that proves a bound going in the other direction than Proposition~\ref{concentration compactly supported cohomology}:

\begin{lem}\label{concentration usual cohomology} If $H^i(\mathrm{Ig}^b,\mathbb F_\ell)_{\mathfrak m}\neq 0$ then $i\geq d_b$.
\end{lem}

\begin{rem} One could prove a variant of Theorem~\ref{semiperversity} for $R\pi_{\HT\ast}^\ast (j_! \mathbb F_\ell)$, where $j$ denotes the Zariski open immersion of the open Shimura variety into its minimal compactification, where one would get an object in ${}^p D^{\leq d}$ at the end. However, this estimate would only yield back Proposition~\ref{concentration compactly supported cohomology}.
\end{rem}

\begin{proof} The proof is identical to the proof of \cite[Corollary 6.1.4]{caraiani-scholze}, but we remark that in \cite{caraiani-scholze}, the authors wanted their perverse sheaves to be constructible, not realizing that the theory works well without constructibility. This explains the circumlocutions involving the quotient by $K_p$ in \cite{caraiani-scholze}.

Let us repeat the argument for the convenience of the reader. Consider the complex of \'etale sheaves $A=(R(\pi_\HT^\circ)_\ast \mathbb F_\ell)_{\mathfrak m}$ on $\Fl_C$. By assumption and Theorem~\ref{thm:open hodge tate fibers} (that also applies to higher rank points; in fact qcqs pushforwards of overconvergent sheaves are overconvergent), we know that it is concentrated on $\bigsqcup_{b',d_{b'}\geq d_b} \Fl^{b'}$, which is of dimension $\leq d-d_b$. In particular, for any formal model $\mathfrak{U}$ of any \'etale $U\to \Fl_C$, the nearby cycles $R\psi A|_U\in D(\mathfrak{U}_k,\mathbb F_\ell)$ are concentrated on a closed subscheme of dimension $\leq d-d_b$: The closure of a subset of the generic fibre of dimension $\delta$ is of dimension at most $\delta$, as the specialization map (from the generic fibre to the special fibre) is specializing. Choosing formal models as in Theorem~\ref{semiperversity}, we note that $R\psi A|_U\in {}^p D^{\geq d}(\mathfrak{U}_k,\mathbb F_\ell)$ as localization at $\mathfrak m$ is a filtered colimit and thus preserves ${}^p D^{\geq d}$. Together, we see that the stalks of $R\psi A|_U$ at all points of dimension $d-d_b$ are concentrated in degrees $\geq d_b$. Now choose a geometric rank $1$ point $\overline{x}$ of $\Fl_C$ of dimension $d-d_b$ and compute the stalk $A_{\overline{x}}$ as a filtered colimit of the stalks $(R\psi A|_U)_{\overline{y}}$ of $R\psi A|_U$ at its specialization $\overline{y}\in \mathfrak{U}_k$ (which will define a point of dimension $d-d_b$ for small enough $U$), over a cofinal system of affinoid \'etale neighborhoods $U$ of $\overline{x}$ as in Theorem~\ref{semiperversity}. This gives the desired result.
\end{proof}

We need two more results about the cohomology of Igusa varieties. These rely on the trace formula for which we have to assume that $F$ contains an imaginary quadratic field $F_0$, that $F^+\neq \mathbb Q$, and we have to fix a character $\varpi: A_{F_0}^\times/F_0^\times\to \C^\times$ such that $\varpi|_{\A^\times/\Q^\times}$ is the quadratic character corresponding to the extension $F_0$. We include all primes above which $\varpi$ is ramified into $S$. We assume moreover that the level $N\geq 3$ is only divisible by primes in $S\setminus \{p\}$, and is divisible by some specific sufficiently large $N_0$ with these properties, cf.~Remark~\ref{choice of N0}.

\begin{thm}\label{thm:igusa comp} Assume that $F^+\neq \Q$, that $p$ is split in the imaginary quadratic field $F_0\subset F$, and that $b\in B(G_{\Q_p},\mu^{-1})$ is such that
\[
H^i(\mathrm{Ig}^b,\mathbb F_\ell)_{\mathfrak m}
\]
is nonzero for exactly one $i$. Then there exists a continuous semisimple Galois representation
\[
\overline\rho_{\mathfrak m}: \mathrm{Gal}(\overline{F}/F)\to \GL_{2n}(\overline{\mathbb F}_\ell)
\]
such that for all primes $v$ dividing a rational prime $q\not\in S$ that splits in $F_0$, the characteristic polynomial of $\overline\rho_{\mathfrak m}(\Frob_v)$ is given by
\[
X^{2n} - T_{1,v}X^{2n-1}+\dots+ (-1)^iq_v^{i(i-1)/2}T_{i,v}X^{2n-i}+\dots + q_v^{n(2n-1)}T_{2n,v}
\]
with notation as in the introduction. Moreover, if $p$ is totally split in $F$ and $\overline\rho_{\mathfrak m}$ is unramified at all places $v$ dividing $p$, such that the Frobenius eigenvalues $\{\alpha_{1,v},\ldots,\alpha_{2n,v}\}$ of $\overline\rho_{\mathfrak m}(\Frob_v)$ satisfy $\alpha_{i,v}\neq p\alpha_{j,v}$ for all $i\neq j$, then $b$ is ordinary.
\end{thm}

\begin{proof} This is Corollary~\ref{cor:existence Galois single degree}. It is proved by computing $H^\ast(\mathrm{Ig}^b,\overline{\mathbb Q}_\ell)$, as a virtual representation, in terms of automorphic representations.
\end{proof}

To get concentration in one degree, we need to understand the cohomology of the boundary of Igusa varieties in order to play off upper bounds on $H_{c-\partial}^i$ with lower bounds on $H^i$. This is achieved in Section~\ref{pink} and gives the following result:

\begin{thm}\label{thm:boundary cohomology} Assume that $F^+\neq \Q$, that $p$ is split in the imaginary quadratic field $F_0\subset F$, and that the map
\[
H_{c-\partial}^i(\mathrm{Ig}^b,\mathbb F_\ell)_{\mathfrak m}\to H^i(\mathrm{Ig}^b,\mathbb F_\ell)_{\mathfrak m}
\]
is not an isomorphism for some $i$. Then there exists a continuous semisimple Galois representation
\[
\overline\rho_{\mathfrak m}: \mathrm{Gal}(\overline{F}/F)\to \GL_{2n}(\overline{\mathbb F}_\ell)
\]
such that for all primes $v$ dividing a rational prime $q\not\in S$ that splits in $F_0$, the characteristic polynomial of $\overline\rho_{\mathfrak m}(\Frob_v)$ is given by
\[
X^{2n} - T_{1,v}X^{2n-1}+\dots+ (-1)^iq_v^{i(i-1)/2}T_{i,v}X^{2n-i}+\dots + q_v^{n(2n-1)}T_{2n,v}.
\]
Moreover, if $b$ is not ordinary, then the length of $\overline\rho_{\mathfrak m}$ is at least $3$.
\end{thm}

\begin{proof} This is proved in Section~\ref{pink}.
\end{proof}

\begin{proof}[Proof of Theorem~\ref{thm:main}] Now assume all of our hypotheses: that $F = F_0\cdot F^+$ with $F^+\neq \mathbb{Q}$, that $p$ is totally split in $F$, and that $\mathfrak m$ is so that $\overline\rho_{\mathfrak m}$ is unramified and generic at all places dividing $p$, and of length at most $2$.

Pick $b$ as in Lemma~\ref{concentration usual cohomology}. If $b$ is not ordinary, then the map
\[
H_{c-\partial}^i(\mathrm{Ig}^b,\mathbb F_\ell)_{\mathfrak m}\to H^i(\mathrm{Ig}^b,\mathbb F_\ell)_{\mathfrak m}
\]
is an isomorphism by Theorem~\ref{thm:boundary cohomology}. Combining Lemma~\ref{concentration usual cohomology} with Proposition~\ref{concentration compactly supported cohomology}, we see that both sides are nonzero only for $i=d_b$. But then Theorem~\ref{thm:igusa comp} gives a contradiction unless $b$ is ordinary.

It follows that $H^i(\mathrm{Ig}^b,\mathbb F_\ell)_{\mathfrak m}$ can be nonzero only if $b$ is ordinary. By Lemma~\ref{concentration usual cohomology}, this shows that also in this case $H^i(\mathrm{Ig}^b,\mathbb F_\ell)_{\mathfrak m}\neq 0$ only for $i\geq d=d_b$. By Theorem~\ref{thm:open hodge tate fibers}, this shows that $(R(\pi_\HT^\circ)_\ast \mathbb F_\ell)_{\mathfrak m}$ is concentrated on the ordinary locus $\Fl(\Q_p)$ and in degrees $\geq d$. We see that
\[
H^i(\mathcal S_{K(p^\infty N),C}^\circ,\mathbb F_\ell)_{\mathfrak m}\cong H^i(\Fl,(R(\pi_\HT^\circ)_\ast \mathbb F_\ell)_{\mathfrak m})
\]
is concentrated in degrees $\geq d$. By a Hochschild-Serre spectral sequence, this implies that
\[
H^i(\mathscr S_{K(N),C}^\circ,\mathbb F_\ell)_{\mathfrak m}
\]
is concentrated in degrees $\geq d$. By Proposition~\ref{prop:cohom good reduction}, we see that
\[
H^i(S_{K(N)}(\C),\mathbb F_\ell)_{\mathfrak m}\cong H^i(\mathscr S_{K(N),\overline{\Q}},\mathbb F_\ell)_{\mathfrak m}\cong H^i(\mathscr S_{K(N),C}^\circ,\mathbb F_\ell)_{\mathfrak m}
\]
is concentrated in degrees $\geq d$, which is what we wanted to prove (by Lemma~\ref{lem:similitude factor} and Proposition~\ref{prop:complex unif}). 

Finally, the case of $H_c^i(X_K,\mathbb F_\ell)_{\mathfrak m}$ follows by Poincar\'e duality applied to the ``dual'' system of Hecke eigenvalues. More precisely, let $\iota: \mathbb{T}^{S}\to \mathbb{T}^{S}$ be the (anti-)involution 
that sends the double coset operator $[K^SgK^S]$ to $[K^Sg^{-1}K^S]$. Set $\mathfrak{m}^{\vee}:=\iota(\mathfrak{m})$. Then Poincar\'e duality, cf. e.g.~\cite[Proposition 2.2.12 and Corollary 2.2.13]{10Authors}, provides an isomorphism between $H_c^i(X_K,\mathbb F_\ell)_{\mathfrak m}$ and the dual of $H^{2d-i}(X_K,\mathbb F_\ell)_{\mathfrak m^{\vee}}$. We only need to check that the conditions of the theorem are also satisfied for $\m^{\vee}$. 

By explicitly computing the characteristic polynomial of $\overline{\rho}_{\mathfrak m^\vee}(\mathrm{Frob}_v)$ for any prime $v\not\in S$ of $F$ in terms of Hecke operators, we deduce the relationship 
\[
\overline\rho_{\mathfrak{m}^\vee}\toisom \overline{\rho}_{\mathfrak m}^\vee |\mathrm{Art}^{-1}_F|^{1-2n}, 
\]
where the global Artin reciprocity map $\mathrm{Art}_F$ is normalized to take uniformizers to geometric Frobenius elements. Note that $\overline\rho_{\mathfrak m^\vee}$ still has length at most $2$, and, when restricted to $\mathrm{Gal}(\overline{F}_{v}/ F_{v})$ for any prime $v|p$ of $F$, is still unramified. Moreover, if $\overline\rho_{\mathfrak m}$ has $\mathrm{Frob}_{v}$-eigenvalues equal to $\alpha_{i,v}$ for $i=1,\dots,2n$, then $\overline\rho_{\mathfrak m^\vee}$ has $\mathrm{Frob}_{v}$-eigenvalues equal to $p^{2n-1}\alpha^{-1}_{i,v}$ for $i=1,\dots,2n$. This implies that $\mathfrak{m}^\vee$ satisfies our genericity hypothesis at all $v|p$.  
\end{proof}

\newpage

\section{Compactifications of Igusa varieties}\label{sec:compactifications of Igusa varieties}

In this section, we construct the partial minimal and toroidal compactifications of Igusa varieties and establish their main geometric properties. Throughout this section, we fix: a prime $p$ unramified in $F$, an integer $N\geq 3$ prime to $p$, and a cone decomposition $\Sigma$ as in Remark~\ref{trivial stabilizer} at principal level $K(N)$. We abbreviate $K=K(N)$ and $ \mathscr S = \mathscr S_K = \mathscr S_{K(N)}$.

\subsection{Well-positioned subsets} In this section, we recall the notion of \emph{well-positioned} subsets introduced by~\cite{boxer, lan-stroh} and the fact that Oort central leaves are well-positioned subsets of the special fiber of $\mathscr{S}$. Fix an algebraically closed field $k$ of characteristic $p$. 

To state the condition of being well-positioned, we note that if $C$ is some complete algebraically closed nonarchimedean field, with ring of integers $\mathcal O_C$, and $x=(A,\iota,\lambda,\eta)\in \mathscr S(C)$, then the associated Raynaud extension
\[
0\to T\to \mathcal G\to \cB\to 0
\]
over $\mathcal O_C$ endows the abelian scheme $\cB$ over $\mathcal O_C$ with a principal polarization, an $\mathcal O_F$-action, and a level-$N$-structure, defining a point of $\mathscr S_Z(\mathcal O_C)$, where $Z$ is the cusp into which $x$ degenerates, and in particular a point $\pi(x)\in \mathscr S_Z(C)$.

\begin{defn}\label{well-positioned subset} A locally closed subset $Y\subseteq \mathscr{S}_k = \mathscr S\times k$ is well-positioned if there exists a family $Y^{\natural}=\{Y^{\natural}_Z\}$ indexed by the cusp labels $Z$ at level $K$ such that 
\begin{enumerate}
\item $Y^{\natural}_Z$ is a locally closed subset of $\mathscr{S}_{Z,k}$.
\item For any $C$ over $k$ and $x=(A,\iota,\lambda,\eta)\in \mathscr S_k(C)$ as above, degenerating into a cusp $Z$, the point $x$ lies in $Y$ if and only if the associated point $\pi(x)=(\cB_C,\ldots)\in \mathscr S_{Z,k}(C)$ lies in $Y^\natural_Z$.
\end{enumerate}
\end{defn}

\begin{remark} Let us verify that this agrees with \cite[Definition 2.2.1]{lan-stroh} in our case. In what follows, we will need to work with an open cover $\mathfrak{X}^\circ_{\sigma}$ of the formal scheme $\mathfrak{X}_{Z,\Sigma_Z}$, indexed by elements $\sigma\in \Sigma^+_Z$. The formal scheme $\mathfrak{X}^{\circ}_{\sigma}$ is obtained by taking the formal completion along a larger closed subscheme of $\Xi_{Z,\Sigma_Z}$ than $\Xi_{\sigma}$, namely
\[
\Xi(\sigma)^+:=\bigcup_{\tau\in \Sigma^+_Z,\bar{\tau}\subset\bar{\sigma}}\Xi_{\tau}.
\]
(The reason for considering $\Xi(\sigma)^+$ is the following: when $\bar{\tau}\subset \bar{\sigma}$ we have $\sigma^\vee\subset \tau^\vee$, which defines
an open embedding $\Xi(\tau)\subset \Xi(\sigma)$. The closed subscheme $\Xi(\sigma)^+$ is precisely the set-theoretic intersection of the boundary $\partial_{Z,\Sigma_Z}$ with $\Xi(\sigma)$.) 
See~\cite[Proposition 2.1.3]{lan-stroh} for properties of this formal scheme; the key property is that there is a canonical isomorphism
\begin{equation}\label{canonical isomorphism of formal completions}
\widehat{\mathscr{S}}^{\tor}_{\cup_{\tau\in \Sigma^+_Z,\bar{\tau}\subset\bar{\sigma}}(Z,[\tau])}\simeq\mathfrak{X}^\circ_{\sigma}.
\end{equation}
For any affine open formal subscheme $\Spf R$ of $\mathfrak{X}^\circ_{\sigma}$, we obtain canonical morphisms $W:=\Spec R \to \Xi(\sigma)$ and $W\to \mathscr{S}^{\tor}$ (induced by the isomorphism~\eqref{canonical isomorphism of formal completions}). Then the two stratifications of $W$ induced by the stratifications of $\Xi(\sigma)$ and $\mathscr{S}^{\tor}$ coincide. In particular, the preimages of $\mathscr{S}$ and $\Xi_Z$ coincide and we denote this open subscheme by $W^0$. Now \cite[Definition 2.1.1]{lan-stroh} asks that for any affine open $\Spf R\subset \mathfrak{X}^\circ_{\sigma}\subset \mathfrak{X}_{Z,\Sigma_Z}$ as above, we have an identification 
\[
Y\times_{\mathscr{S}} W^0=Y^{\natural}_Z\times_{\mathscr{S}_Z} W^0,
\]
where the morphism $W^0\to \mathscr{S}_Z$ is obtained by composing the induced morphism $W^0\to \Xi_Z$ with the canonical morphism $\Xi_Z\to \mathscr{S}_Z$. But this identification can be checked on points, where it amounts precisely to the condition we stated.

In particular, the definition is independent of the choice of $\Sigma$, cf.~also \cite[Lemma 2.2.2]{lan-stroh}. Remark 2.3.8 of~\cite{lan-stroh} shows that this notion is consistent (in our particular case, when there exist good integral models) with~\cite[Definition 3.4.1]{boxer}.
\end{remark}

A well-positioned subset admits partial toroidal and minimal compactifications that satisfy many nice properties. Let $Y\subset \mathscr{S}_k$ be a well-positioned subset. Let $\widetilde{Y}$ be the closure of $Y$ in $\mathscr{S}_k$, with complement $Y_0$. Let $\widetilde{Y}^*$ and $Y_0^*$ denote the closures of $Y$ and $Y_0$ in $\mathscr{S}^*_k$. Define 
\[
Y^*:= \widetilde{Y}^*\setminus Y^*_0
\]
and call this the \emph{partial minimal compactification} of $Y$. Define $Y^{\tor}$ analogously by
\[
Y^{\tor}:= \widetilde{Y}^{\tor}\setminus Y^{\tor}_{0}
\] 
and call it the \emph{partial toroidal compactification} of $Y$. See~\cite[Theorem 2.3.3]{lan-stroh} for the first basic properties of these partial compactifications. We have an identification 
\[
Y^*\times_{\mathscr{S}^\ast_k} \mathscr{S}_{Z,k} = Y^{\natural}_Z
\]
as subsets of $\mathscr{S}_{Z,k}$.

We need the following basic proposition.

\begin{prop}\label{closed immersion partial comp} Let $Y\subset Y'\subset \mathscr S_k$ be well-positioned locally closed subsets such that $Y$ is closed in $Y'$. Then $Y^*$ is a closed subset of $Y'^*$.
\end{prop}

\begin{proof} The maps $W^0\to \mathscr S_Z$, for varying $W^0$, form an fpqc cover, and using this one can check that $Y^\natural_Z$ must be a closed subset of $Y'^\natural_Z$ for all cusp labels $Z$. In particular, the last displayed formula now shows that $Y^\ast$ is a subset of $Y'^\ast$. On the other hand, $\widetilde{Y}^\ast$ is a closed subset of $\widetilde{Y'}^\ast$, and then
\[
Y^\ast = \widetilde{Y}^\ast\setminus Y_0^\ast\subset \widetilde{Y'}^\ast\setminus Y_0^\ast
\]
is a closed subset. As $Y'^*$ is a subset of $\widetilde{Y'}^\ast\setminus Y_0^\ast$ (as $Y_0\subset Y'_0$), the result follows.
\end{proof}

Now we want to apply these ideas to the case of central leaves. For this, we fix a $p$-divisible group $\mathbb X$ with $G$-structures over $k$.

\begin{prop}\label{prop:leaves are well-positioned}\leavevmode For any $p$-divisible group $\mathbb X$ with extra structure over $k$, the associated Oort central leaf $\mathscr{C}^{\mathbb X}\subset \mathscr{S}_k$ is well-positioned. For a cusp label $Z=(\mathrm Z_{N'},X)$, the subset $(\mathscr{C}^{\mathbb X})^{\natural}_{Z}$ is either the central leaf $\mathscr{C}^{\mathbb X_Z}_Z$ on $\mathscr S_{Z,k}$ associated with the unique $p$-divisible group $\mathbb X_Z$ with extra structure that admits a decomposition
\[
\mathbb X\cong \Hom(X,\mu_{p^\infty})\oplus \mathbb X_Z\oplus X\otimes (\mathbb Q_p/\mathbb Z_p)\ ,
\]
or empty if there is no such $\mathbb X_Z$.\footnote{The uniqueness follows from the discussion around Proposition~\ref{classification pdivisible}.}
\end{prop}

\begin{proof} This is~\cite[Proposition 3.4.2]{lan-stroh}, but let us give the proof. Choose $C$ as above and a point $(A,\iota,\lambda,\eta)\in \mathscr S_k(C)$. Let $0\to T\to \mathcal G\to \cB\to 0$ be the Raynaud extension and $X$ the cocharacter group of $T$.

We have two short exact sequences 
\[
0\to \cG_C[p^m]\to A_C[p^m]\to (X/p^mX)_C\to 0
\] 
and 
\[
0\to T_C[p^m]\to \cG_C[p^m]\to \cB_C[p^m]\to 0.
\]
These two short exact sequences give a canonical $\mathcal O_F$-linear filtration on $A_C[p^m]$ which is symplectic with respect to the Weil pairing. Passing to the direct limit over $m$, we get a filtration on $A_C[p^\infty]$ with graded pieces $\Hom(X,\mu_{p^\infty})$, $\cB_C[p^\infty]$, and $X\otimes (\mathbb Q_p/\mathbb Z_p)$.

By Proposition~\ref{classification pdivisible}, this implies that there is an isomorphism
\[
\mathbb X_C\cong \Hom(X,\mu_{p^\infty})\oplus \cB_C[p^\infty]\oplus X\otimes (\mathbb Q_p/\mathbb Z_p)
\]
and, as observed there, the corresponding $p$-divisible group $\cB_C[p^\infty]\cong (\mathbb X_Z)_C$ is then unique up to isomorphism.
\end{proof}

\noindent We let $\mathscr{C}^{\mathbb X,\tor}$ and $\mathscr{C}^{\mathbb X,*}$ denote the partial toroidal and minimal compactifications of the leaf $\mathscr{C}^{\mathbb X}$. Let $\mathscr{C}^{\mathbb X}_Z:=\mathscr{C}^{\mathbb X,*}\times_{\mathscr{S}^*}\mathscr{S}_Z$. By Proposition~\ref{prop:leaves are well-positioned}
and the remark above it, we can identify $\mathscr{C}^{\mathbb X}_Z$ with $\mathscr{C}^{\mathbb X_Z}_Z$. 

\begin{lem}\label{toroidal compactifications of leaves are smooth} The partial toroidal compactification $\mathscr{C}^{\mathbb X,\tor}$ is a smooth variety. 
\end{lem}

\begin{proof} Since we are working over the perfect field $k$, it is enough to show that $\mathscr{C}^{\mathbb X,\tor}$ is a regular scheme. This follows from the fact that $\mathscr{C}^{\mathbb X}$ is smooth, thus regular, and from~\cite[Proposition 2.3.13]{lan-stroh}.
\end{proof}

\subsection{Partial toroidal compactifications of Igusa varieties}
We continue in the same setup, in particular $\mathbb X$ is a $p$-divisible group with $G$-structure over $k$, as before. Let $\cA$ denote the restriction of the semi-abelian scheme over $\mathscr{S}^{\tor}$ to $\mathscr{C}^{\mathbb X,\tor}$. Note that the group schemes $\cA[p^m]$ for $m\in \Z_{\geq 1}$ are quasi-finite and flat, but not finite and flat. Therefore, the inductive system $\cA[p^\infty]$ is not a $p$-divisible group. Nonetheless, we show below that the connected part $\cA[p^\infty]^\circ$ of $\cA[p^\infty]$ (which can be defined as the ind-scheme $\hat{\cA}[p^\infty]$, where $\hat{\cA}$ is the completion of $\cA$ along its identity section) is a $p$-divisible group. 

\begin{prop}\label{connected part extends} The connected part $\cA[p^\infty]^\circ$ of $\cA[p^\infty]$ is a $p$-divisible group over $\mathscr{C}^{\mathbb X,\tor}$.
\end{prop}

\begin{proof} We can check this on the completed strict local rings of $\mathscr C^{\mathbb X,\tor}$; more precisely, we restrict to $\Spec R$, where $R$ is the completed strict local ring of $\mathscr{C}^{\mathbb X,\tor}$ at some point $x$ that lies in a boundary stratum indexed by a cusp label $Z$.

In that case, there is the Raynaud extension
\[0\to T\to \cG\to \mathcal B\to 0
\]
over $\Spec R$ and a map $\cG\to \cA$ over $\Spf R$. On formal completions at the identity, this gives an isomorphism $\hat{\cG}\cong \hat{\cA}$ that is in fact defined over $\Spec R$ (as modulo any power of the augmentation ideal, both schemes are finite over $\Spf R$, so one can apply formal GAGA). This induces an isomorphism $\cG[p^\infty]^\circ\cong \cA[p^\infty]^\circ$ over $\Spec R$. But $\cG[p^\infty]$ is a $p$-divisible group, and its connected part has constant rank, as the abelian variety $\mathcal B$ is given by a map from $\Spec R$ into the leaf $\mathscr C^{\mathbb X_Z}_Z$. Thus, $\cG[p^\infty]^\circ$ is also a $p$-divisible group, as desired.
\end{proof}

Clearly, $\cA[p^\infty]^\circ$ carries an $\mathcal O_F$-action. At every geometric point $\bar x$ of $\mathscr{C}^{\mathbb X,\tor}$, there is an $\mathcal O_F$-linear isomorphism
\[
\cA[p^\infty]^\circ\times_{\mathscr C^{\mathbb X,\tor}} k(\bar x)\cong \mathbb X^\circ\times_{k} k(\bar x) .
\]
Indeed, in the notation of the previous proof, this follows directly from $\mathcal B[p^\infty]\times_k k(\bar x)\cong \mathbb X_Z$ as
\[
\mathbb X^\circ\cong \Hom(X,\mu_{p^\infty})\oplus \mathbb X_Z^\circ .
\]
In particular, we see that the multiplicative part $\cA[p^\infty]^\mu$ is of constant rank, and thus defines a $p$-divisible group itself. Let
\[
\cA[p^\infty]^{(0,1)} = \cA[p^\infty]^\circ / \cA[p^\infty]^\mu
\]
be the biconnected part.

\begin{prop} There exists a polarization on $\cA[p^\infty]^{(0,1)}$ extending the polarization that exists after restriction to $\mathscr C^{\mathbb X}$.
\end{prop}

\begin{proof} Again, this can be checked after restriction to $\Spec R$ for the completed local rings of $\mathscr C^{\mathbb X,\tor}$. With notation as before, this identifies $\cA[p^\infty]^{(0,1)}$ with the biconnected part of $\mathcal B[p^\infty]$, which has a natural principal polarization, by construction compatible with the polarization on the generic fibre of $\cA$.
\end{proof}

\subsubsection{Partial toroidal compactifications of Igusa varieties in the completely slope divisible case}
Assume now that $\mathbb X=\bigoplus_i \mathbb X_i$ is completely slope divisible. Let $\mathrm{Ig}^{\mathbb X}_m$ be the Igusa variety over $\mathscr{C}^{\mathbb X}$ of some finite level $p^m$ as defined in Remark~\ref{rem:finite level Igusa}; recall that this is a finite \'etale cover of $\mathscr{C}^{\mathbb X}$, which is Galois with Galois group $\Gamma_{m,\mathbb X}$.

\begin{thm}\label{toroidal igusa is etale} The finite \'etale cover $\mathrm{Ig}^{\mathbb X}_m\to \mathscr{C}^{\mathbb X}$ extends uniquely to a finite \'etale cover $\mathrm{Ig}^{\mathbb X,\tor}_m\to \mathscr{C}^{\mathbb X,\tor}$, Galois with group $\Gamma_{m,\mathbb X}$.
\end{thm}

In fact, we can be more precise, and give a moduli description of $\mathrm{Ig}^{\mathbb X,\tor}_m$. Note that $\cA[p^\infty]^\circ$ is also completely slope divisible, so we obtain $p$-divisible groups $\cA[p^\infty]_i$ equipped with extra structures of EL or of PEL type for all $i$ such that $\lambda_i>0$. 

\begin{defn}\label{defn:toroidal Igusa level structure} An Igusa level $p^m$ structure $\rho_m$ on a 
$\mathscr{C}^{\mathbb X,\tor}$-scheme $\mathscr{T}$ consists of the following data:

For each $i$ such that $\lambda_i>0$, an isomorphism
\[
\rho_{i, m}: \cA[p^\infty]_i[p^m]\times_{\mathscr{C}^{\mathbb X,\tor}}\mathscr{T} \toisom \mathbb{X}_i[p^m]\times_k\mathscr{T},
\]
that commutes with the $\cO_{F}$-action and lifts fppf locally to $p^{m'}$-torsion for all $m'\geq m$, and an element of $(\Z/p^m\Z)^\times(\mathscr T)$ such that, for all $i,j$ such that $\lambda_i,\lambda_j>0$ and $\lambda_i+\lambda_j=1$, the isomorphisms $\rho_{i,m}$ and $\rho_{j,m}$ commute with the polarizations up the given scalar in $(\Z/p^m\Z)^\times(\mathscr T)$.
\end{defn}

\begin{proof}[Proof of Theorem~\ref{toroidal igusa is etale}] We prove that the moduli problem in Definition~\ref{defn:toroidal Igusa level structure} is representable
by a finite \'etale scheme over $\mathscr{C}^{\mathbb X,\tor}$. 
This follows from Theorem~\ref{thm:construction of Igusa covers} applied to the isoclinic $p$-divisible groups $\cA[p^\infty]_i$ for all $i$ with $\lambda_i>\frac 12$ as $p$-divisible groups with EL structure, and the $p$-divisible group $\cA[p^\infty]_i$ for $i$ with $\lambda = \frac 12$ as $p$-divisible group with PEL structure. Indeed, the isomorphisms for $i$ with $0<\lambda_i<\frac 12$ are then formally determined. We now apply Theorem~\ref{thm:construction of Igusa covers} to each graded piece with its induced extra structures and take the fiber product of the resulting finite \'etale covers of $\mathscr{C}^{\mathbb X,\tor}$.

It is clear that over the open part, this recovers $\mathrm{Ig}^{\mathbb X}_m$, again as the polarization determines the structure on the \'etale quotient from the structure on the multiplicative quotient. In particular, it defines a $\Gamma_{m,\mathbb X}$-torsor, as desired.
\end{proof}

For further use, we would like to give a preliminary description of the formal completions 
along boundary strata in $\mathrm{Ig}^{\mathbb X,\tor}_m$ in terms of the toroidal
boundary charts in Theorem~\ref{existence of toroidal compactifications}.\footnote{A different
description will be given in Theorem~\ref{thm:tor Igusa strata}.}
Let $Z$ be a cusp label at level $N$. This determines a locally closed boundary stratum 
$\mathscr{C}^{\mathbb X,\tor}_Z \subset\mathscr{C}^{\mathbb X,\tor}$. Let $\widehat{\mathscr{C}}^{\mathbb X,\tor}_Z$ denote the completion along this boundary stratum. Let $\widehat{\mathrm{Ig}}^{\mathbb X,\tor}_{m,Z}$
denote the formal completion of $\mathrm{Ig}^{\mathbb X,\tor}_m$ along
the boundary stratum $\mathrm{Ig}^{\mathbb X,\tor}_{m,Z}$ determined by the
preimage of $\mathscr{C}^{\mathbb X,\tor}_Z$.  We have a finite \'etale map of
formal schemes
\[
\widehat{\mathrm{Ig}}^{\mathbb X,\tor}_{m,Z}\to \widehat{\mathscr{C}}^{\mathbb X,\tor}_Z,
\]
which is Galois, with Galois group $\Gamma_{m,\mathbb X}$. 
 
Because $\mathscr{C}^{\mathbb X}$ is well-positioned, we can identify $\widehat{\mathscr{C}}^{\mathbb X,\tor}_Z$ with the quotient by $\Gamma_Z$ of the formal completion of 
\[
\mathscr{C}^{\mathbb X}_{Z,\Sigma_Z}:=\Xi_{Z,\Sigma_Z}\times_{\mathscr{S}_Z} \mathscr{C}^{\mathbb X}_Z
\] 
along $\partial_{Z,\Sigma_Z}\times_{\mathscr{S}_Z} \mathscr{C}^{\mathbb X}_Z$, cf.~\cite[Theorem 2.3.2 (5)]{lan-stroh}. We would like to give a similar description for $\widehat{\mathrm{Ig}}^{\mathbb X,\tor}_{m,Z}$. By construction, over $C_Z$, we have a semi-abelian scheme $\cG_Z$ together with an $\cO_F$-action. Here $\cG_Z$ is a Raynaud extension 
\[
0\to T \to \cG_Z\to \cB_{Z}\to 0,
\] 
where $T$ is the constant torus with character group $X$ (recall that $X$ is part of the torus argument that comes with the cusp label $Z$) and $\cB_{Z}$ is the pullback of the universal abelian scheme over $\mathscr{S}_Z$ to $C_Z$. As observed in the proof of Proposition~\ref{connected part extends}, over $C_Z\times_{\mathscr S_Z} \mathscr C^{\mathbb X}_Z$, the connected part $\cG_Z[p^\infty]^\circ$ defines a $p$-divisible group itself (as the \'etale part has constant rank); let us write $\mathscr H_Z$ for this $p$-divisible group with $\cO_F$-action on $C_Z\times_{\mathscr S_Z} \mathscr C^{\mathbb X}_Z$. Again, it is completely slope divisible, so we get $p$-divisible groups $\mathscr H_{Z,i}$ of slope $\lambda_i$, equipped with polarizations as above.

With the obvious version of Definition~\ref{defn:toroidal Igusa level structure}, we can now define a $\Gamma_{m,\mathbb X}$-torsor over $C_Z\times_{\mathscr S_Z} \mathscr C^{\mathbb X}_Z$, whose pullback to $\Xi_{Z,\Sigma_Z}\times_{\mathscr S_Z} C^{\mathbb X}_Z$ we denote by
\[
\mathrm{Ig}^{\mathbb X}_{Z,\Sigma_Z}\to \Xi_{Z,\Sigma_Z}.
\]
We also let $\partial\mathrm{Ig}^{\mathbb X}_{Z,\Sigma_Z}$ denote the corresponding finite \'etale
cover of $\partial_{Z,\Sigma_Z}\times_{\mathscr{S}_Z}\mathscr{C}^{\mathbb X}_{Z}$. This is 
a closed subscheme of $\mathrm{Ig}^{\mathbb X}_{Z,\Sigma_Z}$. Finally, we let $Y_{Z,\Sigma_Z}$ denote the formal completion of $\mathrm{Ig}^{\mathbb X}_{Z,\Sigma_Z}$ along $\partial\mathrm{Ig}^{\mathbb X}_{Z,\Sigma_Z}$. 

\begin{thm}\label{thm:formal completions Igusa} With the same choice of a splitting of the filtration $\mathrm{Z}_N$ as in Theorem~\ref{existence of toroidal compactifications}~(3), there is a canonical isomorphism of formal schemes 
\[
\widehat{\mathrm{Ig}}^{\mathbb X,\tor}_{m,Z}\toisom Y_{Z,\Sigma_Z}/\Gamma_Z.
\]
\end{thm}

\begin{proof} This follows once we prove that there exists an isomorphism between the pullbacks of $\cA[p^\infty]^\circ$ and $\mathscr{H}_Z$ to $\widehat{\mathscr{C}}^{\mathbb X,\tor}_Z$ that commutes with the extra structures.

We work Zariski locally on $\widehat{\mathscr{C}}^{\mathbb X,\tor}_Z$. We cover $\widehat{\mathscr{C}}^{\mathbb X,\tor}_Z$ by affine open formal subschemes $\Spf R$ which lift to affine opens in the formal completion of $\mathscr{C}^{\mathbb X}_{Z,\Sigma_Z}$ along the toroidal boundary stratum (because taking the quotient by $\Gamma_Z$ is a local isomorphism) and which arise by formal completion from affine opens in $\mathscr{C}^{\mathbb X,\tor}$ and $\mathscr{C}^{\mathbb X}_{Z,\Sigma_Z}$. We obtain induced flat maps of schemes
\[
f_1:\Spec R\to \mathscr{C}^{\mathbb X,\tor}\ \mathrm{and}\ f_2:\Spec R\to \mathscr{C}^{\mathbb X}_{Z,\Sigma_Z}. 
\]
We let $A$ denote the pullback of the semi-abelian scheme $\cA$ along $f_1$ and $\widetilde{A}$ denote the pullback of the Raynaud extension $\widetilde{G}_Z$ along $f_2$. As above, the completions along the identity $\widehat{A}\cong \widehat{\widetilde{A}}$ are isomorphic over $\Spf R$, and thus over $\Spec R$.

Now
\[
f_2^\ast \mathscr H_Z= \widetilde{A}[p^\infty]^\circ\cong A[p^\infty]^\circ = f_1^\ast \cA[p^\infty]^\circ\ ,
\]
as desired.
\end{proof}

\subsubsection{Partial toroidal compactifications of perfect Igusa varieties}
In this subsection, we will repeat the previous constructions for perfect Igusa varieties, allowing now general $\mathbb X$ (not necessarily completely slope divisible).

Let $\mathfrak{Ig}^{\mathbb X}$ be the perfect Igusa variety over $\mathscr{C}^{\mathbb X}$; this is a pro-finite \'etale cover of the perfection of $\mathscr{C}^{\mathbb X}$, Galois with group $\Gamma_{\mathbb X} = \mathrm{Aut}(\mathbb X)$.

\begin{thm}\label{perfect toroidal igusa is etale} The pro-finite \'etale cover $\mathfrak{Ig}^{\mathbb X}\to \mathscr{C}^{\mathbb X}_\perf$ extends uniquely to a pro-finite \'etale cover $\mathfrak{Ig}^{\mathbb X,\tor}\to \mathscr{C}^{\mathbb X,\tor}_\perf$, Galois with group $\Gamma_{\mathbb X}$.
\end{thm}

We remark that if $\mathbb X$ is completely slope divisible, then $\mathfrak{Ig}^{\mathbb X,\tor}$ is the perfection of $\varprojlim_m \mathrm{Ig}^{\mathbb X,\tor}_m$ (by uniqueness).

Again, we can be more precise, and give a moduli description of $\mathfrak{Ig}^{\mathbb X,\tor}$:

\begin{defn}\label{defn:perfect toroidal Igusa level structure} A perfect Igusa level structure $\rho$ on a perfect
$\mathscr{C}^{b,\tor}_{\Sigma}$-scheme $\mathscr{T}$ is given by an $\cO_F$-linear isomorphism
\[
\rho:\cA[p^\infty]^\circ\times_{\mathscr{C}^{\mathbb X,\tor}}\mathscr{T} \toisom \mathbb{X}^\circ\times_k\mathscr{T}
\]
and a scalar in $\Z_p^\times(\mathscr T)$ such that the induced isomorphism
\[
\rho_{(0,1)}:\cA[p^\infty]^{(0,1)}\times_{\mathscr{C}^{\mathbb X,\tor}}\mathscr{T} \toisom \mathbb{X}^{(0,1)}\times_k\mathscr{T}
\]
obtained by quotienting by the multiplicative parts commutes with the polarizations up to the given element of $\Z_p^\times(\mathscr{T})$.
\end{defn}

\begin{proof}[Proof of Theorem~\ref{perfect toroidal igusa is etale}] Uniqueness is clear as the base is normal and the cover is given on an open dense subspace. The proof is now similar to the proof of Theorem~\ref{toroidal igusa is etale}, except that we refer to Proposition~\ref{perfect Igusa} instead of Theorem~\ref{thm:construction of Igusa covers}.
\end{proof}

\begin{remark}\label{rem:not torsor} By Corollary~\ref{cor:fpqc torsor}, $\mathfrak{Ig}^{\mathbb X}$ is an 
fpqc $\underline{\mathrm{Aut}}(\mathbb{X})$-torsor over $\mathscr{C}^{\mathbb X}$ (before perfection). However, this no longer holds true over the boundary strata of the partial toroidal compactification. This phenomenon 
already occurs for the modular curves, in which case the precise structure at the boundary is described in~\cite[\S 3]{howe}. 
\end{remark}

Again, we can give a description of the formal completions along boundary strata. More precisely, we would like to describe the $\Gamma_{\mathbb X}$-torsor
\[
\widehat{\mathfrak{Ig}}^{\mathbb X,\tor}_Z\to \widehat{\mathscr C}^{\mathbb X,\tor}_{Z,\perf}
\]
with notation following the previous subsection.
 
We can identify $\widehat{\mathscr{C}}^{\mathbb X,\tor}_{Z,\perf}$ with the quotient by $\Gamma_Z$ of the formal completion of
\[
\mathscr{C}^{\mathbb X}_{Z,\Sigma_Z,\perf}=(\Xi_{Z,\Sigma_Z}\times_{\mathscr{S}_Z} \mathscr{C}^{\mathbb X}_Z)_\perf
\] 
along its toroidal boundary. We would like to give a similar description for $\widehat{\mathfrak{Ig}}^{\mathbb X,\tor}_Z$. 

As above, we have the $p$-divisible group $\mathscr{H}_Z$ over $C_Z\times_{\mathscr S_Z} \mathscr C^{\mathbb X}_Z$ (the connected part of the $p$-divisible group of the Raynaud extension). On perfect test schemes over $C_Z\times_{\mathscr S_Z} \mathscr C^{\mathbb X}_Z$, we can parametrize $\cO_F$-linear isomorphisms
\[
\mathscr H_Z\cong \mathbb X^\circ
\]
together with a section of $\Z_p^\times$ such that the induced isomorphism of biconnected parts is compatible with the polarization. By Proposition~\ref{perfect Igusa}, this is representable by a $\Gamma_{\mathbb X}$-torsor
\[
C^{\mathfrak{Ig},\mathbb X}_Z\to (C_Z\times_{\mathscr S_Z} \mathscr C^{\mathbb X}_Z)_\perf.
\]
Let us give an explicit description of $C^{\mathfrak{Ig},\mathbb X}_Z$.

\begin{prop}\label{base of perfect torus torsor} The perfect scheme $C^{\mathfrak{Ig},\mathbb X}_Z$ parametrizes points $(\cB,\iota,\lambda,\eta)\in \mathscr S_Z$ together with an extension
\[
0\to T\to \cG\to \cB\to 0
\]
by the split torus $T$ with cocharacter group $X$, and an $\cO_F$-linear embedding
\[
\rho: \cG[p^\infty]\hookrightarrow \mathbb X
\]
that is compatible with the polarization in the following sense: The filtration $\mathrm Z_{\mathbb X}$ of $\mathbb X$ given by
\[
T[p^\infty]\subset \cG[p^\infty]\subset \mathbb X
\]
is symplectic, and the induced isomorphism
\[
\cB[p^\infty] = \cG[p^\infty]/T[p^\infty]\cong \mathrm{Gr}^{\mathrm Z_{\mathbb X}}_{-1}
\]
is compatible with principal polarizations.
\end{prop}

\begin{proof} The description given above would only require an isomorphism $\cG[p^\infty]^\circ=\mathscr H_Z\cong \mathbb X^\circ$. However, over a perfect base the \'etale part is a direct summand, which by duality with the multiplicative part has a natural map into the \'etale part of $\mathbb X$, inducing a canonical extension $\cG[p^\infty]\hookrightarrow \mathbb X$. The result follows easily.
\end{proof}

Now define $\Xi^{\mathfrak{Ig},\mathbb X}_{Z,\Sigma_Z}$ by the cartesian square
\[\xymatrix{
\Xi^{\mathfrak{Ig},\mathbb X}_{Z,\Sigma_Z}\ar[r]\ar[d] & (\Xi_{Z,\Sigma_Z}\times_{\mathscr{S}_Z} \mathscr{C}^{\mathbb X}_Z)_\perf\ar[d]\\
C^{\mathfrak{Ig},\mathbb X}_Z\ar[r] & (C_Z\times_{\mathscr S_Z} \mathscr C^{\mathbb X}_Z)_\perf.
}\]
We would also like to understand this explicitly. The essential point is to understand the $\mathbb S_{Z,\perf}$-torsor
\[
\Xi^{\mathfrak{Ig},\mathbb X}_{Z} \to C^{\mathfrak{Ig},\mathbb X}_Z
\]
of which $\Xi^{\mathfrak{Ig},\mathbb X}_{Z,\Sigma_Z}$ is a partial compactification.

\begin{prop}\label{perfect torus torsor} The $\mathbb S_{Z,\perf}$-torsor
\[
\Xi^{\mathfrak{Ig},\mathbb X}_{Z} \to C^{\mathfrak{Ig},\mathbb X}_Z
\]
is the perfection of the $\mathbb S_Z$-torsor that parametrizes lifts of the map $f_0: X\to \cB^\vee\cong \cB$ (given by the extension $0\to T\to \cG\to \cB\to 0$) to symmetric $f: X\to \cG$.

Explicitly, this can be described as follows. Choose a symplectic splitting $\delta_{\mathbb X}$ of the filtration $\mathrm{Z}_{\mathbb X}$ of $\mathbb X$ over $C^{\mathfrak{Ig},\mathbb X}_Z$. Via the embedding $\rho: \cG[p^\infty]\hookrightarrow \mathbb X$, this induces a splitting of
\[
0\to T[p^\infty]\to \cG[p^\infty]\to \cB[p^\infty]\to 0
\]
and dually a lift of $f_0$ to a map $\tilde{f}_0: X[\tfrac 1p]\to \cB^\vee\cong \cB$. Then the $\mathbb S_{Z,\perf}$-torsor
\[
\Xi^{\mathfrak{Ig},\mathbb X}_{Z} \to C^{\mathfrak{Ig},\mathbb X}_Z
\]
parametrizes symmetric lifts $\tilde{f}: X[\tfrac 1p]\to \cG$ of $\tilde{f}_0$.
\end{prop}

\begin{proof} The first description follows directly from the definition. For the second description, it is enough to observe that it does indeed define a lift of the given $\mathbb S_Z$-torsor to a $\mathbb S_{Z,\perf}$-torsor; by uniqueness, it must be the correct one.
\end{proof}

Let $\mathfrak{Y}_{Z,\Sigma_Z}$ be the completion of $\Xi^{\mathfrak{Ig},\mathbb X}_{Z,\Sigma_Z}$ along its toroidal boundary.

\begin{thm}\label{thm:formal completions perfect Igusa} With the same choice of a splitting of the filtration $\mathrm{Z}_N$ as in Theorem~\ref{existence of toroidal compactifications}~(3), there is a canonical isomorphism of formal schemes 
\[
\widehat{\mathfrak{Ig}}^{\mathbb X,\tor}_Z\toisom \mathfrak{Y}_{Z,\Sigma_Z}/\Gamma_Z.
\]
\end{thm}

\begin{proof} The proof is the same as for Theorem~\ref{thm:formal completions Igusa}.
\end{proof}

A corollary of this description is that $\mathfrak{Ig}^{\mathbb X,\tor}_Z$ only depends on $\mathbb X$ up to quasi-isogenies inducing an isomorphism on \'etale and multiplicative parts.

\begin{cor}\label{cor:correspondence igusa} Let $\phi: \mathbb X\to \mathbb X'$ be an isogeny between $p$-divisible groups with $G$-structure over $k$. Assume that $\phi$ induces an isomorphism on \'etale and multiplicative parts. Then the isomorphism
\[
\mathfrak{Ig}^{\mathbb X}\cong \mathfrak{Ig}^{\mathbb X'}
\]
induced by $\phi$ extends uniquely to an isomorphism
\[
\mathfrak{Ig}^{\mathbb X,\tor}\cong \mathfrak{Ig}^{\mathbb X',\tor}\ .
\]
\end{cor}

\begin{proof} For each cusp $Z$, we will produce an isomorphism
\[
\widehat{\mathfrak{Ig}}^{\mathbb{X},\tor}_{K^p,\Sigma,Z}\cong \widehat{\mathfrak{Ig}}^{\mathbb{X}',\tor}_{K^p,\Sigma,Z}.
\]
For this, we use the description in Theorem~\ref{thm:formal completions perfect Igusa}. First, consider the $\Gamma_{\mathbb X}$-torsor $C^{\mathfrak{Ig},\mathbb X}_Z\to (C_Z\times_{\mathscr S_Z} \mathscr C^{\mathbb X}_Z)_\perf$. We use the description in Proposition~\ref{base of perfect torus torsor}. We have given the Raynaud extension
\[
0\to T\to \cG\to \cB\to 0
\]
together with an embedding
\[
\rho: \cG[p^\infty]\hookrightarrow \mathbb X,
\]
and the isogeny $\phi: \mathbb X\to \mathbb X'$ whose kernel $\cK\subset \mathbb X$ is contained in $\mathbb X^\circ$ and does not meet the multiplicative part. It follows that $\cG'=\cG/\rho^{-1}(\cK)$ defines another semiabelian scheme
\[
0\to T\to \cG'\to \cB'\to 0
\]
(where $\cB'=\cB/\rho^{-1}(\cK)$, noting that the composite $\rho^{-1}(\cK)\to \cG\to \cB$ is still a closed immersion as $\cK$ does not meet the multiplicative part), together with an injection
\[
\rho': \cG'[p^\infty]\hookrightarrow \mathbb X'.
\]
It is clear that this induces an isomorphism
\[
\phi_\ast: C^{\mathfrak{Ig},\mathbb X}_Z\cong C^{\mathfrak{Ig},\mathbb X'}_Z.
\]

We need to compare the perfect torus torsors. Note that the maps $f_0: X\to \cB$ and $f_0': X\to \cB'$ satisfy that the composite $X\buildrel{f_0}\over\longrightarrow \cB\to \cB'$ is $p^m f_0'$ where $p^m$ is the degree of $\phi$. In particular, the pullback of the natural $\mathbb S_Z$-torsor over $C^{\mathfrak{Ig},\mathbb X'}_Z$ under $\phi_\ast$ is the pushout of the natural $\mathbb S_Z$-torsor over $C^{\mathfrak{Ig},\mathbb X}_Z$ under $p^m: \mathbb S_Z\to \mathbb S_Z$. After perfection, we get a natural isomorphism, and as $p^m$ is a scalar, it preserves the cone decomposition.

It is now easy to check that this isomorphism on completions at the boundary is compatible with the given one in the interior, which implies the desired extension.
\end{proof}

\subsection{Partial minimal compactifications of Igusa varieties}

In this section, we will analyze the partial minimal compactifications of Igusa varieties, by using the geometry of the toroidal compactifications from the last subsection. As before, we fix an algebraically closed field $k$ of characteristic $p$ and a $p$-divisible group $\mathbb X$ with $G$-structure as usual.

\subsubsection{Partial minimal compactifications are affine} First, we observe that the partial minimal compactifications of leaves are affine.

\begin{thm}\label{leaves are affine} The partial minimal compactification $\mathscr{C}^{\mathbb X,*}$ is affine.
\end{thm}

\begin{proof} First, we prove that there exists some $\mathbb X$ in the given isogeny class for which the result is true. By a result of Nie, \cite[Proposition 1.5, Corollary 1.6]{nie-fundamental}, there is an Ekedahl--Oort stratum that is completely contained in the given Newton stratum. Taking $\mathbb X$ so that $\mathbb X[p]$ defines such an Ekedahl--Oort stratum, the leaf $\mathscr{C}^{\mathbb X}$ is a closed subset of the corresponding Ekedahl--Oort stratum (as it is closed in its Newton stratum). By \cite[Theorem C]{boxer}, the partial minimal compactifications of Ekedahl-Oort strata are affine. As mentioned above, the definitions of partial minimal compactifications of~\cite{lan-stroh} and~\cite{boxer} are equivalent in our situation. We conclude by Proposition~\ref{closed immersion partial comp}.

In general, we can find an isogeny $\phi:\mathbb X'\to \mathbb X$ from a $p$-divisible group $\mathbb X'$ with $G$-structure for which the result holds, and we may assume that $\phi$ induces isomorphisms of \'etale and multiplicative parts.\footnote{Indeed, $\mathbb X$ and $\mathbb X'$ decompose into the sum of their biconnected part, and their \'etale and multiplicative part, and one can choose the quasi-isogeny on these parts individually. The \'etale and multiplicative parts are given by $\mathcal O_F\otimes \Z_p$-lattices in isomorphic $F\otimes \Q_p$-modules, and thus are also isomorphic.} Using Corollary~\ref{cor:correspondence igusa}, we get a correspondence
\[
\xymatrix{\mathscr{C}^{\mathbb X',*} & \ar[l]_{\pi^{\tor}_1}\ar[r]^{\pi^{\tor}_2} \mathfrak{Ig}^{\mathbb X,\tor}& \mathscr{C}^{\mathbb X,*}}
\]
which extends the diagram 
\[
\xymatrix{\mathscr{C}^{\mathbb X'} & \ar[l]_{\pi_1}\mathfrak{Ig}^{\mathbb X}\ar[r]^{\pi_2}& \mathscr{C}^{\mathbb X}.}
\]
Note that it follows from Theorem~\ref{perfect toroidal igusa is etale} that the scheme $\mathfrak{Ig}^{\mathbb X,\tor}$ is a limit of a tower of schemes of finite type, with transition maps finite surjective maps (namely, finite \'etale maps, and Frobenius maps). By finite presentation, both maps $\pi_1^\tor$ and $\pi_2^\tor$ factor over some term in this tower. Now the result follows from Lemma~\ref{correspondence affine} below, and induction to ensure that the boundary of $\mathscr{C}^{\mathbb X,*}$ is affine.
\end{proof}

\begin{lem}\label{correspondence affine} Consider a diagram
\[\xymatrix{
X_0\ar@{^(->}[d] & C_0\ar@{^(->}[d]\ar[l]\ar[r]  & Y_0\ar@{^(->}[d]\\
X & C\ar[l]_{\pi_1} \ar[r]^{\pi_2} & Y
}\]
of schemes, where the vertical arrows are closed immersions and $C_0\subset C$ is topologically the preimage both of $X_0\subset X$ and of $Y_0\subset Y$. Assume that $X$, $X_0$ and $Y_0$ are affine, and that $\pi_1$ and $\pi_2$ are proper and surjective and finite away from $C_0$. Then $Y$ is affine.
\end{lem}

\begin{proof} Let $C\to C^\prime\to Y$ be the Stein factorization of $\pi_2$. Note that the map $C\to X$ factors uniquely over $C^\prime$ as $H^0(C^\prime,\mathcal O_{C^\prime}) = H^0(C,\mathcal O_C)$ and $X$ is affine. The resulting map $C^\prime\to X$ is automatically proper and surjective (cf.~e.g.~\cite[Tag 03GN]{stacks-project}). We can then replace $C$ by $C^\prime$.

In particular, we can assume that $\pi_2$ is finite and surjective. As affineness descends along finite surjections, it follows that it is enough to prove that $C$ is affine. For this, it suffices to show that $\pi_1$ is finite, for which it suffices to prove that $\pi_1$ is quasi-finite. This is evident away from $X_0$. On $X_0$, the map $C_0\to X_0$ is a proper map of affine schemes (as $C_0\to Y_0$ is finite and $Y_0$ is affine), and thus finite. This gives the result.
\end{proof}

In particular, if one defines the partial minimal compactification $\mathfrak{Ig}^{\mathbb X,\ast}$ as the normalization of $\mathscr C^{\mathbb X,\ast}$ in $\mathfrak{Ig}^{\mathbb X}$, we have the following results.

\begin{prop}\label{perfect min comp} The map $\mathfrak{Ig}^{\mathbb X,\ast}\to \mathscr C^{\mathbb X,\ast}$ is integral, and in particular $\mathfrak{Ig}^{\mathbb X,\ast}$ is affine. It agrees with the Stein factorization of
\[
\mathfrak{Ig}^{\mathbb X,\tor}\to \mathscr C^{\mathbb X,\ast}.
\]
In particular,
\[
\mathfrak{Ig}^{\mathbb X,\ast} = H^0(\mathfrak{Ig}^{\mathbb X,\tor},\cO_{\mathfrak{Ig}^{\mathbb X,\tor}}),
\]
and a $G$-isogeny $\phi: \mathbb X\to \mathbb X'$ inducing isomorphisms of \'etale and multiplicative parts induces an isomorphism
\[
\mathfrak{Ig}^{\mathbb X,\ast}\cong \mathfrak{Ig}^{\mathbb X',\ast}.
\]
\end{prop}

\begin{remark}
One would expect that the final statement does not need $\phi$ to induce isomorphisms of \'etale and multiplicative parts.
\end{remark}

\begin{proof} All statements are immediate.
\end{proof}
 
\subsubsection{Partial minimal compactifications of Igusa varieties in the completely slope divisible case}

From now on, we fix, for a given $b\in B(G_{\Q_p},\mu^{-1})$, a $p$-divisible group with $G$-structure $\mathbb X=\mathbb X_b$ in the given isogeny class such that $\mathbb X_b$ is completely slope divisible. As the perfect Igusa varieties are invariant under isogenies (inducing isomorphisms on \'etale and multiplicative parts, which can always be arranged), it is for the rest enough to understand one choice of $\mathbb X_b$. Thus, by abuse of notation, we denote by $\mathrm{Ig}^b_m:=\mathrm{Ig}^{\mathbb X_b}_m$ be the Igusa variety over $\mathscr{C}^b:=\mathscr C^{\mathbb X_b}$ of level $p^m$; recall that this is a finite \'etale cover of $\mathscr{C}^b$, which is Galois with Galois group $\Gamma_{m,b}=\Gamma_{m,\mathbb X_b}$.

\begin{defn}\label{defn:minimal igusa}
Define the partial minimal compactification $\mathrm{Ig}^{b,*}$ of $\mathrm{Ig}^b$ to be the normalization of $\mathscr{C}^{b,*}$ in $\mathrm{Ig}^b$. 
\end{defn}

\noindent While we cannot in general expect $\mathrm{Ig}^{b,*}$ to be finite \'etale over $\mathscr{C}^{b,*}$, the partial minimal compactifications of Igusa varieties satisfy the following basic properties that follow formally from the definition.

\begin{lem}\label{basic properties of minimal compactifications}\leavevmode
\begin{enumerate}
\item The morphism $h^{b,*}:\mathrm{Ig}^{b,*}\to \mathscr{C}^{b,*}$ is finite and surjective. In particular, $\mathrm{Ig}^{b,*}$ is affine.
\item The morphism $\mathrm{Ig}^b\hookrightarrow \mathrm{Ig}^{b,*}$ is the open immersion of a dense subset and $\mathrm{Ig}^{b,*}$ is normal.
\end{enumerate}
\end{lem}

\begin{proof} This is clear.
\end{proof}

\subsubsection{Igusa cusp labels} In order to describe the boundary strata in $\mathrm{Ig}^{b,*}$ explicitly, we will define a notion of Igusa cusp label that mirrors the notion of cusp label introduced in \S~\ref{cusp labels}.
Recall that we have fixed the completely slope divisible $p$-divisible group $\mathbb{X}_b$ with $G$-structure.

\begin{defn}\label{defn:Igusa cusp label} An Igusa cusp label is a triple $\tilde{Z}=(\mathrm{Z}_b,\mathrm{Z}^p,X)$ where
\begin{enumerate}
\item $\mathrm{Z}_b$ is an $\mathcal O_F$-stable filtration of $\mathbb X_b$ of the form
\[
\mathrm{Z}_{b,-2}\subset \mathrm{Z}_{b,-1}\subset \mathbb X_b
\]
such that $\mathrm{Gr}^{\mathrm{Z}_b}_{-2} = \mathrm{Z}_{b,-2}$ is multiplicative, $\mathrm{Gr}^{\mathrm{Z}_b}_0 = \mathbb X_b/\mathrm{Z}_{b,-1}$ is \'etale, and the polarization identifies these as Cartier dual. In particular, it induces a principal polarization of
\[
\mathrm{Gr}^{\mathrm{Z}_b}_{-1} = \mathrm{Z}_{b,-1}/\mathrm{Z}_{b,-2}.
\]
\item $\mathrm{Z}^p$ is an $\mathcal O_F$-stable symplectic filtration
\[
\mathrm{Z}^p_{-2}\subset \mathrm{Z}^p_{-1}\subset L\otimes_{\Z} \widehat{\Z}^p.
\]
\item $X$ is a finite projective $\cO_F$-module equipped with isomorphisms $X\otimes (\Q_p/\Z_p)\cong \mathrm{Gr}^{\mathrm{Z}_b}_0$, $X\otimes \widehat{\Z}^p\cong \mathrm{Gr}^{\mathrm{Z}^p}_0$.
\end{enumerate}
There is an action of $J_b(\Q_p)\times G(\A_f^p)$ on Igusa cusp labels. If $K\subset J_b(\Q_p)\times G(\A_f^p)$ is a compact open subgroup, then an Igusa cusp label at level $K$ is a $K$-orbit of Igusa cusp labels. For a general closed subgroup $H\subset J_b(\Q_p)\times G(\A_f^p)$, an Igusa cusp label at level $H$ is a compatible family of Igusa cusp labels at level $K$ for all $K\supset H$.
\end{defn}

If $K=\Gamma_{b}(p^m) K^p(N)$ where $\Gamma_b(p^m)=\ker(\Gamma_b\to \Gamma_{m,b})$ is the principal level $p^m$ subgroup of $\Gamma_b$, one can identify Igusa cusp labels at level $K$ with triples $\tilde{Z}=(\mathrm{Z}_{m,b},\mathrm{Z}_{N},X)$ where $Z=(\mathrm{Z}_{N},X)$ form a usual cusp label at level $K(N)$, and $\mathrm{Z}_{m,b}$ is an $\mathcal O_F$-linear symplectic filtration of $\mathbb X_b[p^m]$ with an isomorphism
\[
X/p^m\cong \mathrm{Gr}^{\mathrm{Z}_{m,b}}_0.
\]
(This forces in particular that this is an \'etale $p$-divisible group, and then the symplectic pairing forces $\mathrm{Gr}^{\mathrm{Z}_{m,b}}_{-2}$ to be multiplicative. The filtration will automatically lift modulo higher powers of $p$.)

Given an Igusa cusp label $\tilde{Z}$ at level $\Gamma_b(p^m) K^p(N)$, we define the group 
\[
\Gamma_{\tilde{Z}}:=\{\gamma\in \mathrm{Aut}_{\cO_F}(X)\mid g\equiv 1\mod p^m N\}.
\]
 
\subsubsection{Boundary strata in partial toroidal compactifications of Igusa varieties}

Consider a cusp label $Z=(\mathrm{Z}_{N},X)$ at level $K(N)$. We will now give a description of the formal completion $\widehat{\mathrm{Ig}}^{b,\tor}_{m,Z}$ of $\mathrm{Ig}^{b,\tor}_{m}$ along $\mathrm{Ig}^{b,\tor}_{m,Z}$ in terms of Igusa cusp labels at level $\Gamma_b(p^m) K^p(N)$.

\begin{thm}\label{thm:tor Igusa strata}\leavevmode
\begin{enumerate}
\item There exists a decomposition into open and closed formal subschemes 
\[
\widehat{\mathrm{Ig}}^{b,\tor}_{m,Z} = \bigsqcup_{\tilde{Z}} \widehat{\mathrm{Ig}}^{b,\tor}_{m,\tilde{Z}}
\]
where $\tilde{Z}$ runs over Igusa cusp labels at level $\Gamma_b(p^m) K^p(N)$ above $Z$. 
\item Let $\tilde{Z}$ be an Igusa cusp label at level $\Gamma_b(p^m) K^p(N)$ above $Z$. Fix a symplectic $\mathcal O_F$-linear splitting
\[
\mathbb X_{m,b} = \bigoplus_{i=-2}^0 \mathrm{Gr}^{\mathrm{Z}_{m,b}}_i,
\]
cf.~Proposition~\ref{classification pdivisible}, and as usual a splitting of $\mathrm{Z}_N$. The formal scheme $\widehat{\mathrm{Ig}}^{b,\tor}_{m,\tilde{Z}}$ can be described as follows:

Consider the diagram
\[\xymatrix{
\Xi_{Z}\ar@{^{(}->}[r]\ar[d]&\Xi_{Z,\Sigma_Z}\ar[dl]
\\
C_{Z}\ar[d]&\ 
\\
\mathscr S_{Z}&\ 
}\]
describing $\widehat{\mathscr S}^\tor_{Z}$. Then there is an abelian variety
\[
C_{\tilde{Z}}\to \mathrm{Ig}^b_{Z,m}
\]
over the level-$p^m$-Igusa variety over the leaf $\mathscr C^b_{Z}$ of $\mathscr S_{Z}$, and a commutative diagram
\[\xymatrix{
C_{\tilde{Z}}\ar[r]\ar[d] & C_{Z}\ar[d]\\
\mathrm{Ig}^b_{Z,m}\ar[r]& \mathscr S_{Z}.
}\]
Let $\Xi_{\tilde{Z},\Sigma_Z}\to C_{\tilde{Z}}$ be the pullback of $\Xi_{Z,\Sigma_Z}\to C_{Z}$. Then there is an action of $\Gamma_{\tilde{Z}}$ on all objects, and letting $\mathfrak X_{\tilde{Z},\Sigma_Z}$ be the completion of $\Xi_{\tilde{Z},\Sigma_Z}$ at its toroidal boundary, there is an isomorphism
\[
\widehat{\mathrm{Ig}}^{b,\tor}_{m,\tilde{Z}}\cong \mathfrak X_{\tilde{Z},\Sigma_Z}/\Gamma_{\tilde{Z}}.
\]
\end{enumerate}
\end{thm}

\begin{proof} For the first part, note that in the notation of Theorem~\ref{thm:formal completions Igusa}, over $\widehat{\mathrm{Ig}}^{b,\tor}_{m,Z}$ one parametrizes in particular isomorphisms $\mathscr H_Z^\mu[p^m]\cong \mathbb X_b^\mu[p^m]$. Here $\mathscr H_Z$ is the connected part of the $p$-divisible group of the Raynaud extension. Inside its multiplicative part, there is the $p$-divisible group $T[p^\infty]\subset \mathscr H_Z^\mu$. In particular, we get a multiplicative subspace
\[
T[p^m]\cong \mathscr H_Z^\mu[p^m]\cong \mathbb X_b^\mu[p^m]\subset \mathbb X_b[p^m]\ .
\]
This is necessarily locally constant, and together with its dual defines a symplectic filtration
\[
T[p^m] = \mathrm{Z}_{m,b,-2}\subset \mathrm{Z}_{m,b,-1}\subset \mathbb X_b[p^m].
\]
Together with the identification $T[p^m] = \mathrm{Z}_{m,b,-2}$, this defines the Igusa cusp label $\tilde{Z}$ above $Z$ (as $T$ is the torus with cocharacter group $X$, so $T[p^m] = \Hom(X,\mu_{p^m})$, and the identification $T[p^m]\cong \mathrm{Z}_{m,b,-2}$ is Cartier dual to an identification $\mathrm{Gr}^{\mathrm{Z}_{m,b}}_0\cong X/p^m$).

If we fix the cusp label $\tilde{Z}$ together with a symplectic $\mathcal O_F$-linear splitting
\[
\delta_{m,b}: \mathbb X_b[p^m]\cong \bigoplus_{i=-2}^0 \mathrm{Gr}^{\mathrm{Z}_{m,b}}_i,
\]
then the data of an Igusa level-$p^m$-structure can be described as follows. Recall that over $C_{Z}\times_{\mathscr S_{Z}} \mathscr C^b_{Z}$ we have the Raynaud extension
\[
0\to T\to \cG\to \cB\to 0
\]
where $T$ is the split torus with cocharacter group $X$ and $\cB$ lies in the leaf $\mathscr C^b_{Z}$ of the smaller Shimura variety $\mathscr S_{Z}$. In particular, $\cB[p^\infty]$ carries a slope filtration; we let $\cB[p^\infty]^\mu$ denote the multiplicative graded piece and $\cB[p^\infty]^i$ denote the graded pieces with slopes $\lambda_i\in (0,1)$.

\begin{defn}\label{defn:compatible with filtration} An Igusa level-$p^m$-structure on a $C_{Z}\times_{\mathscr{S}_{Z}}\mathscr{C}^b_{Z}$-scheme $\mathscr{T}$ that is compatible with the Igusa cusp label $\tilde{Z}$ and $\delta_{m,b}$ consists of the following data:
\begin{enumerate}
\item An isomorphism
\[
\rho_{\cB^\mu,m}:\cB[p^m]^\mu\times_{C_{Z}\times_{\mathscr{S}_{Z}}\mathscr{C}^b_{Z}}\mathscr{T}
\toisom (\mathrm{Gr}^{\mathrm{Z}_{m,b}}_{-1})^\mu[p^m]\times_k \mathscr{T}
\]
compatible with the $\cO_F$-action on both sides. 
\item A splitting 
\[
\delta_m:\cB[p^m]^\mu\times_{C_{Z}\times_{\mathscr{S}_{Z}}\mathscr{C}^b_{Z}}\mathscr{T}\to 
\mathscr{H}_Z^\mu[p^m]\times_{C_{Z}\times_{\mathscr{S}_{Z}}\mathscr{C}^b_{Z}}\mathscr{T}
\] 
compatible with the $\cO_F$-action on both sides. 
\item For each $\lambda_i\in (0,1)$, an isomorphism
\[
\rho_{i, m}:\cB[p^m]^i\times_{C_{Z}\times_{\mathscr{S}_{Z}}\mathscr{C}^b_{Z}}\mathscr{T} \toisom \mathbb{X}_i[p^m]\times_k\mathscr{T},
\]
that commutes with the $\cO_{F}$-action, and commutes with the isomorphisms $\cB[p^m]^i\toisom (\cB[p^m]^j)^{\vee}$ and $\mathbb{X}_i[p^m]\toisom (\mathbb{X}_j[p^m])^\vee$ induced by polarizations for $\lambda_i+\lambda_j=1$ up to an element of $(\Z/p^m\Z)^\times(\mathscr{T})$ that is independent of $i$. The isomorphisms $\rho_{i,m}$ are required to lift fppf locally to $p^{m'}$ torsion for any $m'\geq m$.
\end{enumerate}
\end{defn}

Note that the Igusa cusp label itself defines an isomorphism $X/p^m\cong \mathrm{Gr}^{\mathrm{Z}_{m,b}}_0$, and thus $T[p^m]\cong \mathrm{Z}_{m,b,-2}$ by Cartier duality. It is easy to see that Igusa level-$p^m$-structures compatible with $\tilde{Z}$ and $\delta_{m,b}$ are representable by a scheme
\[
C_{\tilde{Z}}\to C_{Z}\times_{\mathscr{S}_{Z}}\mathscr{C}^b_{Z}
\]
that we will identify in a second. Moreover, a direct comparison of moduli problems, and Theorem~\ref{thm:formal completions Igusa}, shows that
\[
\widehat{\mathrm{Ig}}^{b,\tor}_{m,\tilde{Z}}
\]
can be described in terms of the pullback of $\Xi_{Z,\Sigma_{Z}}\to C_{Z}$ to $C_{\tilde{Z}}$ in the way described in the statement. (In Theorem~\ref{thm:formal completions Igusa}, one was taking the quotient by $\Gamma_{Z}$ whereas here it is $\Gamma_{\tilde{Z}}$; this difference is accounted for by the isomorphism $T[p^m]\cong \mathrm{Z}_{m,b,-2}$ that is part of the cusp label.)

It remains to identify $C_{\tilde{Z}}$. We note that the data in parts (1) and (3) precisely define an Igusa level structure of level $p^m$ on $\cB$. This shows that there is a natural map
\[
C_{\tilde{Z}}\to C_{Z}\times_{\mathscr S_{Z}} \mathrm{Ig}^{b}_{Z,m},
\]
where $\mathrm{Ig}^b_{Z,m}\to \mathscr C^b_{Z}$ denotes the Igusa level-$p^m$-covering. Recall that $C_{Z}\to \mathscr S_Z$ parametrizes $\mathcal O_F$-linear Raynaud extensions $0\to T\to \cG\to \cB\to 0$ with a splitting $\cB[N]\to \cG[N]$; equivalently, $C_{Z}$ is given by $\underline{\Hom}_{\mathcal O_F}(\tfrac 1{N} X,\cB^\vee)$. Parametrizing data of type (2) in addition amounts to a splitting of the extension
\[
0\to T[p^m]\to \cG[p^m]^\mu\to \cB[p^m]^\mu\to 0.
\]
This means that $C_{\tilde{Z}}\to \mathrm{Ig}^b_{Z,m}$ is given by $\underline{\Hom}_{\mathcal O_F}(\tfrac 1{N}X,(\cB/\cB[p^m]^\mu)^\vee)$, which is an abelian scheme. (We note that this explicit description also implies that the natural map $C_{\tilde{Z}}\to C_{Z}\times_{\mathscr S_{Z}} \mathrm{Ig}^{b}_{Z,m}$ is finite \'etale.)
\end{proof}

\subsubsection{Boundary strata in partial minimal compactifications of Igusa varieties}\label{sec:min Igusa strata} 
As an application, we can describe the boundary strata of $\mathrm{Ig}^{b,*}_{m}$ in terms of Igusa cusp labels at level $\Gamma_b(p^m) K^p(N)$.

\begin{thm}\label{thm:min Igusa strata} We have a decomposition into locally closed strata
\[
\mathrm{Ig}^{b,*}_m = \bigsqcup_{\tilde{Z}}\mathrm{Ig}^{b}_{\tilde{Z}},
\]
where $\tilde{Z}$ runs over Igusa cusp labels at level $\Gamma_b(p^m) K^p(N)$. 

If $\tilde{Z}$ lies over a cusp label $Z$ of level $K(N)$, then $\mathrm{Ig}^{b}_{\tilde{Z}}\cong \mathrm{Ig}^b_{Z,m} \to \mathscr{C}^{b}_{Z}$ is isomorphic to the Igusa variety of level $p^m$ over the Oort central leaf $\mathscr{C}^b_{Z}$ in the special fiber of the boundary stratum $\mathscr{S}_{Z}$.
\end{thm}

\begin{proof} Observe that $\mathrm{Ig}^{b,*}_m\to \mathscr C^{b,*}$ is the Stein factorization of $\mathrm{Ig}^{b,\tor}_m\to \mathscr C^{b,*}$: since the toroidal compactification is smooth, its Stein factorization factors over the normalization of $\mathscr C^{b,*}$ in $\mathrm{Ig}^b_m$, and then the map is an isomorphism by Zariski's main theorem. The description of the boundary strata of $\mathrm{Ig}^{b,*}_m$ now follows from this observation, from Theorem~\ref{thm:tor Igusa strata}, and from the same arguments as in the case of Shimura varieties using Fourier--Jacobi series, cf.~e.g.~\cite[Section 7]{lan-thesis}. (In fact, our situation is even simpler than the one in~\emph{loc. cit.} because $\mathrm{Ig}^{b,*}_m$ is affine.)
\end{proof}

\newpage

\section{The Hodge--Tate period morphism on compactifications}\label{fibers of HT}

\subsection{Statements} As before, fix a prime $p$ that is unramified in the CM field $F$, and fix an integer $N\geq 3$ prime to $p$ and a cone decomposition $\Sigma$ level $K(N)$ as in Remark~\ref{trivial stabilizer}. We consider the infinite-level minimal and toroidal compactifications
\[\begin{aligned}
\mathcal S_{K(p^\infty N)}^\ast &:= \varprojlim_m \mathscr S_{K(p^m N),\Q_p}^{\ast,\diamondsuit},\\
\mathcal S_{K(p^\infty N)}^\tor &:= \varprojlim_m \mathscr S_{K(p^m N),\Q_p}^{\tor,\diamondsuit}
\end{aligned}\]
as diamonds. They admit the Hodge--Tate period map
\[\begin{aligned}
\pi_\HT^\ast: \mathcal S_{K(p^\infty N)}^\ast&\to \Fl,\\
\pi_\HT^\tor: \mathcal S_{K(p^\infty N)}^\tor&\to \Fl,
\end{aligned}\]
where the flag variety $\Fl$ is the adic space (or diamond) over $\Q_p$ parametrizing totally isotropic $F$-linear subspaces over $V$.

Our goal in this section is to describe the fibers of these maps. For this, let $C$ be any complete algebraically closed nonarchimedean extension of $\Q_p$ with ring of integers $\cO_C$ and residue field $k$. Let $x\in\Fl(C)$ be a point of the flag variety, which by \cite[Theorem B]{scholze-weinstein} gives rise to a $p$-divisible group $\mathbb X_{\cO_C}$ with $G$-structure over $\cO_C$, equipped with an isomorphism $\alpha: T_p(\mathbb X_{\cO_C})\cong L\otimes_{\Z} \Z_p$ compatible with $G$-structures. Associated to the special fibre $\mathbb X_k$, we get the perfect Igusa variety $\mathfrak{Ig}^{\mathbb X}$ equipped with its partial minimal and toroidal compactification $\mathfrak{Ig}^{\mathbb X,\ast}$, $\mathfrak{Ig}^{\mathbb X,\tor}$.

Being perfect schemes, they admit unique flat deformations to formal schemes over $W(k)$, which we can then base change to $\cO_C$. Let us denote simply by a subscript $_C$ the generic fibres of these formal schemes, which are perfectoid spaces over $\Spa C$.

The main theorem of this section is the following:

\begin{thm}\label{thm:compactified fibres} There are natural maps
\[\begin{aligned}
\mathfrak{Ig}^{\mathbb X,\ast}_C&\to (\pi_\HT^\ast)^{-1}(x),\\
\mathfrak{Ig}^{\mathbb X,\tor}_C&\to (\pi_\HT^\tor)^{-1}(x).
\end{aligned}
\]
They are open immersions of perfectoid spaces with the same rank-$1$-points; in fact, the target is the canonical compactification of the source.
\end{thm}

In particular, this implies the following corollary:

\begin{cor}\label{cohomology fibers} There are natural Hecke-equivariant isomorphisms
\[\begin{aligned}
R\Gamma(\mathfrak{Ig}^{\mathbb X,\ast},\mathbb F_\ell)&\cong (R(\pi_\HT^\ast)_\ast \mathbb F_\ell)_x,\\
R\Gamma(\mathfrak{Ig}^{\mathbb X,\tor},\mathbb F_\ell)&\cong (R(\pi_\HT^\tor)_\ast \mathbb F_\ell)_x,\\
\end{aligned}\]
Moreover, the stalks at higher-rank points agree with the stalks at the corresponding rank $1$ points.
\end{cor}

\begin{proof} This follows from Theorem~\ref{thm:compactified fibres} and the arguments in \cite[Section 4.4]{caraiani-scholze}.
\end{proof}

We start by giving an explicit description of the Hodge--Tate period morphism on boundary strata, in terms of toroidal boundary charts. Next, we construct maps from $\mathfrak{Ig}^{\mathbb X,\tor}_C$ into the fibre of the Hodge--Tate period map, by describing the map on toroidal boundary charts. We check that it is (essentially) an isomorphism by checking that it has the correct geometric points; by general properties of (qcqs) diamonds, this is enough to imply the result.

In the final part of this section, we deduce a semiperversity result for $R\pi_{\HT\ast}^\circ \mathbb F_\ell$.

\subsection{The Hodge--Tate period morphism on boundary strata}

Our first goal is to describe the Hodge--Tate period map explicitly on the toroidal compactification. We have the naive integral models
\[
\mathscr S^\tor_{K(p^m N),\mathbb Z_p}
\]
of $\mathscr S^\tor_{K(p^m N),\mathbb Q_p}$ from Theorem~\ref{existence of toroidal compactifications}, regarded as $p$-adic formal schemes. We can take the inverse limit
\[
\mathscr S^\tor_{K(p^\infty N),\mathbb Z_p} = \varprojlim_m \mathscr S^\tor_{K(p^m N),\mathbb Z_p}
\]
as a $p$-adic formal scheme. Fix a cusp label $Z$ of level $K(N)$. We can take the fibre product of formal schemes
\[
\widehat{\mathscr S}^\tor_{K(p^\infty N),Z,\mathbb Z_p} = \mathscr S^\tor_{K(p^\infty N),\mathbb Z_p}\times_{\mathscr S^\tor_{K(N)}} \widehat{\mathscr S}^\tor_{K(N),Z}
\]
with the formal completion at the boundary stratum corresponding to the cusp $Z$. The cusp $Z$ defines a rational parabolic subgroup $P=P_r\subset G$ (up to conjugation) depending on the $\mathcal O_F$-rank $r$ of the $\mathcal O_F$-lattice $X$ implicit in the cusp label $Z$, and the projection to the cusp labels at $p$ defines a natural map
\[
\pi_1: \widehat{\mathscr S}^\tor_{K(p^\infty N),Z,\mathbb Z_p}\to (G/P)(\mathbb Q_p)\ ;
\]
here, $(G/P)(\mathbb Q_p)$ parametrizes symplectic $\mathcal O_F$-linear filtrations
\[
\mathrm{Z}_{p^\infty,-2}\subset \mathrm{Z}_{p^\infty,-1}\subset L\otimes_{\Z} \Z_p
\]
where the $\cO_F\otimes_{\Z} \Z_p$-rank of $\mathrm{Z}_{p^\infty,-2}$ is $r$. In particular, over $\widehat{\mathscr S}^\tor_{K(p^\infty N),Z,\mathbb Z_p}$, we have a natural local system of $\cO_F\otimes_{\Z} \Z_p$-lattices $L_{Z} := \mathrm{Gr}^{\mathrm{Z}_{p^\infty}}_{-1}$ with a perfect alternating form. On the other hand, there is a natural projection
\[
\widehat{\mathscr S}^\tor_{K(p^\infty N),Z,\mathbb Z_p}\to \mathscr S_{Z,K(N),\mathbb Z_p}
\]
remembering only the base $\mathcal B$ of the Raynaud extension, with its extra structure. Moreover, over $\widehat{\mathscr S}^\tor_{K(p^\infty N),Z,\mathbb Z_p}$, we have compatible maps $L_{Z}/p^m\to \mathcal B[p^m]$ for varying $m$, inducing level structures on the generic fibre.

Let
\[
\widehat{\mathscr S}^\tor_{K(p^\infty N),Z,\mathbb Q_p} = \widehat{\mathscr S}^\tor_{K(p^\infty N),Z,\mathbb Z_p}\times_{\Spa \Z_p} \Spa \Q_p
\]
be the generic fibre, as a (pre-)adic space over $\Q_p$ (where we regard any formal scheme as a (pre-)adic space via the natural fully faithful functor). Its associated diamond is an open subspace
\[
(\widehat{\mathscr S}^\tor_{K(p^\infty N),Z,\mathbb Q_p})^\diamondsuit\subset \mathcal S_{K(p^\infty N)}^\tor
\]
whose $C$-valued points are exactly those that specialize into the cusp $Z$. The Hodge--Tate filtration of $\mathcal B$ now defines a totally isotropic subspace
\[
\Lie \mathcal B(1)\otimes \mathcal O_{(\widehat{\mathscr S}^\tor_{K(p^\infty N),Z,\mathbb Q_p})^\diamondsuit}\hookrightarrow L_{Z}\otimes_{\Z_p} \mathcal O_{(\widehat{\mathscr S}^\tor_{K(p^\infty N),Z,\mathbb Q_p})^\diamondsuit} .
\]
Pulling this filtration back to
\[
\mathrm{Z}_{p^\infty,-1}\otimes_{\Z_p} \mathcal O_{(\widehat{\mathscr S}^\tor_{K(p^\infty N),Z,\mathbb Q_p})^\diamondsuit}
\]
then defines a totally isotropic subspace of
\[
L\otimes_{\Z} \mathcal O_{(\widehat{\mathscr S}^\tor_{K(p^\infty N),Z,\mathbb Q_p})^\diamondsuit} ,
\]
yielding a map
\[
\pi_{\HT,Z}: (\widehat{\mathscr S}^\tor_{K(p^\infty N),Z,\mathbb Q_p})^\diamondsuit\to \Fl .
\]

\begin{thm}\label{thm:hodge tate boundary} The map $\pi_{\HT,Z}$ constructed above agrees with the composite
\[
(\widehat{\mathscr S}^\tor_{K(p^\infty N),Z,\mathbb Q_p})^\diamondsuit\to (\mathscr S^\tor_{K(p^\infty N),\mathbb Q_p})^\diamondsuit\buildrel{\pi_\HT^\tor}\over\longrightarrow \Fl .
\]
\end{thm}

\begin{proof} For points away from the boundary, this follows from \cite[Proposition 3.3.1]{scholze} (as the forgetful map to the Siegel modular variety is injective on flag varieties). We deduce the general case by continuity: The locus where the two maps agree is closed. If they do not agree, they differ on some quasicompact open subspace
\[
U\subset |(\widehat{\mathscr S}^\tor_{K(p^\infty N),Z,\mathbb Q_p})^\diamondsuit|\subset |(\mathcal S^\tor_{K(p^\infty N)})^\diamondsuit|,
\]
necessarily contained in the boundary. Note that $U$ is necessarily the preimage of some quasicompact open subspace $U_m\subset |\mathscr S^\tor_{K(p^m N),\mathbb Q_p}|$ for $m$ large enough, where $U\to U_m$ will be surjective, so that $U_m$ is contained in the boundary. But there are no open subspaces of $\mathscr S^\tor_{K(p^m N),\mathbb Q_p}$ contained in the boundary.
\end{proof}

\noindent We can now also describe the restriction of $\pi_{\HT}^\ast$ to the boundary of the minimal compactification $\cS^*_{K(p^\infty N)}$. We have a set-theoretic decomposition 
\[
\cS^*_{K(p^\infty N)} = \bigsqcup_{\widehat{Z}}\cS_{K(p^\infty N), \widehat{Z}}
\]
in terms of cusp labels at level $K(p^\infty N)$ on the generic fiber, defined by taking the inverse limit
of the corresponding decomposition on the level of schemes. In particular, any such $\widehat{Z}$ defines a filtration
\[
\mathrm{Z}_{p^\infty,-2}\subset \mathrm{Z}_{p^\infty,-1}\subset L\otimes_{\Z} \Z_p
\]
as above, and hence $L_Z=\mathrm{Gr}_{-1}^{Z_{p^\infty}}$. Let $\Fl_{\widehat{Z}}$ denote the flag variety (as an adic space over $\Q_p$) parametrizing totally isotropic $F$-linear subspaces of $L_Z\otimes \Q_p$; this embeds naturally into $\Fl$ (taking the preimage of a totally isotropic subspace of $L_Z\otimes \Q_p$ in $Z_{p^\infty,-1}\otimes \Q_p$, it defines a totally isotropic subspace of $L\otimes \Q_p$).

\begin{cor} The restriction of $\pi_{\HT}^*: \cS^*_{K(p^\infty N)}\to \Fl$ to $\cS_{K(p^\infty N), \widehat{Z}}$ 
agrees with the composition 
\[
\cS_{K(p^\infty N),\widehat{Z}}\to \Fl_{\widehat{Z}}\hookrightarrow \Fl,
\]
where the first morphism is the Hodge--Tate period morphism for the smaller Shimura variety
$\cS_{K(p^\infty N), \widehat{Z}}$. 
\end{cor}

\subsection{Construction of the map}

In this subsection, fix a complete algebraically closed nonarchimedean field $C$ over $\mathbb Q_p$ with ring of
integers $\cO_C$ and residue field $k$, and fix a point $x\in \Fl(C)$ of the flag variety. This corresponds to a pair $(\mathbb X_{\cO_C},\alpha)$ consisting of a $p$-divisible group $\mathbb X_{\cO_C}$ with $G$-structure over $\cO_C$ and a trivialization $\alpha: T_p(\mathbb X_{\cO_C})\cong L\otimes_{\Z}\Z_p$ compatible with $G$-structures.

Note that the filtration $\mathbb X_{\cO_C}^\mu\subset \mathbb X_{\cO_C}^\circ\subset \mathbb X_{\cO_C}$ induces by transport of structure via $\alpha$ a symplectic $\cO_F$-linear filtration of $L\otimes_{\Z} \Z_p$; we fix a symplectic splitting $\delta$ of this filtration. Transporting it back via $\alpha$, this induces an $\mathcal O_F$-linear splitting
\[
\delta_{\mathbb X_{\cO_C}}: \mathbb X_{\cO_C}\cong \mathbb X_{\cO_C}^\mu\oplus \mathbb X_{\cO_C}^{(0,1)}\oplus \mathbb X_{\cO_C}^{\et}
\]
under which the polarization similarly decomposes into a direct sum.

Let $\mathbb X=\mathbb X_k$ be the special fiber of $\mathbb X_{\cO_C}$, which comes with an induced splitting $\delta_{\mathbb X}$, and fix a section $k\to \cO_C/p$.

\begin{prop}\label{prop:trivial mod epsilon} There exists $\epsilon \in \Q$, $1\geq \epsilon>0$ such that there exists an isomorphism 
\[
\rho: \mathbb{X}\times_k \cO_C/p^{\epsilon} \cong \mathbb X_{\cO_C}\times_{\cO_C}\cO_C/p^\epsilon 
\]
of $p$-divisible groups with $G$-structures, lifting the identity. Moreover, we can arrange that $\delta_{\mathbb X_{\cO_C}}\times_{\cO_C} \cO_C/p^\epsilon = \delta_{\mathbb X}\times_k \cO_C/p^\epsilon$.
\end{prop}

\begin{proof} Consider the $p$-divisible group $X:=\mathbb X_{\cO_C}\times_{\cO_C}\cO_C/p$ equipped with 
the induced extra structures. Its associated Dieudonn\'e module is a finite projective $A_{\mathrm{crys}}$-module equipped with a Frobenius; let $(M,\phi)$ be the corresponding $\phi$-module over $B_{\mathrm{crys}}^+ = A_{\mathrm{crys}}[\tfrac 1p]$. Let $(M_0,\phi_0)$ be similarly the Dieudonn\'e module associated with $\mathbb X\times_k \cO_C/p$. By \cite[Theorem A]{scholze-weinstein}, it suffices to find an isomorphism $(M,\phi)\cong (M_0,\phi_0)$ of $G$-Dieudonn\'e modules reducing to the identity on $W(k)[\tfrac 1p]$ (indeed, then the associated quasi-isogeny will be an isomorphism over $k$, and thus over $\cO_C/p^\epsilon$ for $\epsilon$ small enough). It is in fact enough to find any such isomorphism, as we can always change it by an automorphism of $(M_0,\phi_0)$.

Now \cite[Th\'eor\`eme 5.1]{farguesGbun} classifies $G$-Dieudonn\'e modules over $B_{\mathrm{crys}}^+$, as they are equivalent to $G$-bundles on the Fargues-Fontaine curve, and in particular shows that they are determined up to isomorphism by their restriction to $W(k)[\tfrac 1p]$, cf.~\cite[Th\'eor\`eme 5.6]{farguesGbun}.

For the final statement, note that we can also apply the argument individually to all three summands (identifying the outer ones as Cartier duals) and then pass to the direct sum.
\end{proof}

We consider the perfect toroidally compactified Igusa variety $\mathfrak{Ig}^{\mathbb X,\tor}$ associated with $\mathbb X$. By construction, this comes with a map
\[
f^{\tor}:\mathfrak{Ig}^{\mathbb X,\tor}\to \mathscr{S}^{\tor}
\]
which factors through the partial toroidal compactification $\mathscr{C}^{\mathbb X,\tor}$ of the leaf corresponding to $\mathbb{X}$.
  
We can base change everything from $k$ to $\cO_C/p^\epsilon$ and express everything in terms of $\mathbb X_{\cO_C}\times_{\cO_C}\cO_C/p^\epsilon$ using $\rho$. Our goal in this section is to lift the morphism 
 \[
f^{\tor}_\epsilon: \mathfrak{Ig}^{\mathbb X,\tor}_{\cO_C/p^\epsilon}\to 
\mathscr{S}^{\tor}_{\cO_C/p^\epsilon}
 \]
to a morphism of $p$-adic formal schemes. We first define the formal schemes:
we set
\[
\mathfrak{Ig}^{\mathbb X,\tor}_{\cO_C}: = W(\mathfrak{Ig}^{\mathbb X,\tor})\times_{W(k)}\cO_C,
\]
where $W(\mathfrak{Ig}^{\mathbb X,\tor})$ is the $p$-adic formal scheme over $\Spf W(k)$ obtained by taking an affine cover of $\mathfrak{Ig}^{\mathbb X,\tor}$, applying Witt vectors, and gluing. This is also the unique flat formal lift of $\mathfrak{Ig}^{\mathbb X,\tor}_{\cO_C/p^\epsilon}$ to $\cO_C$. On the other hand, we have the $p$-adic formal scheme $\mathscr{S}^\tor_{\cO_C}$, and we want to construct a map
\[
f^\tor_{\cO_C}: \mathfrak{Ig}^{\mathbb X,\tor}_{\cO_C}\to \mathscr S^\tor_{\cO_C},
\]
and in fact a lift to
\[
g^\tor: \mathfrak{Ig}^{\mathbb X,\tor}_{\cO_C}\to \mathscr S^\tor_{K(p^\infty N),\cO_C}.
\]

Let us recall the construction away from the boundary: Over $\mathfrak{Ig}^{\mathbb X}_{\cO_C/p^\epsilon}$, we have an abelian scheme 
$\cA_{\cO_C/p^\epsilon}$ with extra structures together with an isomorphism of $p$-divisible groups with $G$-structures
\begin{equation}\label{eq:isom of p-divisible groups}
\cA_{\cO_C/p^\epsilon}[p^\infty]\toisom \mathbb X_{\cO_C}\times_{\cO_C} \mathfrak{Ig}^{\mathbb X}_{\cO_C/p^\epsilon}.
\end{equation}
By Serre-Tate theory, the abelian scheme $\cA_{\cO_C/p^\epsilon}$ with $G$-structure lifts 
uniquely to an abelian scheme $\cA_{\cO_C}$ with $G$-structure over $\mathfrak{Ig}^{\mathbb X}_{\cO_C}$ equipped with an isomorphism
\[
\cA_{\cO_C}[p^\infty]\toisom \mathbb X_{\cO_C}\times_{\cO_C}\mathfrak{Ig}^{\mathbb X}_{\cO_C}
\]
of $p$-divisible groups with $G$-structures. This induces the morphism of formal schemes
\[
f: \mathfrak{Ig}^{\mathbb X}_{\cO_C}\to \mathscr{S}_{\cO_C}, 
\]
such that $\cA_{\cO_C}$ can be identified with the pullback under $f$ of the universal abelian scheme over $\mathscr{S}_{\cO_C}$. Now the composition of ($T_p$ applied to) the isomorphism $\cA_{\cO_C}[p^\infty]\cong \mathbb X_{\cO_C}$ with $\alpha: T_p(\mathbb X_{\cO_C})\cong L\otimes_{\Z} \Z_p$ gives a full level structure on the generic fibre of $\cA_{\cO_C}$, and thus (via the definition of integral models via normalization and the normality of Igusa varieties) a lift to a map
\[
g: \mathfrak{Ig}^{\mathbb X}_{\cO_C}\to \mathscr{S}_{K(p^\infty N),\cO_C}.
\]
Our goal now is to extend this to toroidal compactifications. 

\begin{thm}\label{thm:lift to infinite level} The morphism $g$ constructed above extends to a morphism of $p$-adic formal schemes
\[
g^{\tor}: \mathfrak{Ig}^{\mathbb X,\tor}_{\cO_C}\to \mathscr{S}^{\tor}_{K(p^\infty N),\cO_C}.
\]
\end{thm}

\begin{proof} By schematic density of the open part, the map is necessarily unique; to prove its existence, we can argue on formal completions along boundary strata. Fix a cusp label $Z=(\mathrm Z_{N},X)$ at level $N$.

First, we recall the description of
\[
\widehat{\mathfrak{Ig}}^{\mathbb X,\tor}_Z
\]
from Theorem~\ref{thm:formal completions perfect Igusa}. By Proposition~\ref{base of perfect torus torsor}, over $C^{\mathfrak{Ig},\mathbb X}_Z$, we have a Raynaud extension
\[
0\to T\to \cG\to \cB\to 0
\]
together with an $\cO_F$-linear embedding
\[
\rho: \cG[p^\infty]\hookrightarrow \mathbb X
\]
with \'etale quotient, and that is compatible with polarizations.

Now we lift these structures to $\cO_C$. By base change, everything exists over $\cO_C/p^\epsilon$. The injection $\cG[p^\infty]\hookrightarrow \mathbb X_b$ deforms uniquely to an injection $\cG[p^\infty]_{\cO_C}\hookrightarrow \mathbb X_{\cO_C}$ as the quotient is \'etale. By Serre--Tate theory, this induces a unique deformation of the Raynaud extension compatibly with the deformation of $\cG[p^\infty]$, yielding a Raynaud extension
\[
0\to T\to \cG_{\cO_C}\to \cB_{\cO_C}\to 0
\]
over $C^{\mathfrak{Ig},\mathbb X}_{Z,\cO_C}$ with an injection $\rho_{\cO_C}: \cG_{\cO_C}[p^\infty]\hookrightarrow \mathbb X_{\cO_C}$, which is $\mathcal O_F$-linear and compatible with polarizations after taking the quotient by the isotropic subspace $T[p^\infty]$.

We also need to describe the lift of the $\mathbb S_{Z,\perf}$-torsor
\[
\Xi^{\mathfrak{Ig},\mathbb X}_Z\to C^{\mathfrak{Ig},\mathbb X}_Z.
\]
To do this, we fix a symplectic $\cO_F$-linear splitting $\delta_{\mathbb X_{\cO_C},Z}$ of the filtration $\mathrm Z_{\mathbb X_{\cO_C}}$ of $\mathbb X_{\cO_C}$ over $C^{\mathfrak{Ig},\mathbb X}_{Z,\cO_C}$ given by
\[
T[p^\infty]\subset \cG_{\cO_C}[p^\infty]\subset \mathbb X_{\cO_C},
\]
compatible with the splitting $\delta_{\mathbb X_{\cO_C}}$ chosen above. Equivalently, we only have to choose a splitting on \'etale parts (it will induce a similar splitting on multiplicative parts, which together with the splitting $\delta_{\mathbb X}$ induces the desired splitting). Note that we automatically have
\[
\delta_{\mathbb X_{\cO_C},Z}\times_{\cO_C} \cO_C/p^\epsilon = (\delta_{\mathbb X_{\cO_C},Z}\times_{\cO_C} k)\times_k \cO_C/p^\epsilon
\]
as this holds for $\delta_{\mathbb X_{\cO_C}}$ by Proposition~\ref{prop:trivial mod epsilon}, and splittings of the \'etale parts are rigid.

Using the splitting $\delta_{\mathbb X_{\cO_C},Z}$ and $\rho_{\cO_C}$, we get a splitting of
\[
0\to T[p^\infty]\to \cG_{\cO_C}[p^\infty]\to \cB_{\cO_C}[p^\infty]\to 0
\]
and thus a map $\tilde{f}_0: X[\tfrac 1p]\to \cB_{\cO_C}$. Then the lift
\[
\Xi^{\mathfrak{Ig},\mathbb X}_{Z,\cO_C}\to C^{\mathfrak{Ig},\mathbb X}_{Z,\cO_C}
\]
of the $\mathbb S_{Z,\perf}$-torsor
\[
\Xi^{\mathfrak{Ig},\mathbb X}_Z\to C^{\mathfrak{Ig},\mathbb X}_Z
\]
parametrizes symmetric lifts $\tilde{f}: X[\tfrac 1p]\to \cG_{\cO_C}$ of $\tilde{f}_0$. Indeed, this describes a flat lift of this relatively perfect scheme, and over $\cO_C/p^\epsilon$, it agrees with the description from Proposition~\ref{perfect torus torsor}, using critically that $\delta_{\mathbb X_{\cO_C}}$ is constant modulo $\cO_C/p^\epsilon$.

Now fix any $m\geq 0$. Note that $\widehat{\mathfrak{Ig}}^{b,\tor}_Z$ decomposes into open and closed subsets $\widehat{\mathfrak{Ig}}^{b,\tor}_{\tilde{Z}}$ according to Igusa cusp labels $\tilde{Z}$ of level $\Gamma_b(p^m)K^p(N)$; concretely, this is given by the symplectic $\mathcal O_F$-stable filtration $T[p^m]\subset \mathrm{Z}_{b,-1}[p^m]\subset \mathbb X[p^m]$ above. Note that $\mathbb X_{\cO_C}$ and the isomorphism $\mathbb X_{\cO_C}[p^m](C)\cong L/p^m$ induce a map from Igusa cusp labels to usual cusp labels. For an Igusa cusp label $\tilde{Z}$ of level $\Gamma_b(p^m) K^p(N)$, let by abuse of notation $\tilde{Z}$ also denote the corresponding usual cusp label of level $K(p^m N)$. We want to construct maps
\[
\widehat{\mathfrak{Ig}}^{b,\tor}_{\tilde{Z}, \cO_C}\to \widehat{\mathscr S}^\tor_{K(p^m N),\tilde{Z},\cO_C}.
\]
For this, we want to use the explicit description of the right-hand side, and so choose a symplectic splitting $L/p^m\cong\bigoplus_{i=-2}^0 \mathrm{Gr}^{\mathrm{Z}_{p^m}}_i$ compatible with splitting $\delta$ chosen in the beginning of this section. (Also choose a similar splitting on $L/N$; we will in the following ignore the discussion of level structures away from $p$ as they are identical in the two setups.) By transfer of structure along $\mathbb X_{\cO_C}[p^m](C)\cong L/p^m$, this induces a splitting $\mathrm{Gr}^{\mathrm{Z}_{m,b}}_{-1}\to \mathrm{Z}_{m,b,-1}$, which then induces a splitting of
\[
0\to T[p^m]\to \cG_{\cO_C}[p^m]\to \cB_{\cO_C}[p^m]\to 0.
\]
This gives a map to the abelian variety $C_{\tilde{Z}}$ in the description of $\widehat{\mathscr S}^\tor_{K(p^m N),\tilde{Z}}$. Moreover, this splitting also induces a lift of the map $f_0: X\to \cB_{\cO_C}$ corresponding to the Raynaud extension to a map $f_m: \tfrac 1{p^m} X\to \cB_{\cO_C}$. By the description of the torus torsor
\[
\Xi^{\mathfrak{Ig},\mathbb X}_{Z,\cO_C}\to C^{\mathfrak{Ig},\mathbb X}_{Z,\cO_C}
\]
above, we get a map from $\Xi^{\mathfrak{Ig},\mathbb X}_{Z,\cO_C}$ to the torsor of lifts of $f_m$ to $\cG_{\cO_C}$ over $C_{\tilde{Z}}$, and then also a map of torus compactifications. The given structures also evidently define a pre-level structure on $\cB_{\cO_C}[p^m]$ (as in the definition of $\mathscr S_K^{\mathrm{pre}}$) but then by normality of $\widehat{\mathfrak{Ig}}^{b,\tor}_{Z,\cO_C}$ a level structure. In the inverse limit, these structures define a map
\[
\widehat{\mathfrak{Ig}}^{b,\tor}_{Z,\cO_C}\to \widehat{\mathscr{S}}^{\tor}_{K(p^\infty N),Z,\cO_C},
\]
as desired.

It remains to show that these maps on toroidal boundary charts actually extend the map defined on the open stratum. We argue modulo any power $p^M$ of $p$. Take any small enough
\[
\Spf R\subset \widehat{\mathfrak{Ig}}^{b,\tor}_{Z,\cO_C/p^M}
\]
and let $U_R\subset \Spec R$ be the complement of the boundary. We get two induced maps $U_R\to \mathscr S_{K(p^\infty N),\cO_C/p^M}$: One via the map on toroidal boundary charts, the other by mapping $U_R\to \mathfrak{Ig}^b_{\cO_C/p^M}$ to the open Igusa variety first, and then using the map $g$ above. We need to see that the two maps agree. Note that the first map also factors naturally over $\mathfrak{Ig}^b_{\cO_C/p^M}\to \mathscr S_{K(p^\infty N),\cO_C/p^M}$: Indeed, the abelian variety $A$ over $U_R$ will come with an isomorphism $A[p^\infty]^\circ\cong \mathbb X_{\cO_C}^\circ$ which is $\cO_F$-linear and compatible with polarizations after taking the quotient by the multiplicative part; moreover, the chosen splitting $\delta$ (and thus $\delta_{\mathbb X_{\cO_C}}$) induces an $\cO_F$-linear symplectic splitting
\[
A[p^\infty]\cong A[p^\infty]^\mu\oplus A[p^\infty]^{(0,1)}\oplus A[p^\infty]^\et.
\]
Now this comes with the desired isomorphism to $\mathbb X_{\cO_C}$, giving the lift $U_R\to \mathfrak{Ig}^b_{\cO_C/p^M}$ of the first map $U_R\to \mathscr S_{K(p^\infty N),\cO_C/p^M}$. At this point, we have to show that two maps $U_R\to \mathfrak{Ig}^b_{\cO_C/p^M}$ agree, where we note that all schemes are flat lifts of schemes that are relatively perfect over $\cO_C/p$. As thus the cotangent complex vanishes, it suffices to see that these maps agree over $\cO_C/p^\epsilon$. But in that case, everything is simply the base change from $k$, so we get the desired result.
\end{proof}

\subsection{The fibers of the toroidal compactification}

\begin{thm}\label{thm:tor fibers of pi HT} The morphism $g^{\tor}$ constructed in Theorem~\ref{thm:lift to infinite level} induces a map of diamonds
\[
\mathfrak{Ig}^{b,\tor}_{C}\to (\pi_\HT^\tor)^{-1}(x)
\]
that is an open immersion with the same rank $1$ points. As the target is partially proper over $\Spa(C,\cO_C)$, this implies that it is the canonical compactification of $\mathfrak{Ig}^{b,\tor}_C$.
\end{thm}

\begin{proof} By Lemma~\ref{check isomorphism} below, it is enough to check that it induces a bijection on $(C,\cO_C)$-valued points (for all possibly larger $C$). Note that on both sides, the $C$-valued points decompose according to Igusa cusp labels of level $K^p(N)$ (using Theorem~\ref{thm:hodge tate boundary} for the right-hand side) via looking at the induced $k$-valued point, and the map is by construction compatible with this decomposition. Thus, fix an Igusa cusp label $\tilde{Z}$ of level $K^p(N)$. In particular, we fix a symplectic $\cO_F$-stable filtration
\[
\mathrm Z_{b,-2}\subset \mathrm{Z}_{b,-1}\subset \mathbb X_{\cO_C}
\]
with $\mathrm{Z}_{b,-2}\cong \Hom(X,\mu_{p^\infty})$ a trivialized multiplicative group. Moreover, we fix an $\cO_F$-linear symplectic splitting $\delta_b$ of this filtration. Let
\[
\mathrm{Z}_{p^\infty,-2}\subset \mathrm{Z}_{p^\infty,-1}\subset L\otimes_{\Z} \Z_p
\]
be the induced filtration (with induced splitting), using $T_p(\mathbb X_{\cO_C})\cong L\otimes_{\Z} \Z_p$.

On the left-hand side, we parametrize now in addition the following data:\footnote{We ignore the level-$N$-structures away from $p$.}
\begin{enumerate}
\item A principally polarized abelian variety $\mathcal B$ over $\cO_C$ with $\mathcal O_F$-action as usual.
\item An $\mathcal O_F$-linear extension
\[
0\to T\to \cG\to \cB\to 0
\]
where $T$ is a torus with cocharacter group $X$. This induces in particular an $\cO_F$-linear map $f_0: X\to \cB$.
\item An isomorphism
\[
\cG[p^\infty]\cong \mathrm Z_{b,-1}
\]
compatible with the $T[p^\infty] = \mathrm Z_{b,-2}$. By the splitting $\delta_b$, this induces an $\cO_F$-linear splitting of
\[
0\to T[p^\infty]\to \cG[p^\infty]\to \cB[p^\infty]\to 0
\]
and in particular $f_0$ extends to a map $\tilde{f}_0: X[\tfrac 1p]\to \cB$.
\item Away from the boundary, a lift of $\tilde{f}_0$ to an $\cO_F$-map $\tilde{f}: X[\tfrac 1p]\to \cG(C)$ that is symmetric and positive. In general, we have a torsor $\mathcal P$ over $\cO_C$ for the torus $\mathbb S_{\tilde{Z}}$ with cocharacter group generated by $\tfrac 1{p^\infty} X\otimes X$, parametrizing symmetric lifts of $\tilde{f}_0$ to $X[\tfrac 1p]\to \cG$, and the cone decomposition $\Sigma_Z$ gives rise to an embedding $\mathcal P\subset \mathcal P_{\Sigma_Z}$. Then the final datum is an $\cO_C$-point of $\mathcal P_{\Sigma_Z}$ whose special fibre lies in the boundary.
\end{enumerate}

On the right-hand side, we parametrize the following data:
\begin{enumerate}
\item[(1')] A principally polarized abelian variety $\mathcal B$ over $\cO_C$ with $\mathcal O_F$-action as usual, and an $\mathcal O_F$-linear symplectic isomorphism
\[
T_p(\mathcal B)\cong \mathrm{Gr}^{\mathrm{Z}_{p^\infty}}_{-1},
\]
that induces an isomorphism
\[
\mathcal B[p^\infty]\cong \mathrm{Gr}^{\mathrm{Z}_b}_{-1},
\]
by Theorem~\ref{thm:hodge tate boundary} (and \cite[Theorem B]{scholze-weinstein}).
\item[(2')] An $\mathcal O_F$-linear extension
\[
0\to T\to \cG\to \cB\to 0
\]
where $T$ is a torus with cocharacter group $X$. This induces in particular an $\cO_F$-linear map $f_0: X\to \cB$.
\item[(3')] An $\mathcal O_F$-linear splitting of
\[
0\to T[p^\infty]\to \cG[p^\infty]\to \cB[p^\infty]\to 0;
\]
in particular, this induces a lift of $f_0$ to $\tilde{f}_0: X[\tfrac 1p]\to \cB$.
\item[(4')] Away from the boundary, a lift of $\tilde{f}_0$ to a symmetric positive $\tilde{f}: X[\tfrac 1p]\to \cG(C)$. In general, we have a torsor $\mathcal P'$ over $\cO_C$ for the torus $\mathbb S_{\tilde{Z}}$ with cocharacter group generated by $\tfrac 1{p^\infty} X\otimes X$, parametrizing symmetric lifts of $\tilde{f}_0$ to $X[\tfrac 1p]\to \cG$, and the cone decomposition $\Sigma_Z$ gives rise to an embedding $\mathcal P'\subset \mathcal P'_{\Sigma_Z}$. Then the final datum is an $\cO_C$-point of $\mathcal P'_{\Sigma_Z}$ whose special fibre lies in the boundary.
\end{enumerate}

It is clear that (2) and (2') as well as (4) and (4') correspond. But (3) is equivalent to a splitting of $0\to T[p^\infty]\to \cG[p^\infty]\to \cB[p^\infty]\to 0$ plus an isomorphism $\cB[p^\infty]\cong \mathrm{Gr}^{\mathrm{Z}_b}_{-1}$, i.e.~(3') and the extra part of (1'), as desired.
\end{proof}

\begin{lem}\label{check isomorphism} Let $f: X\to Y$ be a map from a quasicompact separated perfectoid space $X$ to a proper diamond $Y$ over $\Spd\ C$ for some complete algebraically closed extension $C$ of $\Q_p$. Assume that for all complete algebraically closed $C'/C$, the map $X(C',\cO_{C'})\to Y(C',\cO_{C'})$ is a bijection. Then $f$ induces an isomorphism $\overline{X}\to Y$ where $\overline{X}$ is the canonical compactification of $X$ over $\Spd\ C$, in the sense of~\cite[Proposition 18.6]{scholze-diamonds}. In particular, if $X/\Spd\ C$ is compactifiable, then $f$ is an open immersion.
\end{lem}

\begin{proof} We may replace $X$ by $\overline{X}$ and assume that $X/\Spd\ C$ is also proper. Then $f$ induces bijections on $(C',C'^+)$-valued points for general $C'^+\subset C'$ containing $\cO_C$ (as these agree with the $(C',\cO_{C'})$-valued points by properness). Now the result follows from \cite[Lemma 11.11]{scholze-diamonds}.
\end{proof}

\subsection{The fibers of the minimal compactification}

\begin{thm}\label{thm:min fibers of pi HT}
The morphism $f^{\tor}$ induces an open immersion
\[
f^*: \mathfrak{Ig}^{b,*}_C\to (\pi_\HT^\ast)^{-1}(x).
\]
of affinoid perfectoid spaces with the same rank $1$ points.
\end{thm}

\begin{proof} The fact that $\mathfrak{Ig}^{b,*}_C$ is affinoid perfectoid follows from 
Lemma~\ref{basic properties of minimal compactifications}. By \cite[Theorem 4.1.1]{scholze}, we know that $(\pi_\HT^{\underline{\ast}})^{-1}(x)$ is an affinoid perfectoid space. By \cite[Theorem 1.17]{bhatt-scholze-prisms}, we see that $(\pi_\HT^\ast)^{-1}(x)$ is affinoid perfectoid as well. The theorem follows from Theorem~\ref{thm:tor fibers of pi HT} after taking global sections, and noting the results of Lemma~\ref{lem:global sections} below. 
\end{proof}

\begin{lem}\label{lem:global sections}\leavevmode
\begin{enumerate}
\item The canonical morphism $\mathfrak{Ig}^{b,\tor}_C\to \mathfrak{Ig}^{b,*}_C$
induces an isomorphism on global sections.
\item The canonical morphism $\mathscr F^\tor := (\pi_\HT^\tor)^{-1}(x)\to \mathscr F^\ast := (\pi_\HT^\ast)^{-1}(x)$
induces an isomorphism on global sections. 
\end{enumerate}
\end{lem}

\begin{proof} The first part is clear (we know it over $k$, thus over $\cO_C/p^\epsilon$ by base change, then we can lift to $\cO_C$, and invert $p$).

For the second part, let $\pi: \cS^{\tor}_{K(p^\infty N)}\to \cS^{*}_{K(p^\infty N)}$ denote the canonical projection. It is enough to prove that 
the natural map of sheaves on $(\mathscr{F}^*)_{\et}$ 
\begin{equation}\label{eq:almost iso 2}
\cO^{+a}_{\mathscr{F}^*}/p^n\to \pi_*\cO^{+a}_{\mathscr{F}^{\tor}}/p^n  
\end{equation}
is an almost isomorphism. The result in the statement of the 
lemma follows from this by passing to the inverse limit over 
$n$, inverting $p$, and taking global sections. 

We now claim that it is enough to prove that the natural map of sheaves on $(\cS^{*}_{K(p^\infty N),C})_{\et}$
\begin{equation}\label{eq:almost iso 1}
\cO^{+a}_{\cS^{*}_{K(p^\infty N),C}}/p^n \to \pi_*\cO^{+a}_{\cS^{\tor}_{K(p^\infty N),C}}/p^n
\end{equation}
is an almost isomorphism. Since the pullback functor $\nu^*$ from the \'etale site to the quasi-pro-\'etale site is fully faithful,
cf.~\cite[Proposition 14.8]{scholze-diamonds}, it is enough to consider the pullbacks of the sheaves to the corresponding quasi-pro-\'etale sites. 

The map $x\to \Fl_C$ is quasi-pro-\'etale, since the point $x$ can be identified with $\varprojlim \cU$, where the limit
runs over a cofinal set of quasi-compact open neighbourhoods $\cU$ of $x$ in $\Fl_{G,\mu}$. 
This implies that the immersions 
$\iota^{\tor}: \mathscr{F}^{\tor}\to \cS^{\tor}_{K(p^\infty N),C}$ and 
$\iota:\mathscr{F}^*\to \cS^{*}_{K(p^\infty N),C}$ are quasi-pro-\'etale as well. 
Therefore, $(\mathscr{F}^{\tor})_{\qproet}$ 
is a slice category of $(\cS^{\tor}_{K(p^\infty N), C})_{\qproet}$ 
and we can identify the sheaves 
\[
(\iota^{\tor})^*(\cO^+_{\cS^{\tor}_{K(p^\infty N), C}}/p^n) = \cO^+_{\mathscr{F}^{\tor}}/p^n 
\]
on $(\cF^{\tor})_{\qproet}$. Similarly, we can identify the sheaves 
\[
\iota^*(\cO^+_{\cS^{*}_{K(p^\infty N),C}}/p^n) = \cO^+_{\mathscr{F}^{*}}/p^n 
\]
on $(\mathscr{F}^*)_{\qproet}$. By the base change result in~\cite[Corollary 16.9]{scholze-diamonds},
we obtain a natural isomorphism of sheaves on $(\mathscr{F}^*)_{\qproet}$. 
\[
\iota^*\circ \pi_*\cO^+_{\cS^{\tor}_{K(p^\infty N), C}}/p^n\toisom 
\pi_{*}\circ (\iota^{\tor})^*\cO^+_{\cS^{\tor}_{K(p^\infty N), C}}/p^n.
\]
We conclude that~\eqref{eq:almost iso 1} implies~\eqref{eq:almost iso 2}. 

It remains to prove the almost isomorphism~\eqref{eq:almost iso 1}.
At each finite level $K(p^m N)$, we have a natural isomorphism 
$\Z/p^n\Z \toisom \pi_{K(p^m N),*}(\Z/p^n\Z)$ of sheaves on $(\cS^{*}_{K(p^m N), C})_{\et}$,
by base change, e.g.~\cite[Corollary 16.10]{scholze-diamonds}, and by the fact that the fibers of the Stein factorisation are geometrically connected. 
We tensor with $\cO^+_{\cS^{*}_{K(p^m N), C}}/p^n$ and apply the 
relative primitive comparison isomorphism in the form of~\cite[Theorem 3.13]{scholzesurvey};
this gives an almost isomorphism 
\[
\cO^{+a}_{\cS^{*}_{K(p^m N), C}}/p^n\toisom \pi_{K(p^m N)*}\cO^{+a}_{\cS^{\tor}_{K(p^m N), C}}/p^n. 
\]
We conclude by taking a direct limit over the finite levels $K(p^m N)$.
\end{proof}

\subsection{Semiperversity}

Let $C$ be any complete algebraically closed extension of $\mathbb Q_p$ with ring of integers $\mathcal O_C$ and residue field $k$.

\begin{thm}\label{thm:semiperversity} Consider
\[
R\pi_{\HT\ast}^\circ \mathbb F_\ell\in D(\Fl_C,\mathbb F_\ell).
\]
Any geometric point $\overline{x}$ of $\Fl_C$ has a cofinal system of affinoid \'etale neighborhoods $U=\Spa(A)\to \Fl_C$ such that, denoting $\mathfrak{U}=\Spf(A^\circ)$, the nearby cycles
\[
R\psi(R\pi_{\HT\ast}^\circ \mathbb F_\ell)|_U\in {}^p D^{\geq d}(\mathfrak{U}_k,\mathbb F_\ell)
\]
are semiperverse, where $d=[F^+:\Q]n^2$ is the dimension of the Shimura variety.
\end{thm}

The proof of the theorem is complicated by the nonproperness of $\pi_{\HT}^\circ$. We resolve this issue by passing to an analogue of the Borel--Serre compactification, which is a compactification that does not change the cohomology. As an algebraic substitute for this, we take the toroidal compactification at $\ell^\infty$-level. Namely, let
\[
j_{N \ell^\infty}: \mathscr S_{K(N\ell^\infty),\Q}\to \mathscr S^\tor_{K(N\ell^\infty),\Q}
\]
denote the open immersion, where the inverse limit over levels $N\ell^m$ is taken in the category of schemes. The following lemma shows that at $\ell^\infty$-level, the compactification has the same cohomology as the open part.

\begin{lem}\label{shimura ell infty} The natural map
\[
\mathbb F_\ell\to Rj_{N \ell^\infty,\ast} \mathbb F_\ell
\]
is an isomorphism.
\end{lem}

\begin{proof} This is a standard consequence of the structure of toroidal compactifications, see for example~\cite[\S 2.7]{pink}. For completeness, we sketch the argument: At each finite level, the boundary of the toroidal compactification is a divisor with normal crossings, so one can compute $Rj_{N\ell^m,\ast} \mathbb F_\ell$ explicitly: The $R^1j_{N\ell^m,\ast} \mathbb F_\ell$ is at each point freely generated by the divisors passing through this point, and the higher $R^i j_{N\ell^m,\ast} \mathbb F_\ell$ are wedge powers. Going up the $\ell^\infty$-tower, the boundary divisors get more and more ramified with transition maps of degree divisible by $\ell$, so the transition maps are zero on $R^i j_{N\ell^m,\ast} \mathbb F_\ell$ for $i>0$, giving the desired result in the colimit over $m$.
\end{proof}

Similarly:

\begin{lem}\label{igusa ell infty} For any Igusa variety $\mathfrak{Ig}^{\mathbb X}$, the restriction map
\[
H^i(\mathfrak{Ig}^{\mathbb X,\tor}_{K(N\ell^\infty)},\mathbb F_\ell)\to H^i(\mathfrak{Ig}^{\mathbb X}_{K(N\ell^\infty)},\mathbb F_\ell)
\]
is an isomorphism.
\end{lem}

\begin{proof} One can assume that $\mathbb X$ is completely slope divisible and reduce to the similar assertion for the non-perfect schemes
\[
\mathrm{Ig}^{\mathbb X}_{m,K(N\ell^\infty)}\subset \mathrm{Ig}^{\mathbb X,\tor}_{m,K(N\ell^\infty)}
\]
that have the same local structure as Shimura varieties, so the same argument as for Shimura varieties works.
\end{proof}

\begin{proof}[Proof of Theorem~\ref{thm:semiperversity}] First, we note that by the Hochschild--Serre spectral sequence, it is enough to show that
\[
R\psi(R\pi_{\HT,\ell^\infty,\ast}^\circ \mathbb F_\ell)|_U\in {}^p D^{\geq d}(\mathfrak{U}_k,\mathbb F_\ell)
\]
where
\[
\pi_{\HT,\ell^\infty}^\circ: \mathcal S_{K(p^\infty N \ell^\infty),C}^\circ\to \Fl_C
\]
is the version with $\ell^\infty$-level structure. Combining Lemma~\ref{igusa ell infty}, Theorem~\ref{thm:open hodge tate fibers} and Corollary~\ref{cohomology fibers}, we see that the natural map
\[
R\pi_{\HT,\ell^\infty,\ast}^\tor \mathbb F_\ell\to R\pi_{\HT,\ell^\infty,\ast}^\circ \mathbb F_\ell
\]
is an isomorphism. From now on, we will work with the left-hand side.

By \cite[Theorem 4.1.1]{scholze}, for a cofinal collection of $U$, the fibre product
\[
\mathcal{S}^{\underline{\ast}}_{K(p^\infty N),U} := \mathcal{S}^{\underline{\ast}}_{K(p^\infty N),C}\times_{\Fl_C} U
\]
is affinoid perfectoid; moreover, it arises via base change from some affinoid $U'=\mathcal{S}^{\underline{\ast}}_{K(p^m N),U}$ that is \'etale over $\mathcal{S}^{\underline{\ast}}_{K(p^m N),C}$, for $m$ large enough.

Write
\[
\mathcal{S}^{\underline{\ast}}_{K(p^\infty N),U}=\mathrm{Spa}(R_{K(p^\infty N),U})
\]
and $U=\Spa(A)$. Consider the map
\[
\Spf(R_{K(p^\infty N),U}^\circ)\to \Spf(A^\circ).
\]
After reduction modulo $p$, this factors over $\Spec(R_{K(p^m N),U}^\circ/p)$ for large enough $m$ (as $A^\circ/p$ is of finite presentation over $\cO_C/p$), and the resulting map of affine schemes satisfies the valuative criterion of properness (cf.~proof of \cite[Proposition 6.1.3]{caraiani-scholze}), so is an integral map. In particular, pushforward preserves ${}^p D^{\geq d}$.

As taking nearby cycles in the setting of adic spaces commutes with pushforwards, we find that it is enough to show that if
\[
g_m^{\infty,\infty}: \mathcal S^{\tor}_{K(p^\infty N \ell^\infty),U}\to U'=\mathcal S^{\underline{\ast}}_{K(p^m N),U}
\]
denotes the projection, then
\[
R\psi(Rg_{m\ast}^{\infty,\infty} \mathbb F_\ell)\in {}^p D^{\geq d}(\Spec(R_{K(p^m N),U}^\circ/p),\mathbb F_\ell).
\]
The \'etale map $U\to \Fl_C$ that we started with here can now be replaced by any \'etale map $U'\to \mathcal S^{\underline{\ast}}_{K(p^m N),C}$ from some affinoid $U'=\Spa(B)$. Let us denote by subscripts $_{U'}$ the base changes of $U'$ to various spaces. Then in particular
\[
g_m^{\infty,\infty}: \mathcal S^{\tor}_{K(p^\infty N \ell^\infty),U} = \mathcal S^{\tor}_{K(p^\infty N \ell^\infty),U'}\to \mathcal S^{\underline{\ast}}_{K(p^m N),U}=U'
\]
and we want to show
\[
R\psi(Rg_{m\ast}^{\infty,\infty} \mathbb F_\ell)\in {}^p D^{\geq d}(\Spec(B^\circ/p),\mathbb F_\ell).
\]

If we let for any $m_1\geq m$, $m_2\geq 0$
\[
g_m^{m_1,m_2}: \mathcal S^{\tor}_{K(p^{m_1} N \ell^{m_2}),U'}\to U'
\]
be the corresponding projection, and
\[
j^{m_1,m_2}: \mathcal S_{K(p^{m_1} N \ell^{m_2}),U'}\hookrightarrow \mathcal S^{\tor}_{K(p^{m_1} N \ell^{m_2}),U'}
\]
the open immersion (this map arises via an \'etale base change of the analytification of an open embedding of smooth schemes over $C$), then on the one hand
\[
Rg^{\infty,\infty}_{m\ast} \mathbb F_\ell = \varinjlim_{m_1,m_2} Rg_{m\ast}^{m_1,m_2} \mathbb F_\ell\ ,
\]
while on the other hand for any fixed $m_1$, the natural map
\[
\varinjlim_{m_2} Rg_{m\ast}^{m_1,m_2} \mathbb F_\ell\to \varinjlim_{m_2} Rg_{m\ast}^{m_1,m_2} Rj^{m_1,m_2}_\ast \mathbb F_\ell
\]
is an isomorphism. For the latter claim, it suffices to check the analogous assertion on schemes (as all operations arise via \'etale base change from similar operations between analytifications of schemes of finite type, where they agree with the operations on the scheme level), and then it follows from $Rj^{m_1,\infty}_\ast \mathbb F_\ell = \mathbb F_\ell$ in the setting of schemes, which is Lemma~\ref{shimura ell infty}.

Finally, it suffices to see that for all $m_1,m_2$,
\[
R\psi(Rg_{m\ast}^{m_1,m_2} Rj^{m_1,m_2}_\ast \mathbb F_\ell)\in {}^p D^{\geq d}(\Spec(B^\circ/p),\mathbb F_\ell).
\]
But the sheaf $Rg_{m\ast}^{m_1,m_2} Rj^{m_1,m_2}_\ast \mathbb F_\ell$ is simply the pushforward of $\mathbb F_\ell$ under the composite of the open immersion $\mathcal S_{K(p^{m_1} N \ell^{m_2}),U'}\hookrightarrow \mathcal S^{\underline{\ast}}_{K(p^{m_1} N \ell^{m_2}),U'}$ and the finite map $\mathcal S^{\underline{\ast}}_{K(p^{m_1} N \ell^{m_2}),U'}\to U'$, which can be computed in the same way on the level of schemes (choosing an algebraization of the \'etale map $U'\to \mathcal S^{\underline{\ast}}_{K(p^m N),C}$), and lies in ${}^p D^{\geq d}$ in the scheme setting. Now the result follows from $t$-exactness of nearby cycles for the perverse $t$-structure, cf.~\cite{illusie-autour} (and compatibility of nearby cycles between schemes and adic spaces, \cite[Theorem 3.5.13]{huber}).
\end{proof}

\newpage

\section{The cohomology of Igusa varieties and automorphic representations}\label{igusa}

In this section, we analyze the $\overline{\mathbb Q}_\ell$-cohomology of Igusa varieties as a virtual representation, and describe it in terms of automorphic representations. As an application, we can prove the existence of associated Galois representations when the cohomology is concentrated in one degree. 

The methods and results of this section are parallel to those of Section 5 of~\cite{caraiani-scholze}, but we will replace the conditions imposed there on the 
ramification of primes in $F$ with the condition that the level be sufficiently small at ``bad places''. In~\cite{shin-stable}, Sug Woo Shin derived a formula for the 
alternating sum of the compactly supported cohomology of Igusa varieties 
as a sum of stable orbital integrals for $G$ and its elliptic endoscopic groups. 
If $F$ is not imaginary quadratic and the level $N$ is sufficiently divisible, we can reinterpret this formula as the geometric side of the twisted trace formula. 
We then compare it to the spectral side of the twisted trace formula to compute it in terms of virtual representations of $G(\mathbb{A}_f^p)\times J_b(\mathbb{Q}_p)$.

Throughout this section, we will consider Igusa varieties $\mathrm{Ig}^b$ with infinite level at $p$, whose cohomology gives rise to representations of $J_b(\mathbb{Q}_p)$. The notation will distinguish between the case when we have some finite tame level $K(N)\subset G(\A_{f}^p)$ or whether we consider the inverse system of Igusa varieties over all possible tame levels. 

\subsection{Statements}

As usual, we fix a prime $p$ that is unramified in $F$. In this section, we make the following further assumptions:
\begin{enumerate}
\item The CM field $F$ contains an imaginary quadratic field $F_0$ in which $p$ is split;
\item The totally real subfield $F^+$ is not equal to $\Q$.
\end{enumerate}
We also fix a set $S$ of places of $\Q$ containing $\infty$ and all primes dividing $p\ell\Delta_F$. For technical reasons, we also fix a character $\varpi: \A_{F_0}^\times/F_0^\times\to \C^\times$ extending the quadratic character of $\A^\times/\Q^\times\to \{\pm 1\}$ corresponding to $F_0$ (via the embedding $\A^\times/\Q^\times\hookrightarrow \A_{F_0}^\times/F_0^\times$), and add to $S$ all primes above which $\varpi$ is ramified.

Throughout this section we also fix an isomorphism $\iota_\ell: \mathbb{\overline Q}_\ell\toisom \C$.

Let $\Spl_{F_0/\Q}$ be the set of primes that split in $F_0$, and let
\[
\mathbb T^S = \prod_{q\not\in S,q\in \Spl_{F_0/\Q}} \mathbb Z[G(\Q_q)//G(\Z_q)]
\]
be the associated spherical Hecke algebra. If $v$ is a place of $F$ dividing a prime $q\not\in S$ that splits in $F_0$, then $v$ determines a prime $\mathfrak q|q$ of $F_0$, and
\[
G(\Q_q) = \GL_{2n}(F_v)\times \prod_{w|\mathfrak q,w\neq v} \GL_{2n}(F_w)\times \Q_q^\times.
\]
For any $i=1,\ldots,2n$, we define $T_{i,v}\in \mathbb T^S$ as the Hecke operator corresponding to the $G(\Z_q)$-double coset
\[
\GL_{2n}(\cO_{F_v})\mathrm{diag}(\underbrace{\varpi_v,\ldots,\varpi_v}_i,1,\ldots,1)\GL_{2n}(\cO_{F_v}) \times\prod_{w\mid \mathfrak{p}, w\neq v}\GL_{2n}(\cO_{F_w})\times \Z_q^\times
\]
inside $G(\Q_q)$.

Fix $b\in B(G_{\Q_p},\mu^{-1})$ and fix a $p$-divisible group $\mathbb X=\mathbb X_b$ with $G$-structure over $\overline{\mathbb F}_p$ in the isogeny class given by $b$; we assume that $\mathbb X$ is completely slope divisible. Note that if $v$ is a prime of $F$ dividing $p$, inducing a prime $\mathfrak p|p$ of $F_0$, then we have a decomposition
\[
J_b(\Q_p)\cong J_{b_v}(F_v)\times \prod_{w|\mathfrak p,w\neq v} J_{b_w}(F_w)\times \Q_p^\times
\]
where $b_v\in B(\GL_{2n,F_v})$ and $b_w\in B(\GL_{2n,F_w})$ are the elements induced by $b$. If $\pi$ is an irreducible representation of $J_b(\Q_p)$, we will write
\[
\pi = \pi_v\otimes \bigotimes_{w|\mathfrak p,w\neq v} \pi_w\otimes \chi
\]
for the associated decomposition. Note that $J_{b_v}(F_v)$ is an inner form of a Levi subgroup of $\GL_{2n}(F_v)$. As such, any irreducible smooth $\overline{\mathbb Q}_\ell$ representation $\pi_v$ of $J_{b_v}(F_v)$ has a semisimple $L$-parameter
\[
\sigma_{\pi_v}: W_{F_v}\to \GL_{2n}(\overline{\mathbb Q}_\ell),
\]
using the results of Badulescu, \cite{badulescu}, on the Jacquet--Langlands correspondence for all inner forms of general linear groups.

\begin{rem} Recall that the semisimple Langlands parameter associated to a Frobenius semisimple Weil-Deligne representation
\[
W_{F_v}\times \SL_2(\C)\to \GL_m(\C)
\]
is the restriction along the embedding
\[
W_{F_v}\to W_{F_v}\times \SL_2(\C): w\mapsto (w,\mathrm{diag}(|w|^{1/2},|w|^{-1/2}))
\]
where $|\cdot|: W_{F_v}\to \R_{>0}$ is the map sending a Frobenius to $q_v$. This has the property that nontrivial extensions can only exist between irreducible smooth representations with the same semisimple Langlands parameter; in fact the semisimple Langlands parameter gives exactly a point of the Bernstein center. Also, it is compatible with normalized parabolic induction.
\end{rem}

We fix a level $N\geq 3$ that is only divisible by primes in $S_f^p := S_{\mathrm{fin}}\setminus \{p\}$. The goal of this section is to describe the virtual representation
\[
[H_c(\mathrm{Ig}^b_{K(N)},\overline{\mathbb Q}_\ell)] := \sum_i (-1)^i [\varinjlim_{m} H_c^i(\mathrm{Ig}^b_{m,K(N)},\overline{\mathbb Q}_\ell)]
\]
of $J_b(\Q_p)\times \mathbb T^S$, under an additional assumption on $N$. As a virtual representation of $J_b(\Q_p)\times \mathbb T^S$, we can write
\[
[H_c(\mathrm{Ig}^b_{K(N)},\overline{\mathbb Q}_\ell)] = \sum_{j\in J} n_j \pi_j\otimes \psi_j
\]
over some countable index set $J$, where each $\pi_j$ is an irreducible smooth representation of $J_b(\Q_p)$, $\psi_j: \mathbb T^S\to \overline{\mathbb Q}_\ell$ is a character, and we assume that all $n_j\neq 0$ and all $\pi_j\otimes \psi_j$ are pairwise distinct.

The goal of this section is to prove the following theorem.

\begin{thm}\label{thm:igusa computation}  Assume that the level $N$ is divisible by some $N_0\geq 3$ as in~Remark~\ref{choice of N0}; in particular, $N_0$ is only divisible by primes in $S_f^p := S_{\mathrm{fin}}\setminus \{p\}$. 

For each $j\in J$, there is a continuous semisimple Galois representation
\[
\rho_j: \mathrm{Gal}(\overline{F}/F)\to \GL_{2n}(\overline{\mathbb Q}_\ell)
\]
that is almost everywhere unramified, and for all primes $v$ of $F$ dividing a prime $q\not\in S$ that splits in $F_0$, the representation $\rho_j$ is unramified at $v$, and the characteristic polynomial of $\rho_j(\Frob_v)$ is given by
\[
X^{2n} - \psi_j(T_{1,v})X^{2n-1}+\dots+ (-1)^iq_v^{i(i-1)/2}\psi_j(T_{i,v})X^{2n-i}+\dots + q_v^{n(2n-1)}\psi_j(T_{2n,v}),
\]
where $q_v$ is the cardinality of the residue field at $v$.

Moreover, for all primes $v$ dividing $p$, the semisimple Langlands parameter of $\pi_{j,v}|\cdot|^{\frac 12-n}$ is given by the semisimple Langlands parameter induced by $\rho_j|_{\mathrm{Gal}(\overline{F}_v/F_v)}$.
\end{thm}

In particular, we get the following corollary.

\begin{cor}\label{cor:existence Galois single degree} Let $N$ be divisible by $N_0$ as above. Assume that $\mathfrak m\subset\mathbb T^S$ is a maximal ideal such that
\[
H^i(\mathrm{Ig}^b_{K(N)},\mathbb F_\ell)_{\mathfrak m}\neq 0
\]
in exactly one degree. Then there exists a continuous semisimple Galois representation
\[
\overline\rho_{\mathfrak m}: \mathrm{Gal}(\overline{F}/F)\to \GL_{2n}(\overline{\mathbb F}_\ell)
\]
such that for all primes $v$ of $F$ dividing a prime $q\not\in S$ that splits in $F_0$, the representation $\overline\rho_{\mathfrak m}$ is unramified at $v$, and the characteristic polynomial of $\overline\rho_{\mathfrak m}(\Frob_v)$ is given by the reduction of
\[
X^{2n} - T_{1,v}X^{2n-1}+\dots+ (-1)^iq_v^{i(i-1)/2}T_{i,v}X^{2n-i}+\dots + q_v^{n(2n-1)}T_{2n,v}
\]
modulo $\mathfrak m$, where $q_v$ is the cardinality of the residue field at $v$.

Moreover, if $p$ splits completely in $F$ and for all $v|p$, the representation $\overline\rho_{\mathfrak m}$ is unramified with the eigenvalues $\{\alpha_{1,v},\ldots,\alpha_{2n,v}\}$ of $\overline\rho_{\mathfrak m}(\Frob_v)$ satisfying $\alpha_{i,v}\neq p\alpha_{j,v}$ for all $i\neq j$, then $b$ is ordinary.
\end{cor}

\begin{proof} Recall the involution $\iota: \mathbb{T}^S\to \mathbb{T}^S$, as in~\cite[\S 2.2.11]{10Authors} and set $\mathfrak{m}^\vee:=\iota(\mathfrak{m})$. By the assumption of concentration in one degree and Poincar\'e duality, it follows that $H^i_c(\mathrm{Ig}^b_{K(N)},\mathbb Z_\ell)_{\mathfrak m^\vee}$ is also concentrated in one degree and torsion-free, so there must be some $j\in J$ such that $\psi_j: \mathbb T^S\to \overline{\mathbb Q}_\ell$ reduces to $\mathfrak m^\vee$. Then we can apply Theorem~\ref{thm:igusa computation}, define $\overline\rho_{\mathfrak m^\vee}$ as the semisimplification of the reduction of $\rho_j$ modulo $\ell$, and set 
\[
\overline{\rho}_{\mathfrak m}: = \overline{\rho}_{\mathfrak{m}^\vee}|\mathrm{Art}_F^{-1}|^{1-2n}. 
\]

For the final statement, assume that $b$ is not ordinary. Then $J_b$ is a nontrivial inner form of a Levi subgroup of $G_{\Q_p}$. In particular, there is some prime $v$ dividing $p$ such that $J_{b_v}$ is a nontrivial inner form of a Levi subgroup of $\GL_{2n,F_v}=\GL_{2n,\Q_p}$. It follows that there are no irreducible smooth representations $\pi_v$ of $J_{b_v}(F_v)$ such that $\sigma_{\pi_v}$ is a direct sum of characters $\chi_1\oplus \ldots \oplus \chi_{2n}$ such that $\chi_i/\chi_j$ is not the cyclotomic character for any $i\neq j$. (Indeed, such $L$-parameters correspond to generic principal series representations which do not transfer to any nontrivial inner form, cf.~\cite[Lemma 5.4.3]{caraiani-scholze}.)

Now we distinguish two cases. If $p\not\equiv 1\mod \ell$, then we claim that $\rho_j$ must be unramified at $v$; necessarily no two Frobenius eigenvalues can have ratio $p$ (as they do not have ratio $p$ modulo $\ell$ by assumption), leading to the desired contradiction. To show that $\rho_j|_{\mathrm{Gal}(\overline{F}_v/F_v)}$ is unramified, look at the deformation theory of the unramified representation $\overline{\rho}_j|_{\mathrm{Gal}(\overline{F}_v/F_v)}$ that is a direct sum of unramified characters with eigenvalues $\alpha_{1,v},\ldots,\alpha_{2n,v}$ and check, using the condition $\alpha_{i,v}\neq p\alpha_{j,v}$, that the obstructions, deformations, and automorphisms are the same whether regarded as representations of $\mathrm{Gal}(\overline{\mathbb Q}_p/\Q_p)$ or $\mathrm{Gal}(\overline{\mathbb F}_p/\mathbb F_p)$.\footnote{We thank Koshikawa for pointing out this argument.}

On the other hand, if $p\equiv 1\mod \ell$, then the similar deformation theoretic argument implies that $\rho_j|_{\mathrm{Gal}(\overline{F}_v/F_v)}$ is a direct sum of characters, and again their ratio cannot be the cyclotomic character as this does not happen modulo $\ell$, cf.~\cite[Lemma 6.2.2]{caraiani-scholze}.
\end{proof}

\subsection{Unitary groups and endoscopic triples}\label{sec: unitary groups}

If $\vec{n}=(n_i)_{i=1}^s$ is a vector with entries positive integers, 
one can define a quasi-split group $G_{\vec{n}}$ over $\mathbb{Q}$ 
as in~\cite[Section 3.1]{shin-galois}. Define $\GL_{\vec{n}}:=\prod_{i=1}^s \GL_{n_i}$ 
and let $i_{\vec{n}}: \GL_{\vec{n}}\hookrightarrow \GL_{(\sum_i n_i)}$ be the natural embedding. 
Let \[\Phi_{\vec{n}}:=i_{\vec{n}}(\Phi_{n_1},\dots,\Phi_{n_s}),\] 
where $\Phi_m$ is the matrix in $\GL_m$ with entries 
$(\Phi_m)_{ij}=(-1)^{i+1}\delta_{i,m+1-j}$.
Then $G_{\vec{n}}$ is the algebraic group over $\mathbb{Q}$ sending a $\mathbb{Q}$-algebra $R$ to 
\[G_{\vec{n}}(R)=\{(\lambda, g)\in R^\times \times 
\GL_{\vec{n}}(F\otimes_{\mathbb{Q}}R)| g\cdot\Phi_{\vec{n}}\cdot {}^tg^c=\lambda\Phi_{\vec{n}}\}.\] 
Since $G_{\mathbb Q}$ is quasi-split, we have $G_{\mathbb Q}\cong G_{2n}$. 

We identify the Langlands dual group of $G_{\vec{n}}$ with \[\widehat{G}_{\vec{n}}=\C^\times \times \prod_{\sigma\in \Phi^+}\GL_{\vec{n}}(\C),\] where
$\Phi^+$ is the set of embeddings $F\hookrightarrow \C$ above a fixed embedding $\tau_0:F_0 \hookrightarrow \C$. Let $\mathscr{E}^\mathrm{ell}(G)$ be the set of 
isomorphism classes of elliptic endoscopic triples for $G$. By \cite[Section 3.2]{shin-galois}, a set of representatives for these isomorphism classes is given by
\[\left\{(G_{2n},s_{2n},\eta_{2n})\right\}\cup 
\left\{(G_{n_1,n_2},s_{n_1,n_2},\eta_{n_1,n_2})|n_1+n_2=2n,n_1\geq n_2\geq 0\right\},\]  
where $(n_1,n_2)$ may be excluded if all of $n_1$, $n_2$ and $[F^+:\Q]$ are odd numbers (see condition 7.4.3 of~\cite{kottwitz-cuspidal}).  
Here, $s_{2n}=1\in \widehat{G}_n, s_{n_1,n_2}=(1,(I_{n_1},-I_{n_2}))\in \widehat{G}_{n_1,n_2}$, 
$\eta_{2n}:\widehat{G}_{2n}\to \widehat{G}_{2n}$ is the 
identity map and $\eta_{n_1,n_2}:\widehat{G}_{n_1,n_2}\to \widehat{G}_{2n}$ 
is the natural embedding induced by $\GL_{n_1}\times \GL_{n_2}\hookrightarrow \GL_{2n}$.

We can extend the map $\eta_{n_1,n_2}:\widehat{G}_{n_1,n_2}\to \widehat{G}_{2n}$ on Langlands dual 
groups to an $L$-morphism of $L$-groups $\tilde{\eta}_{n_1,n_2}: {}^LG_{n_1,n_2}\to {}^LG_{2n}$. 
Then the $L$-group $^LG_{\vec{n}}:=\widehat{G}_{\vec{n}}\rtimes W_{\Q}$ is the semidirect product defined by
\[w(\lambda, (g_{\sigma})_\sigma)w^{-1} = (\lambda, (g_{w^{-1}\sigma})_{\sigma}),\ \mathrm{if}\ w\in W_{F_0}\]
and 
\[w(\lambda, (g_{\sigma})_\sigma)w^{-1} = \left(\lambda \prod_{\sigma\in \Phi^+}\det g_{\sigma}, (\Phi_{\vec{n}}\ ^tg^{-1}_{cw^{-1}\sigma}\Phi_{\vec{n}}^{-1})_{\sigma}\right) \ \mathrm{if}\ w\not\in W_{F_0}.\] 

We will use the same formula for $\tilde{\eta}_{n_1,n_2}$ as in Section 3.2 of~\cite{shin-galois}, but we recall the
precise definition here because it will be important when we attach Galois representations
to systems of Hecke eigenvalues. As in the beginning of this section, we fix $\varpi: \mathbb{A}^\times_{F_0}/F_0^\times \to \mathbb{C}^\times$
such that $\varpi |_{\mathbb{A}^\times/\mathbb{Q}^\times}$ is the quadratic character corresponding to $F_0$.
Using the map $\mathrm{Art}_{F_0}$, we can also view $\varpi$ as a character
$W_{F_0}\to \mathbb{C}^\times$. 
Set $\epsilon:\mathbb{Z}\to \{0,1\}$
to be the unique map such that $\epsilon(m)\cong m\pmod{2}$. Let $w^*$ be a fixed element in
$W_{\mathbb{Q}}\setminus W_{F_0}$. 
We extend $\eta_{n_1,n_2}$ to an $L$-morphism $\tilde{\eta}_{n_1,n_2}$ by
\[w\in W_{F_0}\mapsto \left(\varpi(w)^{-N(n_1,n_2)},
\left(\begin{smallmatrix}\varpi(w)^{\epsilon(n_1)}\cdot I_{n_1}&0\\0&\varpi(w)^{\epsilon(n_2)}\cdot I_{n_2}\end{smallmatrix}\right)_{\sigma\in \Phi^+}\right)\rtimes w\]
\[w^*\mapsto \left(a_{n_1,n_2},(\Phi_{n_1,n_2}\Phi_{2n}^{-1})_{\sigma\in \Phi^+}\right)\rtimes w^*,\]
where $N(n_1,n_2)=[F^+:\Q](n_1\epsilon(n_1)+n_2\epsilon(n_2))/2\in\Z$ and $a_{n_1,n_2}$ is a square root of $(-1)^{N(n_1,n_2)}$. Using the definition of the $L$-groups $^LG_{\vec{n}}$ one 
can check that $\tilde{\eta}_{n_1,n_2}$ is indeed an $L$-morphism. We remark
that this is an adaptation to unitary similitude groups of the $L$-morphism
defined in Section 1.2 of~\cite{rogawski}.

\subsection{A stable trace formula}\label{sec: stable trace formula}

Let $\phi\in C_c^\infty(J_b(\mathbb{Q}_p)\times G(\mathbb{A}_f^p))$ be a function of the form
\[
\phi = \phi_p\cdot \phi_{S_f^p}\cdot \phi^S
\]
where $\phi_p\in C_c^\infty(J_b(\Q_p))$,
\[
\phi_{S_f^p}\in C_c^\infty(G(\A_{S_f^p}))
\]
is simply the characteristic function of the principal congruence subgroup $K(N)$ where $N$ is sufficiently divisible as defined later, and
\[
\phi^S\in \cH^\ur(G(\mathbb A^S))
\]
is in the unramified Hecke algebra. We assume moreover that $\phi$ is acceptable in the sense of Definition 6.2 of~\cite{shin-igusa}.

For $G_{\vec{n}}\in \mathscr{E}^\mathrm{ell}(G)$ as in Section~\ref{sec: unitary groups}, we can define the transfer of $\phi$ to $\phi^{\vec{n}}\in C_c^\infty(G_{\vec{n}}(\mathbb{A}_f))$ as follows. 

\begin{defn}\label{definition of transfer} Let $\phi^{\vec{n}}:=\phi^{\vec{n}}_p\cdot \phi^{\vec{n}}_{S_f^p}\cdot \phi^{\vec{n},S}\cdot \phi^{\vec{n}}_\infty$, where:
\begin{enumerate}
\item $\phi^{\vec{n}}_p$ is constructed in Section 6 of~\cite{shin-stable} and is the same function used in Sections 5.2 and 5.4 of~\cite{caraiani-scholze}.
\item $\phi^{\vec{n}}_{S_f^p}$ is the Langlands-Shelstad transfer of $\phi_{S_f^p}$.
\item $\phi^{\vec{n},S}\in \cH^{\ur}(G_{\vec{n}}(\mathbb A^S)$ is given by $\tilde{\eta}_{\vec{n}}^*(\phi^S)$, where  
\[\tilde{\eta}_{\vec{n}}^*:\cH^\ur(G(\mathbb{A}^S))\to \cH^\ur(G_{\vec{n}}(\mathbb{A}^S))\]
is induced by the $L$-morphism $\tilde{\eta}_{\vec{n}}$ defined in Section~\ref{sec: unitary groups}.
\item $\phi^{\vec{n}}_{\infty}$ is a linear combination of Euler-Poincar\'e functions, the same as in Section 5.2 of~\cite{caraiani-scholze}.
\end{enumerate}
\end{defn}

The following is the main result of~\cite{shin-stable}.

\begin{thm}\label{stable trace formula} If $\phi\in C_c^\infty(J_b(\mathbb{Q}_p)\times G(\mathbb{A}_{S_f^p}))\otimes \cH^\ur(G(\mathbb A^S))$ is as above, then
\[\mathrm{tr}\left(\phi|H_c(\Ib,\mathbb{\overline Q}_\ell) \right)=\sum_{G_{\vec{n}}}\iota(G,G_{\vec{n}})ST_e^{G_{\vec{n}}}(\phi^{\vec{n}}),\] 
where $G_{\vec{n}}$ runs over the set of representatives of 
$\mathscr{E}^{\mathrm{ell}}(G)$ and $\phi^{\vec{n}}$ is 
obtained from $\phi$ as in Definition~\ref{definition of transfer}.
\end{thm}

\subsection{Base change and compatibility of transfers}\label{sec: base change}
Let $\mathbb{G}_{\vec{n}}:=\mathrm{Res}_{F_0/\mathbb{Q}}(G_{\vec{n}}\times_{\mathbb{Q}}F_0)$. The Langlands dual group $\widehat{
\mathbb{G}}_{\vec{n}}$ can be identified with \[\widehat{\mathbb{G}}_{\vec{n}}=\C^\times \times \C^\times\times \prod_{\sigma\in \Phi}\GL_{\vec{n}}(\C),\] where $\Phi$ is the set of embeddings $F\hookrightarrow \C$. We can describe the $L$-group $^L\mathbb{G}_{\vec{n}}$ as the semidirect product $\widehat{G}_{\vec{n}}\rtimes W_{\mathbb Q}$, where
\[w(\lambda_+, \lambda_{-},(g_{\sigma})_{\sigma\in \Phi})w^{-1}= (\lambda_+, \lambda_{-},(g_{w^{-1}\sigma})_{\sigma\in \Phi}) \ \mathrm{if}\ w\in W_{F_0}\]
and if $w\not\in W_{F_0}$
\[w(\lambda_+,\lambda_-,(g_\sigma)_{\sigma\in \Phi})w^{-1}= \left(\lambda_{-}\prod_{\sigma\in \Phi\setminus \Phi^+}\det g_{\sigma},\lambda_{+}\prod_{\sigma\in \Phi^+}\det g_{\sigma},(\Phi_{\vec{n}}\ ^tg^{-1}_{cw^{-1}\sigma}\Phi_{\vec{n}}^{-1})_{\sigma\in \Phi}\right).\] 
One can define $L$-morphisms $BC_{\vec{n}}:{}^L G_{\vec{n}}\to {}^L\mathbb{G}_{\vec{n}}$ and $\tilde{\zeta}_{n_1,n_2}: {}^L\mathbb{G}_{n_1,n_2}\to {}^L\mathbb{G}_{2n}$ and there is a commutative diagram of $L$-morphisms 
\begin{equation}\label{commutative diagram of L-morphisms}\xymatrix{^LG_{n_1,n_2}\ar[d]^{BC_{n_1,n_2}}\ar[r]^{\tilde \eta_{n_1,n_2}}&^LG_{2n}\ar[d]^{BC_{2n}}\\^L\mathbb{G}_{n_1,n_2}\ar[r]^{\tilde\zeta_{n_1,n_2}}&^L\mathbb{G}_{2n}.}\end{equation}
We give the precise definition of $\tilde{\zeta}_{n_1,n_2}$, as it will be important for explicitly defining the transfer of functions and representations at unramified places, which in turn will determine the Galois representations. On the level of dual groups $\widehat{\mathbb G}_{\vec{n}}$, the morphism is induced by the embedding $\GL_{\vec{n}}\hookrightarrow \GL_{2n}$. We extend this to an $L$-morphism by 
\[w\in W_{F_0}\mapsto \left(\varpi(w)^{-N(n_1,n_2)}, \varpi(w)^{-N(n_1,n_2)},
\left(\begin{smallmatrix}\varpi(w)^{\epsilon(n_1)}\cdot I_{n_1}&0\\0&\varpi(w)^{\epsilon(n_2)}\cdot I_{n_2}\end{smallmatrix}\right)_{\sigma\in \Phi}\right)\rtimes w\]
\[w^*\mapsto \left(a_{n_1,n_2},a_{n_1,n_2},(\Phi_{n_1,n_2}\Phi_{2n}^{-1})_{\sigma\in \Phi}\right)\rtimes w^*.\] The $L$-morphism $BC_{\vec{n}}$ is defined by 
\[(\lambda,(g_{\sigma})_{\sigma\in \Phi^+})\rtimes w\mapsto (\lambda, \lambda, (g_{\sigma})_{\sigma\in \Phi^+},  (g_{c\sigma})_{\sigma\in \Phi\setminus \Phi^+}).\] 
One can check the commutativity of the diagram~\eqref{commutative diagram of L-morphisms} by a direct computation. 

In this section, we review the associated base change for the groups $G_{\vec{n}}$ and $\mathbb{G}_{\vec{n}}$ as well as the twisted trace formula. Recall that we have fixed a finite set of bad places $S$. Everything is identical with Section 5.3 of~\cite{caraiani-scholze}, except for the base change transfer of functions at primes in $S_f^p$.

More precisely, if $v$ is a finite place of $\mathbb{Q}$ such that $v\not\in S$, then the dual map to the $L$-morphism $BC_{\vec{n}}$ defines the transfer \[BC_{\vec{n}}^*:\cH^{\mathrm{ur}}(\mathbb{G}_{\vec{n}}(\mathbb{Q}_v))\to \cH^{\mathrm{ur}}(G_{\vec{n}}(\mathbb{Q}_v)),\] (Case 1) of Section 4.2 of~\cite{shin-galois}. We remark that, if $v\not\in S$ then we can define a transfer corresponding to the $L$-morphism $\tilde{\zeta}_{\vec{n}}$ via
\[\tilde{\zeta}^*_{\vec{n}}:\cH^{\mathrm{ur}}(\mathbb{G}_{2n}(\mathbb{Q}_v))\to \cH^{\mathrm{ur}}(\mathbb{G}_{\vec{n}}(\mathbb{Q}_v)).\] 
Moroever the diagram~\ref{commutative diagram of L-morphisms} gives rise to a commutative diagram of transfers
\begin{equation}\label{commutative diagram transfer of functions}\xymatrix{\cH^\ur(\mathbb{G}_{2n}(\Q_v))\ar[d]^{BC^*_n}\ar[r]^{\tilde{\zeta}^*_{n_1,n_2}}&\cH^\ur(\mathbb{G}_{n_1,n_2}(\mathbb{Q}_v))\ar[d]^{BC^*_{n_1,n_2}} \\ \cH^\ur(G_{2n}(\mathbb{Q}_v))\ar[r]^{\tilde{\eta}^*_{n_1,n_2}}& \cH^\ur(G_{n_1,n_2}(\mathbb{Q}_v))}.\end{equation}
We note that we can also transfer unramified representations and we have the corresponding commutative diagram
\begin{equation}\label{commutative diagram transfer of representations}\xymatrix{\mathrm{Rep}^\ur(G_{n_1,n_2}(\Q_v))\ar[d]^{BC_{n_1,n_2,*}}\ar[r]^{\tilde \eta_{n_1,n_2,*}}&\mathrm{Rep}^\ur(G_{2n}(\Q_v))\ar[d]^{BC_{2n,*}}\\ \mathrm{Rep}^\ur(\mathbb{G}_{n_1,n_2}(\Q_v))\ar[r]^{\tilde\zeta_{n_1,n_2,*}}&\mathrm{Rep}^\ur(\mathbb{G}_{2n}(\Q_v))}.\end{equation} 
Both commutative diagrams come from equation (4.18) of~\cite{shin-galois}.
Moreover, transfer of functions and transfer of representations are, in some sense, inverse operations, so that we have
\[\mathrm{tr}\ \pi(\tilde{\eta}^*_{n_1,n_2}(\phi_v))=\mathrm{tr}\ \tilde{\eta}_{n_1,n_2,*}(\pi)(\phi_v)\]
and similarly for $BC_{\vec{n}}$ and $\tilde{\zeta}_{\vec{n}}$.

\begin{remark}\label{explicit transfer} We give explicit formulas for $\tilde{\zeta}^*_{\vec{n}}$ and for $\tilde{\zeta}_{\vec{n},*}$, using the fact that both $\mathbb{G}_{2n}$ and $\mathbb{G}_{n_1,n_2}$ are products of (restrictions of scalars of) general linear groups. Let $Q_{n_1,n_2}$ be a parabolic subgroup of $\mathbb{G}_{2n}$ with Levi subgroup $\mathbb{G}_{n_1,n_2}$. Recall that we have chosen a character $\varpi : \A^\times_{F_0}/F_0^\times\to \C^\times$. Let $\chi_{\varpi}$ be the character $\mathbb{G}_{n_1,n_2}(\A)\to \C^\times$ defined by 
\[(\lambda, g_1,g_2)\in \A_{F_0}^\times \times \GL_{n_1}(\A_F)\times \GL_{n_2}(\A_F)\mapsto \varpi\left(\lambda^{-N(n_1,n_2)}\prod_{i=1}^2N_{F/F_0}(\det(g_i))^{\epsilon(n_i)}\right).\]
Given a place $v$ of $\Q$ and a function $f_v\in C^\infty_c(\mathbb{G}_{2n}(\Q_v))$, we let $f_v^{Q_{n_1,n_2}}$ be the constant term of $f_v$ along $Q_{n_1,n_2}$ and we set $\tilde{\zeta}^*_{n_1,n_2}(f_v)=f_v^{Q_{n_1,n_2}}\cdot \chi_{\varpi, v}$. Given $\Pi_v \in \mathrm{Rep}(\mathbb{G}_{n_1,n_2}(\Q_v))$, define $\tilde\zeta_{n_1,n_2,*}(\Pi_v):=\mathrm{n-Ind}_{Q_{n_1,n_2}}^{\mathbb{G}_{2n}}(\Pi_v\otimes \chi_{\varpi, v})$. The compatibility $\mathrm{tr}\Pi_v(\tilde\zeta^*_{n_1,n_2}(f_v))=\mathrm{tr}\tilde\zeta_{n_1,n_2,*}(\Pi_v)(f_v)$ follows from Lemma 3.3 of~\cite{shin-galois}. 
\end{remark}

If $v=p$, then it splits in $F_0$ and Section 4.2 of~\cite{shin-galois}, (Case 2), constructs a $BC$-transfer $\phi_p\in C_c^\infty(G_{\vec{n}}(\mathbb{Q}_p))$ of $f_p\in C_c^\infty(\mathbb{G}_{\vec{n}}(\mathbb{Q}_p))$. Moreover, in this case one can check directly that $BC^*_{\vec{n}}$ is surjective. 

At $\infty$, the transfer is defined in Section 4.3 of~\cite{shin-galois}. Let $\xi$ be an irreducible algebraic representation of $(G_{\vec{n}})_{\mathbb{C}}$, giving rise to the representation $\Xi$ of $(\mathbb{G}_{\vec{n}})_{\mathbb{C}}$ which is just $\Xi:=\xi\otimes \xi$. Recall that $\phi_{G_{\vec{n}},\xi}$ is the Euler-Poincar\'e function for $\xi$. Associated to $\Xi$, Labesse defined a twisted analogue of the Euler-Poincar\'e function, a Lefschetz function $f_{\mathbb{G}_{\vec{n}},\Xi}$~\cite{labesse-lefschetz}. The discussion on page 24 of~\cite{shin-galois} implies that $f_{\mathbb{G}_{\vec{n}},\Xi}$ and $\phi_{G_{\vec{n}},\xi}$ are $BC$-matching functions.

At places $v\in S_f^p$, we will be less explicit. We simply note that Theorem 3.3.1 of~\cite{labesse} guarantees the existence of a base-change transfer: given $f^{\vec{n}}_v\in C_c^\infty(\mathbb{G}_{\vec{n}}(\Q_v))$, there exists $\phi^{\vec{n}}_v\in C_c^\infty(G_{\vec{n}}(\Q_v))$ such that $f^{\vec{n}}_v$ and $\phi^{\vec{n}}_v$ have matching orbital integrals. We will be more interested in \emph{inverse transfer}: finding $f^{\vec{n}}_v$ given $\phi^{\vec{n}}_v$ such that the two are $BC$-matching functions. For this, we will follow the strategy used in Section 8 of~\cite{morel-trace}. Proposition 3.3.2 of~\cite{labesse}, recalled below, gives a sufficient condition for a function $\phi^{\vec{n}}_v$ to be a $BC$-transfer of some $f^{\vec{n}}_v$.

\begin{prop}\label{labesse-inverse transfer} Let $\phi^{\vec{n}}\in C_c^\infty(G_{\vec{n}}(\Q_v))$ be such that the stable orbital integral of $\phi^{\vec{n}}$, $SO_{\gamma}(\phi^{\vec{n}})$, is equal to $0$ for every $\gamma\in G_{\vec{n}}(\Q_v)$ that is not a norm. Then there exists $f^{\vec{n}}\in C_c^\infty(\mathbb{G}_{\vec{n}}(\Q_v))$ such that $\phi^{\vec{n}}$ is a $BC$-transfer of $f^{\vec{n}}$. 
\end{prop}

The following is an analogue of Lemma 8.3.6 of~\cite{morel-trace} to the case where $F$ is a general imaginary CM field rather than an imaginary quadratic field. 

\begin{lem}\label{detecting norms in derived group} Let $D_{\vec{n}}:=G_{\vec{n}}/G^{\mathrm{der}}_{\vec{n}}$. Then a semisimple element of $G_{\vec{n}}(\Q_v)$ is a norm if and only if its image in $D_{\vec{n}}(\Q_v)$ is a norm.
\end{lem}

\begin{proof} This follows in the same way as Lemma 8.3.6 of~\cite{morel-trace}. We give more details for completeness. If $v$ splits in $F_0$, then every element of $G_{\vec{n}}(\Q_v)$ is a norm and the same holds for $D_{\vec{n}}(\Q_v)$, so the lemma is immediate in this case. Assume that $v$ is inert or ramified in $F_0$ and let $w$ be the place of $F_0$ above $v$. For every Levi subgroup $M$ of $G_{\vec{n}}$, set $D_M:=M/M^{\mathrm{der}}$. For a semisimple element $\gamma\in G_{\vec{n}}(\Q_v)$, let $M$ be a Levi subgroup of $G_{\vec{n}}$ such that $\gamma\in M(\Q_v)$ and $\gamma$ is elliptic in $M$. Then Proposition 2.5.3 of~\cite{labesse} says that $\gamma$ is a norm if and only if its image in $D_M(\Q_v)$ is a norm. It remains to check that an element of $D_M(\Q_v)$ is a norm if and only if its image under the canonical map $D_M(\Q_v)\to D_{\vec{n}}(\Q_v)$ is a norm. 

The Levi $M$ is $G_{\vec{n}}(\Q_v)$ conjugate to a standard Levi subgroup of $G_{\vec{n}}$ defined over $\Q$. We may assume that $M$ is a standard Levi subgroup of $G_{\vec{n}}$ defined over $\Q$, in which case we have explicit descriptions of $M, D_M$ and $D_{\vec{n}}$. There exist $s,m_1,m_2\in \mathbb{N}$ such that $n_i-m_i$ is non-negative and even for $i=1,2$ and $M$ is the product of $G_{m_1,m_2}$ and of $s$ groups obtained by restriction of scalars of general linear groups over $F$. We have isomorphisms
\[D_{\vec{n}}\toisom \left\{(\lambda,z_1,z_2)\in \mathbb{G}_m\times (\mathrm{Res}_{F/\Q}\mathbb{G}_m)^2| z_i\bar{z}_i=\lambda^{n_i}\right\}.\]
and
\[D_M\toisom D_{M,l}\times D_{M,h},\]
where $D_{M,l}=(\mathrm{Res}_{F/\Q}\mathbb{G}_m)^s$ and
\[D_{M,h}\toisom\left\{(\lambda,z_1,z_2)\in \mathbb{G}_m\times (\mathrm{Res}_{F/\Q}\mathbb{G}_m)^2| z_i\bar{z}_i=\lambda^{m_i}\ \mathrm{if}\ m_i>0\ \mathrm{and}\ z_i=1\ \mathrm{if}\ m_i=0.\right\}\] The map $D_M\to D_{\vec{n}}$ is $1$ on $D_{M,l}$ and given by 
\[(\lambda,z_1,z_2)\mapsto (\lambda,\lambda^{(n_1-m_1)/2}z_1, \lambda^{(n_2-m_2)/2}z_2)\]
on $D_{M,h}$. Every element of $D_{M,l}(\Q_p)$ is the norm of some element in $\mathbb{D}_{M,l}(\Q_p)$, where $\mathbb{D}_{M,l}:=\mathrm{Res}_{F_0/\Q_p}(D_{M,l}\times_{\Q_p}F_0)$. The lemma follows, since an element of $D_{\vec{n}}$ or $D_{M,h}$ is a norm if and only if the corresponding multiplier $\lambda$ is a norm.  
\end{proof}

By combining Proposition~\ref{labesse-inverse transfer}, Lemma~\ref{detecting norms in derived group} and the observation that the group of norms in a torus contains an open neighborhood of the identity, we deduce that every function in $C_c^\infty (G_{\vec{n}}(\Q_v))$ with support a small enough neighborhood of the identity is the $BC$-transfer of some function in $C_c^\infty(\mathbb{G}_{\vec{n}}(\Q_v))$. 

\begin{lem}\label{condition H} Let $v\in S_f^p$. Every function $\phi_v\in C_c^\infty(G_{2n}(\Q_v))$ with support in a small enough neighborhood of the identity has the property that its transfer $\phi^{\vec{n}}\in  C_c^\infty(G_{\vec{n}}(\Q_v))$ is the $BC$-transfer of some $f^{\vec{n}}\in C_c^\infty(\mathbb{G}_{\vec{n}}(\Q_v))$. 
\end{lem}

\begin{rem}\label{choice of N0} We can now explain the choice of $N_0$: It is such that the principal congruence subgroup of level $N_0$ is small enough in the sense of this lemma, for all $v\in S_f^p$. (Recall that $G\cong G_{2n}$.)
\end{rem}

\begin{proof} This is essentially the same as the proof of the first part of Lemma 8.4.1 of~\cite{morel-trace}. The only difference is that we use Lemma~\ref{detecting norms in derived group} instead of Lemma 8.3.6 of~\cite{morel-trace}. We sketch the proof. For every $G_{\vec{n}}\in \mathscr{E}^\mathrm{ell}(G_{2n})$, choose an embedding $G_{\vec{n}}\hookrightarrow G_{2n}$. By Lemma~\ref{detecting norms in derived group} and the observation that the group of norms in a torus contains an open neighborhood of the identity, there exists an open neighborhood of the identity $U_{\vec{n}}\subset G_{\vec{n}}$ such that every semisimple element of $U_{\vec{n}}$ is a norm. By Lemma 8.4.2 of~\cite{morel-trace} (with $\mathbb{G}_{\vec{n}}$ the Levi subgroup of $\mathbb{G}_{2n}$), there exists a neighborhood $V_{\vec{n}}$ of the identity in $G_{2n}(F_0\otimes_{\Q}\Q_v)$ such that every semisimple element of $G_{\vec{n}}(\Q_v)$ that is $G_{2n}(F_0\otimes_{\Q}\Q_v)$-conjugate to an element of $V_{\vec{n}}$ is $G_{\vec{n}}(\Q_v)$-conjugate to an element of $U_{\vec{n}}$. 

Let $V:=\cap_{G_{\vec{n}}\in \mathscr{E}^{\mathrm{ell}}(G)}V_{\vec{n}}$, then $U:=V\cap G_{2n}(\Q_v)$ is an open neighborhood of the identity in $G_{2n}(\Q_v)$. Any function $\phi\in C_c^\infty(G_{2n}(\Q_v))$ with support contained in $U$ will satisfy the desired property. Choose a transfer $\phi^{\vec{n}}$ of $\phi$ to $G_{\vec{n}}$. To show that $\phi^{\vec{n}}$ is the base change transfer of some $f^{\vec{n}}\in C_c^\infty(\mathbb{G}_{\vec{n}}(\Q_v))$, it is enough, by Proposition~\ref{labesse-inverse transfer} to show that, if for some semisimple $\gamma\in G_{\vec{n}}(\Q_v)$ we have $SO_{\gamma}(\phi^{\vec{n}})\not = 0$, then $\gamma$ is a norm. By the definition of transfer, if $SO_{\gamma}(\phi^{\vec{n}})\not =0$, then there exists $\delta \in G_{2n}(\Q_v)$ which is associated to $\gamma$ such that the orbital integral $O_{\delta}(\phi)\not =0$. The fact that $\gamma$ and $\delta$ are associated means that they are $G_{2n}(F_0\otimes_{\Q}\Q_v)$-conjugate and the fact that $O_{\delta}(\phi)\not = 0$ means that $\delta \in U\subseteq V_{\vec{n}}$. By the construction of $V_{\vec{n}}$, we see that $\gamma$ is $G_{\vec{n}}(\Q_v)$-conjugate to an element of $U_{\vec{n}}$, which is a norm, so $\gamma$ is a norm itself. 
\end{proof}

\subsection{The twisted trace formula}
The first part of this section is identical to the corresponding section of~\cite{caraiani-scholze}. The essential difference is that we give more detail in the proof of Proposition~\ref{comparing to the geometric side}, since we want to emphasize the role that \emph{cuspidal subgroups} play in it.  

Define the group \[\mathbb{G}^+_{\vec{n}}:=(\mathrm{Res}_{F_0/\mathbb{Q}}\mathbb{G}_m\times \mathrm{Res}_{F/\mathbb{Q}}\GL_{\vec{n}})\rtimes\{1,\theta\},\] where $\theta(\lambda,g)\theta^{-1}=(\lambda^c, \lambda^cg^\sharp)$ and $g^\sharp=\Phi_{\vec{n}}^tg^c\Phi_{\vec{n}}^{-1}$. There is a natural $\mathbb{Q}$-isomorphism $\mathbb{G}_{\vec{n}}\toisom \mathbb{G}_{\vec{n}}^\circ$, which extends to an isomorphism 
\[\mathbb{G}_{\vec{n}}\rtimes \Gal(F_0/\mathbb{Q})\toisom \mathbb{G}_{\vec{n}}^+\] 
so that $c\in \Gal(F_0/\mathbb{Q})$ maps to $\theta$. Using this isomorphism, we write $\mathbb{G}_{\vec{n}}$ and $\mathbb{G}_{\vec{n}}\theta$ for the two cosets. 

If $f\in C^\infty_c(\mathbb{G}_{\vec{n}}(\mathbb{A}))$ (with trivial character on $A^\circ_{\mathbb{G}_n,\infty}$), then we define $f\theta$ to be the function on $\mathbb{G}_{\vec{n}}\theta(\mathbb{A})$ obtained via translation by $\theta$. The (invariant) twisted trace formula (see~\cite{arthura, arthurb}) gives an equality 
\begin{equation} \label{twisted} I^{\mathbb{G}_{\vec{n}}\theta}_{\mathrm{geom}}(f\theta)=I^{\mathbb{G}_{\vec{n}}\theta}_{\mathrm{spec}}(f\theta).\end{equation} 
The left hand side of the equation is defined in Section 3 of~\cite{arthurb}, while the right hand side is defined in Section 4 of loc. cit.

Let $f_{\mathbb{G}_{\vec{n}},\Xi}$ and $\phi_{G_{\vec{n}},\xi}$ be as defined above. The following is analogous to Corollary 4.7 of~\cite{shin-galois}. 

\begin{prop}\label{comparing to the geometric side} We have the following equality:
\begin{equation}I^{\mathbb{G}_{\vec{n}}\theta}_{\mathrm{geom}}(f^{\vec{n}}\theta)=\sum_{\gamma}SO^{G_{\vec{n}}(\mathbb{A})}_{\gamma}(\phi^{\vec{n}})=\tau(G_{\vec{n}})^{-1}\cdot ST_e^{G_{\vec{n}}}(\phi^{\vec{n}}), \end{equation} when $\phi^{\vec{n}}$ and $f^{\vec{n}}$ satisfy \[\phi^{\vec{n}} =(\phi^{\vec{n}})^S\cdot \phi^{\vec{n}}_{S_f}\cdot \phi_{G_{\vec{n}},\xi}\ \mathrm{and}\ f^{\vec{n}}=(f^{\vec{n}})^S\cdot f^{\vec{n}}_{S_f}\cdot f_{\mathbb{G}_{\vec{n}},\Xi}\] with $(\phi^{\vec{n}})^S$ a $BC$-transfer of $(f^{\vec{n}})^S$, $\phi^{\vec{n}}_{S_f}$ a $BC$-transfer of $f^{\vec{n}}_{S_f}$. The sum in the center runs over a set of representatives of $\Q$-elliptic semisimple stable conjugacy classes in $G_{\vec{n}}(\Q)$ and $\tau(G_{\vec{n}})$ is the Tamagawa number of $G_{\vec{n}}$. 
\end{prop}

\begin{proof} The second equality follows exactly as in the proof of Corollary 4.7 of~\cite{shin-galois}. Moreover, Theorem 4.3.4 of~\cite{labesse} rewrites the sum of stable orbital integrals in the center as the \emph{elliptic} part of the twisted trace formula for $\mathbb{G}_{\vec{n}}\theta$. Comparing with the definition of the elliptic part in Section 4.1 of~\cite{labesse}, we get \[\sum_{\gamma}SO^{G_{\vec{n}}(\mathbb{A})}_{\gamma}(\phi^{\vec{n}})=\sum_{\delta}\mathrm{vol}(I_{\delta\theta}(\Q)A_{\mathbb{G}_{\vec{n}}\theta}\setminus I_{\delta\theta}(\A))\cdot O_{\delta\theta}^{\mathbb{G}_{\vec{n}}(\mathbb{A})}(f^{\vec{n}}),\] where $\gamma$ is as above and $\delta$ runs over a set of representatives for $\theta$-elliptic $\theta$-conjugacy classes in $\mathbb{G}_{\vec{n}}(\Q)$.

The key point in proving the first equality is to simplify the geometric side of the twisted trace formula for $\mathbb{G}_{\vec{n}}\theta$, using similar techniques to those in Chapter 7 of~\cite{arthurb} to show that only this part contributes. This is done in Proposition 4.5 of~\cite{shin-galois}, but we repeat the argument here in order to emphasize the role that cuspidal subgroups play. Theorem 7.1.(b) and Corollary 7.4 of~\cite{arthurb} give a way to simplify the geometric side of the trace formula if the test function $f^{\vec{n}}$ is cuspidal at two places. The test function $f^{\vec{n}}$ that we constructed is only necessarily cuspidal at $\infty$ but we will use the assumption that $[F^+:\Q]\geq 2$ instead and we will explain why twisted analogues of these two results hold. 

The geometric side of the twisted trace formula is a linear combination of invariant distributions $I_{\mathbb{M}}^{\mathbb{G}_{\vec{n}}\theta}(\delta\theta, f^{\vec{n}}\theta)$, where $\mathbb{M}$ runs over Levi subsets of $\mathbb{G}_{\vec{n}}$. (The twisted analogue of) Theorem 7.1.(b) of~\cite{arthurb} shows that the geometric side can be simplified to a linear combination of distributions on $\mathbb{G}_{\vec{n}}\theta$ itself. The proof of Theorem 7.1.(b) of~\cite{arthurb} goes through as long as one can show that the distributions $I_{\mathbb{M}}^{\mathbb{G}_{\vec{n}}\theta}(\delta\theta,f^{\vec{n}}\theta)$ vanish for every proper subset $\mathbb{M}\subset \mathbb{G}_{\vec{n}}$ and semisimple element $\delta \in \mathbb{M}(\Q)$. Since $f^{\vec{n}}_{\mathbb{G}_{\vec{n}},\Xi}$ is cuspidal, Proposition 9.1 (the splitting formula) and Corollary 8.3 of~\cite{arthura} imply that $I_{\mathbb{M}}^{\mathbb{G}_{\vec{n}}\theta}(\delta\theta,f^{\vec{n}}\theta)$ is a multiple of $I_{\mathbb{M}}^{\mathbb{G}_{\vec{n}}\theta}(\delta\theta,f_{\mathbb{G}_{\vec{n}},\Xi}\theta)$ so it is enough to show that the distribution at $\infty$ vanishes. Proposition 8.2.3 of~\cite{morel-trace} shows that the term $ I_{\mathbb{M}}^{\mathbb{G}_{\vec{n}}\theta}(\delta\theta,f_{\mathbb{G}_{\vec{n}},\Xi}\theta)$ vanishes if $\mathbb{M}$ is not a \emph{cuspidal} Levi subset of $\mathbb{G}_{\vec{n}}\theta$. We recall what cuspidal subgroups and subsets are in Definition~\ref{cuspidal subgroups} below and show in Lemma~\ref{no cuspidal subgroups} that, when $[F^+:\Q]\geq 2$, $\mathbb{G}_{\vec{n}}\theta$ does not admit any cuspidal subsets. 

The argument above implies that the conclusion of Theorem 7.1.(b) of~\cite{arthurb} holds: the geometric side of the twisted trace formula in this case is a linear combination of distributions on $\mathbb{G}_{\vec{n}}\theta$ itself. Recall that $f^{\vec{n}}$ is cuspidal at $\infty$. Theorem A.1.1 of~\cite{labesse} implies that $O_{\delta\theta}^{\mathbb{G}_{\vec{n}}(\mathbb{A})}(f^{\vec{n}})=0$ unless $\delta$ is $\theta$-elliptic in $\mathbb{G}_{\vec{n}}(\mathbb{R})$. We deduce the twisted analogue of the first condition of Corollary 7.4 of~\cite{arthurb}. The second condition (that $f^{\vec{n}}$ is cuspidal at an additional place) is not needed for the proof to go through: it was only imposed in order to invoke Theorem 7.1 of loc. cit. 

By the twisted analogue of Corollary 7.4 of~\cite{arthurb}, we deduce that 
\[I^{\mathbb{G}_{\vec{n}}\theta}_{\mathrm{geom}}(f^{\vec{n}}\theta)=\sum_{\delta}\mathrm{vol}(I_{\delta\theta}(\Q)A_{\mathbb{G}_{\vec{n}}\theta}\setminus I_{\delta\theta}(\A))\cdot O_{\delta\theta}^{\mathbb{G}_{\vec{n}}(\mathbb{A})}(f^{\vec{n}}),\] where $\delta$ runs over a set of representatives for $\theta$-elliptic $\theta$-conjugacy classes in $\mathbb{G}_{\vec{n}}(\Q)$. The proposition follows. 
\end{proof}

\begin{defn}\label{cuspidal subgroups}\leavevmode\begin{enumerate}
\item A cuspidal group is a reductive group $M$ over $\mathbb Q$ such that $(M/A_M)_{\mathbb R}$ contains a maximal $\R$-torus that is $\R$-anisotropic. Here $A_M$ denotes the maximal $\Q$-split torus in the center of $M$.
\item A cuspidal Levi subgroup $M$ of $G_{\vec{n}}$ is a subgroup which is a cuspidal group.
\item For a Levi subset $\mathbb{M}$ of $\mathbb{G}_{\vec{n}}\theta$, we let $\mathbb{M}^\circ$ be the intersection of the subgroup of $\mathbb{G}^+_{\vec{n}}$ generated by $M$ with $\mathbb{G}_{\vec{n}}$. We say that a subset $\mathbb{M}$ of $\mathbb{G}_{\vec{n}}\theta$ is cuspidal if $\mathbb{M}^\circ\cap \mathbb G_{\vec{n}}$ is a cuspidal Levi subgroup of $\mathbb G_{\vec{n}}$. 
\end{enumerate}
\end{defn}

\begin{lem}\label{no cuspidal subgroups} $\mathbb{G}_{\vec{n}}\theta$ admits no proper cuspidal subsets. 
\end{lem}

\begin{proof} It is enough to show that $\mathbb G_{\vec{n}}$ admits no proper cuspidal Levi subgroups. Suppose that $M\subsetneq \mathbb G_{\vec{n}}$ is a cuspidal Levi subgroup. Then $M$ contains a direct factor of the form $\mathrm{Res}_{F/\Q}\GL_h$ for some $0<h\leq n$, which would also have to be cuspidal group. However, the center of $\mathrm{Res}_{F/\Q}\GL_h\times_{\Q}\R$ contains a split torus of rank $[F^+:\Q]>1$, whereas the maximal split torus in the center of $\mathrm{Res}_{F/\Q}\GL_h$ has rank $1$.  
\end{proof}

\begin{remark}\label{remark on imaginary quadratic case} When $F$ is an imaginary quadratic field, so $F^+=\Q$, the unitary similitude groups $G_{\vec{n}}$ do admit proper cuspidal subgroups. For example, for $G_{2n}$, these are the Levi subgroups of the form $(\mathrm{Res}_{F/\Q}\mathbb{G}_m)^s\times G_{2(n-s)}$ for some $s\in \mathbb{Z}_{>0}$. In this case, one cannot directly identify the stable trace formula for Igusa varieties with the geometric side of the (twisted) trace formula. The latter contains more terms coming from the cuspidal subgroups.  
\end{remark}

\subsection{Construction of test functions}\label{test functions}
In this section, we explain how to construct our test functions. This differs slightly from the construction of test functions in Section 6 of~\cite{shin-galois} and Section 5 of~\cite{caraiani-scholze}, but only at primes in $S_f^p$. We fix the isomorphism $G\cong G_{2n}$. We let $(f^{2n})^S$ be any function in $\cH^\ur(\mathbb{G}_{2n}(\A^S))$ and take $\phi^S\in \cH^\ur(G(\A^S))$ be its $BC$-transfer. We let $\phi_{S_f^p}\in C^\infty_c(G(\A_{S_f^p}))$ be the characteristic function of a principal congruence subgroup for some sufficiently large level $N$, as guaranteed by Lemma~\ref{condition H}. We let $\phi_p\in C_c^\infty(J_b(\Q_p))$ and set 
\[\phi:=\phi_p\cdot \phi_{S_f^p}\cdot \phi^S,\]
and assume, in the first instance, that it is acceptable.

From these test functions, we construct all the other test functions we will need. First, for each elliptic endoscopic group $G_{\vec{n}}$, we let $\phi^{\vec{n}}$ be the function constructed from $\phi$ as in Definition~\ref{definition of transfer}. Let $(f^{n_1,n_2})^S$ be obtained from $(f^{2n})^S$ by transfer along the $L$-morphism $\tilde{\zeta}_{n_1,n_2}$. We choose $f^{\vec{n}}_{S_f^p}$ as in Lemma~\ref{condition H} to be any function such that $\phi^{\vec{n}}_{S_f^p}$ is a $BC$-transfer of $f^{\vec{n}}_{S_f^p}$. We choose $f^{\vec{n}}_p$ so that $BC^*_{\vec{n}}(f^{\vec{n}}_p)=\phi^{\vec{n}}_p$ (recall that $BC^*_{\vec{n}}$ is surjective at $p$). We define $f^{\vec{n}}_\infty$ explicitly, as a linear combination of Lefschetz functions for representations $\Xi(\varphi_{\vec{n}})$ of $\mathbb{G}_{\vec{n}}$ for which $\tilde \zeta_{\vec{n}} \circ \varphi_{\vec{n}}$ corresponds to the trivial representation of $\mathbb G_{2n}$ (see (6.7) of~\cite{shin-galois} for the precise formula). Finally, we set 
\[f^{\vec{n}}:= f^{\vec{n}}_p\cdot(f^{\vec{n}})_{S_f^p}\cdot (f^{\vec{n}})^S\cdot f^{\vec{n}}_{\infty}.\] 
By the commutative diagram~\eqref{commutative diagram transfer of functions}, we see that $(\phi^{\vec{n}})^S$ is the transfer of $(f^{\vec{n}})^S$ along $BC^*_{\vec{n}}$. We can therefore apply Proposition~\ref{comparing to the geometric side} to $f^{\vec{n}}$ and $\phi^{\vec{n}}$.  We obtain the following result, analogous to Theorem 5.3.2 of~\cite{caraiani-scholze}.

\begin{thm}\label{spectral side} We have an equality 
\[\mathrm{tr}(\phi|[H_c(\mathrm{Ig}^b,\mathbb{\overline Q}_\ell)])=\tau(G)\sum_{G_{\vec{n}}\in \mathscr{E}^{\mathrm{ell}}(G)}\epsilon_{\vec{n}}\cdot I^{\mathbb{G}_{\vec{n}}\theta}_{\mathrm{spec}}(f^{\vec{n}}\theta),\]
where $\epsilon_{\vec{n}}=\frac{1}{2}$ if $\vec{n}=(n,n)$ and $\epsilon_{\vec{n}} = 1$ otherwise. 
\end{thm}

Recall the group morphism 
\[\mathrm{Red}^b_{\vec{n}}:\mathrm{Groth}(G_{\vec{n}}(\Q_p))\to \mathrm{Groth}(J_b(\Q_p))\] 
from Section 5.4 of~\cite{caraiani-scholze}. This is the representation-theoretic counterpart of constructing the transfer $\phi_p^{\vec{n}}$ of $\phi_p$. 
The following is an analogue of Lemma 5.5.1 of~\cite{caraiani-scholze}.

\begin{lem}\label{linear combination} The trace $\mathrm{tr}(\phi|[H_c(\mathrm{Ig}^b,\mathbb{\overline Q}_\ell)])$
can be written as a linear combination of terms of the form
\[
\mathrm{tr}\left(\mathrm{Red}^b_{\vec{n}}(\pi^{\vec{n}}_p)(\phi_p)\right)\mathrm{tr}\left((\Pi^{\vec{n}})_{S_f^p}((f^{\vec{n}})_{S_f^p})\circ A_{S_f^p}\right)\mathrm{tr}\left( (\tilde\zeta_{\vec{n}\ast} \Pi^{\vec{n}})^S((f^{2n})^S)\right) ,
\]
where $\pi^{\vec{n}}_p\in \mathrm{Rep}(G_{\vec{n}}(\Q_p))$ base changes to $\Pi^{\vec{n}}_p\in \mathrm{Rep}(\mathbb{G}_{\vec{n}}(\Q_p))$, the component at $p$ of a $\theta$-stable isobaric irreducible automorphic representation $\Pi^{\vec{n}}$ of $\mathbb{G}_{\vec{n}}$. Moreover, $\Pi^{\vec{n}}_\infty$ is cohomological (with respect to the trivial algebraic representation). The conclusion holds even without the assumption that $\phi$ is acceptable. 
\end{lem}

\begin{proof} We sketch the proof. Since we are only interested in obtaining a linear combination, we do not keep track of endoscopic signs or constants. First, assume that $\phi$ was chosen to be an acceptable function. Then we can apply Theorem~\ref{spectral side} to $\mathrm{tr}(\phi|[H_c(\mathrm{Ig}^b,\mathbb{\overline Q}_\ell)])$.

For each $G_{\vec{n}}\in \mathscr{E}^{\mathrm{ell}}(G)$, we simplify the spectral side $I^{\mathbb{G}_{\vec{n}}\theta}_{\mathrm{spec}}(f^{\vec{n}}\theta)$ as in Section of~\cite{shin-galois} and Section 5.3 of~\cite{caraiani-scholze}. Fix a minimal Levi subgroup $M_0$ of $G_{\vec{n}}$. For each Levi subgroup $M$ containing $M_0$, choose a parabolic subgroup $Q$ with Levi $M$. The spectral side $I^{\mathbb{G}_{\vec{n}}\theta}_{\mathrm{spec}}(f^{\vec{n}}\theta)$ can be written as a linear combination of terms of the form $\mathrm{tr}\left(\mathrm{n-Ind}_Q^{\mathbb{G}_{\vec{n}}}(\Pi_M)_\xi (f^{\vec{n}})\circ A\right)$, where $\Pi_M$ runs over irreducible, $\Phi_{\vec{n}}^{-1}\theta$-stable subrepresentations of the (relatively) discrete spectrum $R_{M,\mathrm{disc}}$.  The subscript $\xi$ indicates a possible twist by a character of $A^\circ_{\mathbb{G}_{\vec{n}},\infty}$ corresponding to an irreducible algebraic representation $\xi$ of $G_{\vec{n}}$ and $A$ is a normalized intertwiner on $\mathrm{n-Ind}^{\mathbb{G}_{\vec{n}}}_Q(\Pi_M)_{\xi}$. This follows by combining Proposition 4.8 and Corollary 4.14 of~\cite{shin-galois}. 

We take $\Pi^{\vec{n}}$ to be $\mathrm{n-Ind}^{\mathbb{G}_{\vec{n}}}_Q(\Pi_M)_{\xi}$. Choose a decomposition 
\[A=A_p\cdot A_{S_f^p}\cdot A^S\cdot A_\infty\]
as a product of normalized intertwining operators. Then we can rewrite the desired trace as a linear combination of terms of the form 
\[\mathrm{tr}\left(\Pi^{\vec{n}}_p(f^{\vec{n}}_p)\circ A_p\right) \mathrm{tr}\left((\Pi^{\vec{n}})_{S_f^p}((f^{\vec{n}})_{S_f^p})\circ A_{S_f^p}\right) \mathrm{tr}\left( (\Pi^{\vec{n}})^S((f^{\vec{n}})^S)\circ A^S\right).\] 
In order to rewrite the traces at $p$, use the fact that the base change transfer of representations at $p$ is injective, since $p$ splits in $F_0$, to construct $\pi_p^{\vec{n}}$; then the term at $p$ is equal to $\mathrm{tr}\left(\pi^{\vec{n}}_p(\phi^{\vec{n}}_p)\right)$. Then appeal to Lemma 5.4.2 of~\cite{caraiani-scholze} which identifies the latter trace with $\mathrm{tr}\left(\mathrm{Red}^b_{\vec{n}}(\pi^{\vec{n}}_p)(\phi_p)\right)$. On the other hand, if we ignore the sign that comes from the choice of the normalized intertwiner $A^{\infty,p}$, then recalling that $(f^{\vec{n}})^S$ is the transfer of $(f^{2n})^S$ along $\tilde{\zeta}^*_{\vec{n}}$, we can rewrite
\[\mathrm{tr}\left( (\Pi^{\vec{n}})^S((f^{\vec{n}})^S)\right)=\mathrm{tr}\left( \tilde{\zeta}_{\vec{n},*}(\Pi^{\vec{n}})^S((f^{2n})^S)\right).\]

We now explain why $\Pi^{\vec{n}}$ has the desired properties. The fact that $\Pi^{\vec{n}}$ is $\theta$-stable follows from the fact that $\Pi_M$ is $\Phi_{\vec{n}}^{-1}\theta$-stable. The fact that $\Pi^{\vec{n}}$ is irreducible follows from the fact that $\Pi_M$ is irreducible and unitary, and $\Pi_M$ is isobaric because it contributes to the (relatively) discrete automorphic spectrum $R_{M,\mathrm{disc}}$ and $M$ is a product of general linear groups. Then $\Pi^{\vec{n}}$ is also isobaric because it is irreducible. 

Finally, we remove the assumption that $\phi$ be acceptable using Lemma 6.4 of~\cite{shin-igusa}. (The idea is that the twist $\phi_p^{(N')}$ of any $\phi_p$ by a power of Frobenius makes $\phi$ an acceptable function for any large enough $N'$ and that, as long as we keep $\phi^p$ fixed, we have expressed the desired trace as a finite linear combination of traces of $\phi_p^{(N')}$ against irreducible representations of $J_b(\Q_p)$.  The argument in the proof of Lemma 6.4 of~\cite{shin-igusa} proves that the desired equality holds for every integer $N'$ and, in particular, for $N'=0$.)
\end{proof}

\begin{remark} A fixed $\mathbb T^S$-character of $[H_c(\Ib, \overline{\Q}_{\ell})]$ could be obtained from several $\Pi^{\vec{n}}$ for several different $G_{\vec{n}}\in \mathscr{E}^\mathrm{ell}(G)$. For example, in the Case $ST$ which is discussed in Section 6 of~\cite{shin-galois}, the contribution is from an endoscopic group $\mathbb{G}_{n_1,n_2}$ but also from a Levi subgroup $M$ of $\mathbb{G}$. 
\end{remark}

\subsection{Galois representations} Lemma~\ref{linear combination} essentially finishes the proof of Theorem~\ref{thm:igusa computation}; it only remains to apply the known results on existence of Galois representations for regular $L$-algebraic, essentially self-dual, cuspidal automorphic representations of $\GL_{m}(\A_F)$. (Here we use the notions of $L$-algebraic and $C$-algebraic representations due to Buzzard-Gee~\cite{buzzardgee} and note that in the case of general linear groups these notions only differ by a character twist.)

Our goal is now to construct a Galois representation \[\rho_{\Pi}:\mathrm{Gal}(\overline F/F)\to \GL_{2n}(\overline{\Q}_{\ell})\] attached to the automorphic representation $\Pi:=\tilde{\zeta}_{\vec{n},*}(\Pi^{\vec{n}})$ (or rather the automorphic representation of $\GL_{2n}(\A_F)$ obtained from $\Pi$ by forgetting the similitude factor) as in Lemma~\ref{linear combination}.

By Lemma~\ref{linear combination}, $\Pi^{\vec{n}}_\infty$ is cohomological, which implies that $\Pi^{\vec{n}}$ is $C$-algebraic. Write $\Pi^{\vec{n}}=\psi\otimes \Pi_1\otimes \Pi_2$ according to the decomposition $\mathbb{G}_{n_1,n_2}(\A)=\A^\times_{F_0}\times \GL_{n_1}(\A_F)\times \GL_{n_2}(A_F)$. Each $\Pi_i$ is a regular $C$-algebraic, $\theta$-stable isobaric automoprhic representation of $\GL_{n_i}(\A_F)$. The automorphic representation $\Pi_i|\det|^{(1-n_i)/2}$ is regular $L$-algebraic. 

Recall that we have chosen an isomorphism $\iota_\ell:\overline{\Q}_\ell\toisom \C$.

\begin{thm}\label{Galois representations conjugate self-dual} There exists a Galois representation 
\[\rho_i:\mathrm{Gal}(\overline F/F)\to \GL_{n_i}(\overline{\Q}_\ell)\]
such that for any place $\mathfrak{q}$ of $F$, 
\[WD\left(\rho_i|_{\mathrm{Gal}(\overline F_{\mathfrak{q}}/F_{\mathfrak{q}})}\right)^{\mathrm{F-ss}}\simeq\iota^{-1}_\ell\mathrm{rec}\left(\Pi_{i,\mathfrak{q}}|\det|^{(1-n_i)/2}\right),\] where $\mathrm{rec}$ denotes the local Langlands correspondence normalized as in~\cite{harris-taylor}. 
\end{thm}

\begin{proof} This is proved just as Theorem 5.5.3 of~\cite{caraiani-scholze}. Recall that the representation $\Pi^{\vec{n}}$ was constructed as $\mathrm{n-Ind}_{Q}^{\mathbb{G}_{\vec{n}}}(\Pi_M)$ (recall that for us $\xi$ is trivial), for some automorphic representation $\Pi_M$ which is $\Phi_{\vec{n}}^{-1}\theta$-stable and which occurs in the discrete automorphic spectrum of some Levi subgroup $M$ of $\mathbb{G}_{\vec{n}}$. This means we can write 
\[\Pi_i=\mathrm{n-Ind}_{Q_i}^{\GL_{n_i}}\left(\Pi_{M_i}\right),\]
where $M_i$ is the Levi subgroup of parabolic subgroup $Q_i$ of $\GL_{n_i}$ and $\Pi_{M_i}$ is $\Phi_{n_i}^{-1}\theta$-stable and occurs in the discrete automorphic spectrum of $M_i$. The classification of the discrete automorphic spectrum for general linear groups due to Harris--Taylor and Moeglin--Waldspurger~\cite{moeglin-waldspurger} together with the fact that $\Pi_{M_i}$ is $\Phi_{n_i}^{-1}\theta$-stable tells us that $\Pi_{M_i}$ can be expressed in terms of regular $L$-algebraic, conjugate self-dual cuspidal automorphic representations of (possibly a product of) general linear groups. The existence of the Galois representation $\rho_i$ and the compatibility with the local Langlands correspondence now follows from the main theorems of~\cite{shin-galois, chenevier-harris, caraiani}. 
\end{proof}

Now we can prove Theorem~\ref{thm:igusa computation}.

\begin{proof}[Proof of Theorem~\ref{thm:igusa computation}] We use the description from Lemma~\ref{linear combination}, where we fix $\phi_{S_f^p}$ to be the characteristic function of $K(N)$, and note that we can realize all elements of $\mathbb T^S$ as the base change of some $(f^{2n})^S$ as we only used places split in $F_0$ in the definition of $\mathbb T^S$. Lemma~\ref{linear combination} shows that for each $j\in J$, there is some cuspidal automorphic representation $\Pi^{\vec{n}}=\psi \otimes \Pi_1\otimes \Pi_2$ of $\mathbb{G}_{\vec{n}}$ such that $(\Pi^{\vec{n}})^S$ gives the Hecke character $\psi_j$, and $\pi_j = \mathrm{Red}^b_{\vec{n}}(\pi^{\vec{n}}_p)$ where $\pi^{\vec{n}}_p$ is an irreducible representation of $G_{\vec{n}}(\Q_p)$ base changing to $\Pi^{\vec{n}}_p$. By Theorem~\ref{Galois representations conjugate self-dual}, there exist Galois representations $\rho_i$ associated to the $L$-algebraic representations $\Pi_i|\det|^{(1-n_i)/2}$.

The character $\varpi$ satisfies $\varpi_\infty(z)=(z/\bar z)^{\delta/2}$ for some odd integer $\delta$, since $\varpi_\infty:\C^\times\to \C^\times$ extends the sign character on $\mathbb{R}^\times$. The character of $\GL_{n_i}(\A_F)$ defined by $|\det|^{n_i/2-n}\varpi(N_{F/F_0}\circ\det)^{\epsilon(n_i)}$ is $L$-algebraic, since $(n_i+\epsilon(n_i)\delta)/2\in \Z$, so it corresponds to a character $\epsilon_i:\mathrm{Gal}(\overline F/F)\to \overline{\Q}_\ell^\times$. 

Let $\Pi:=\tilde\zeta_{n_1,n_2,*}(\Pi^{\vec{n}})$. Write $\Pi=\psi'\otimes \Pi^0$, according to the decomposition $\mathbb{G}_{2n}(\A)=\A_{F_0}^\times \times \GL_{2n}(\A_F)$. Set \[\Pi^0_i:=\Pi_i|\det|^{(1-n_i)/2}|\det|^{n_i/2-n}\varpi(N_{F/F_0}\circ\det)^{\epsilon(n_i)}\] By the definition of $\tilde\zeta_{n_1,n_2,*}$, we get the identity \[\Pi^0|\det|^{1/2-n}=\mathrm{n-Ind}_{\GL_{n_1}\times \GL_{n_2}}^{\GL_{2n}}\left(\Pi^0_1\otimes \Pi^0_2\right).\]
The representation on the RHS is $L$-algebraic and normalized parabolic induction is compatible with this notion and with the local Langlands correspondence $\mathrm{rec}$, so the term on the LHS, $\Pi^0|\det|^{1/2-n}$, is also $L$-algebraic, with corresponding Galois representation $\rho_{\Pi^S,\ell}:=\oplus_{i=1}^2\rho_i\otimes \epsilon_i$ (matching via $\mathrm{rec}$).

In particular, one can compare the Hecke eigenvalues at good places, as stated in Theorem~\ref{thm:igusa computation}. On the other hand, regarding the representation $\pi_j$ of $J_b(\Q_p)$, we recall that it is given by $\mathrm{Red}^b_{\vec{n}}(\pi^{\vec{n}}_p)$, and then follow the semisimple Langlands parameter through all normalized parabolic inductions and Langlands correspondences. Note that $\mathrm{Red}^b_{\vec{n}}$ is also by construction compatible with semisimple Langlands parameters.
\end{proof}

\newpage

\section{Boundary cohomology of Igusa varieties}\label{pink}

In this section, we compute the cohomology of the (partial) boundary of Igusa varieties.

\subsection{Statements}
More precisely, fix as usual a prime $p$ that is unramified in $F$. Fix any $p$-divisible group with $G$-structure $\mathbb X$ over an algebraically closed field $k$ of characteristic $p$. By abuse of notation, we write $\mathfrak{Ig}^b_{K(N)} = \mathfrak{Ig}^{\mathbb X}_{K^p(N)}$ for the associated Igusa variety, where $b=b(\mathbb X)\in B(G_{\Q_p},\mu^{-1})$ is as usual; here $N\geq 3$ is any integer prime to $p$. Let $\partial \mathfrak{Ig}^{b,*}\subset \mathfrak{Ig}^{b,*}$ be the boundary of the partial minimal compactification of the Igusa variety. It admits a natural stratification in terms of conjugacy classes $[P]$ of maximal rational parabolic subgroups $P\subsetneq G_\Q$: Note that a set of representatives for these are given by the stabilizer $P_r$ of the chain
\[
0\subset F^r\subset F^{2n-r}\subset F^{2n}
\]
for $r=1,\ldots,n$. Then the stratum $\mathfrak{Ig}^{b,*}_{[P]}\subset \partial \mathfrak{Ig}^{b,*}$ can be defined as the preimage of all strata $\mathscr S_Z\subset \mathscr S^\ast$ where the cusp label $Z=(\mathrm Z_N,X)$ has an $\cO_F$-lattice $X$ of rank $r$. The strata are naturally Hecke-equivariant.

We can pass to the inverse limit over all $N\geq 3$ prime to $p$: Let
\[
\mathfrak{Ig}_\infty^b = \varprojlim_N \mathfrak{Ig}^b_{K(N)}
\]
and define similarly $\mathfrak {Ig}^{b,*}_\infty$, $\mathfrak{Ig}^{b,\tor}_\infty$, $\partial \mathfrak{Ig}^{b,*}_\infty$ and $\mathfrak{Ig}^{b,*}_{\infty,[P]}$.

Let
\[
j: \mathfrak{Ig}^b_\infty\hookrightarrow \mathfrak{Ig}^{b,*}_\infty
\]
be the open immersion and
\[
i_{[P]}: \mathfrak{Ig}^{b,*}_{\infty,[P]}\hookrightarrow \mathfrak{Ig}^{b,*}_\infty
\]
the locally closed immersion. We consider
\[
R\Gamma_c( \mathfrak{Ig}^{b,*}_{\infty,[P]}, i_{[P]}^\ast Rj_\ast \mathbb F_\ell).
\]
This is naturally a complex of smooth representations of $J_b(\Q_p)\times G(\A_f^p)$ (where we note that both groups act naturally on all objects in the definition). Here and in the following, we use an extension of the notion of compactly supported cohomology to the case of schemes that admit integral maps to schemes of finite type, as has been defined in work of Hamacher \cite{hamacheretale}. (Concretely, this amounts to the colimit of the compactly supported cohomologies of an approximating tower with finite transition maps.)

The main theorem of this section is the following. For the statement, we fix the standard rational parabolic $P\in [P]$ given by the stabilizer of $0\subset F^r\subset F^{2n-r}\subset F^{2n}$. Its Levi group $M$ is given by $\mathrm{Res}_{F/\Q} \GL_r\times G_{2(n-r)}$ where $G_{2(n-r)}$ is the variant of $G$ with $n$ replaced by $n-r$. Let
\[
X_r = \left(\prod_{\tau: F^+\hookrightarrow \R} M_r^{\mathrm{herm},>0}(\C)\right)/\mathbb R_{>0}
\]
where $M_r^{\mathrm{herm},>0}(\C)$ denotes the space of positive definite hermitian matrices, which is the symmetric space for $\GL_r(F\otimes_\Q \R)$. Finally, fix a symplectic $\cO_F$-stable filtration
\[
\mathrm{Z}_b: 0\subset \mathrm{Z}_{b,-2}\subset \mathrm{Z}_{b,-1}\subset \mathbb X_b
\]
with an isomorphism $\mathrm{Z}_{b,-2}\cong \Hom(\mathcal O_F^r,\mu_{p^\infty})$. This induces a parabolic subgroup $P_b(\Q_p)\subset J_b(\Q_p)$ of the self-quasi-isogenies preserving this filtration, and a $p$-divisible group $\mathbb X_P = \mathrm{Z}_{b,-2}/\mathrm{Z}_{b,-1}$ with $G_{2(n-r)}$-structure; we denote by $b_P$ its  associated isocrystal with $G_{2(n-r)}$-structure. 

\begin{thm}\label{thm:boundary cohomology igusa} There is a natural $J_b(\Q_p)\times G(\A_f^p)$-equivariant isomorphism
\[
R\Gamma_c( \mathfrak{Ig}^{b,*}_{\infty,[P]}, i_{[P]}^\ast Rj_\ast \mathbb F_\ell)\cong \mathrm{Ind}_{P_b(\Q_p)\times P(\A_f^p)}^{J_b(\Q_p)\times G(\A_f^p)} R\Gamma(\GL_r(F)\backslash (X_r\times \GL_r(\A_{F,f})),\mathbb F_\ell)\otimes R\Gamma_c(\mathfrak{Ig}^{b_P}_\infty,\mathbb F_\ell).
\]
The action of $P_b(\Q_p)\times P(\A_f^p)$ on
\[
R\Gamma(\GL_r(F)\backslash (X_r\times \GL_r(\A_{F,f})),\mathbb F_\ell)\otimes R\Gamma_c(\mathfrak{Ig}^{b_P}_\infty,\mathbb F_\ell)
\]
is through its Levi quotient.
\end{thm}

We will prove this result in several steps. First, we realize the structure of parabolic induction geometrically and reduce to a statement about $P_b(\Q_p)\times P(\A_f^p)$-representations. Then we construct an equivariant map, as the cup product of two different maps, realizing the two factors. The factor $R\Gamma_c(\mathfrak{Ig}^{b_P}_\infty,\mathbb F_\ell)$ is easy, while the other factor $R\Gamma(\GL_r(F)\backslash (X_r\times \GL_r(\A_{F,f})),\mathbb F_\ell)$ will arise by realizing this real manifold inside the perfectoid space that is the punctured formal neighborhood of $\mathfrak{Ig}^{b,*}_{\infty,P}$. The latter is probably the most novel part of the proof. Finally, we show that the map is an isomorphism. For this, we use the explicit toroidal boundary charts. (We could have done this computation directly, but this would have obscured the Hecke-equivariance.)

In the last subsection, we deduce Theorem~\ref{thm:boundary cohomology} on possible systems of Hecke eigenvalues in the boundary. In particular, we show, under the same technical assumptions as in the last section, that all systems of Hecke eigenvalues appearing in the cohomology of Igusa varieties admit associated Galois representations.

\subsection{Construction of the map}

The goal of this section is to construct a natural map
\[
\mathrm{Ind}_{P_b(\Q_p)\times P(\A_f^p)}^{J_b(\Q_p)\times G(\A_f^p)} R\Gamma(\GL_r(F)\backslash (X_r\times \GL_r(\A_{F,f})),\mathbb F_\ell)\otimes R\Gamma_c(\mathfrak{Ig}^{b_P}_\infty,\mathbb F_\ell)\to R\Gamma_c( \mathfrak{Ig}^{b,*}_{\infty,[P]}, i_{[P]}^\ast Rj_\ast \mathbb F_\ell).
\]

\subsubsection{Parabolic induction}
First, we prove that
\[
R\Gamma_c( \mathfrak{Ig}^{b,*}_{\infty,[P]}, i_{[P]}^\ast Rj_\ast \mathbb F_\ell)
\]
is necessarily parabolically induced, by showing that the space $\mathfrak{Ig}^{b,*}_{\infty,[P]}$ is parabolically induced. More precisely, we note that there is a natural $J_b(\Q_p)\times G(\A_f^p)$-equivariant map
\[
\mathfrak{Ig}^{b,*}_{\infty,[P]}\to J_b(\Q_p)/P_b(\Q_p)\times G(\A_f^p)/P(\A_f^p)
\]
induced by cusp labels, as $J_b(\Q_p)/P_b(\Q_p)\times G(\A_f^p)/P(\A_f^p)$ parametrizes pairs $(\mathrm Z_b,\mathrm Z^p)$ ``of rank $r$''. Let
\[
\mathfrak{Ig}^{b,*}_{\infty,P}\subset \mathfrak{Ig}^{b,*}_{\infty,[P]}
\]
be the fibre over the identity. It is then formal that
\[
R\Gamma_c( \mathfrak{Ig}^{b,*}_{\infty,[P]}, i_{[P]}^\ast Rj_\ast \mathbb F_\ell) = 
\mathrm{Ind}_{P_b(\Q_p)\times P(\A_f^p)}^{J_b(\Q_p)\times G(\A_f^p)} R\Gamma_c(\mathfrak{Ig}^{b,*}_{\infty,P},i_P^\ast Rj_\ast \mathbb F_\ell)
\]
as $J_b(\Q_p)\times G(\A_f^p)$-representations.

\subsubsection{Cup products}
It remains to construct a map
\[
R\Gamma(\GL_r(F)\backslash (X_r\times \GL_r(\A_{F,f})),\mathbb F_\ell)\otimes R\Gamma_c(\mathfrak{Ig}^{b_P}_\infty,\mathbb F_\ell)\to R\Gamma_c(\mathfrak{Ig}^{b,*}_{\infty,P},i_P^\ast Rj_\ast \mathbb F_\ell)
\]
as a representation of $P_b(\Q_p)\times P(\A_f^p)$. To do this, we will construct two maps:
\begin{enumerate}
\item A map
\[
R\Gamma_c(\mathfrak{Ig}^{b_P}_\infty,\mathbb F_\ell)\to R\Gamma_c(\mathfrak{Ig}^{b,*}_{\infty,P},\mathbb F_\ell).
\]
\item A map
\[
R\Gamma(\GL_r(F)\backslash (X_r\times \GL_r(\A_{F,f})),\mathbb F_\ell)\to R\Gamma(\mathfrak{Ig}^{b,*}_{\infty,P},i_P^\ast Rj_\ast \mathbb F_\ell).
\]
\end{enumerate}
\noindent Both maps will by construction be $P_b(\A_f^p)\times G(\A_f^p)$-equivariant, where the action on the left-hand side will factor over $J_{b_P}(\Q_p)\times G_{2(n-r)}(\A_f^p)$ respectively $\GL_r(\A_{F,f})$, both of which are quotients of the Levi. The desired map will then arise as the cup product of the two maps, via the natural map
\[
R\Gamma_c(\mathfrak{Ig}^{b,*}_{\infty,P},\mathbb F_\ell)\otimes R\Gamma(\mathfrak{Ig}^{b,*}_{\infty,P},i_P^\ast Rj_\ast \mathbb F_\ell)\to R\Gamma_c(\mathfrak{Ig}^{b,*}_{\infty,P},i_P^\ast Rj_\ast \mathbb F_\ell).
\]

The first map comes directly from the profinite map
\[
\mathfrak{Ig}^{b,*}_{\infty,P}\to \mathfrak{Ig}^{b_P}_\infty
\]
obtained by passing to the limit in Theorem~\ref{thm:min Igusa strata}.

\subsubsection{Perfectoid magic}\label{sec:magic}
Thus, it remains to construct a natural equivariant map
\[
R\Gamma(\GL_r(F)\backslash (X_r\times \GL_r(\A_{F,f})),\mathbb F_\ell)\to R\Gamma(\mathfrak{Ig}^{b,*}_{\infty,P},i_P^\ast Rj_\ast \mathbb F_\ell).
\]
Note that this seems a priori tricky, as the space on the left is naturally only a real manifold, while on the right we have a perfect scheme. (We believe that one could obtain this result, including the desired Hecke equivariance, also by following Pink's original method, but we found the geometric arguments below more enlightening.)

Let us indicate how one can bridge the gap by using analytic geometry; in the present situation, we will naturally get perfectoid spaces. Let us first give an idealized version of the argument and then add in the details.

Let $\mathcal{I}g^b_{\infty,P}$ be the perfectoid space obtained from $\mathfrak{Ig}^{b,*}_\infty$ by taking the punctured formal neighborhood of $\mathfrak{Ig}^{b,*}_{\infty,P}$. This agrees with the perfectoid space obtained from $\mathfrak{Ig}^{b,\tor}_\infty$ by taking the punctured formal neighborhood of the preimage of $\mathfrak{Ig}^{b,*}_{\infty,P}$. Using the corresponding description of the toroidal boundary, and restricting for the moment to a specific cusp label $\tilde{Z}$, we see that over $\mathcal{I}g^b_{\infty,\tilde{Z}}$ we have the universal abelian variety $\cA$ as well as a Raynaud extension
\[
0\to \cT\to \cG\to \cB\to 0
\]
and a lift of the corresponding map $f_0: X\to \cB$ to a symmetric map $f: X\to \cG$; equivalently, a section of the Poincar\'e bundle $\cP\to \cB\times \cB$ over $X\times X$; such that $\cA=\cG/X$. Note that the principally polarized abelian variety $\cB$ has locally good reduction, and in particular the Poincar\'e bundle $\cP$ over $\cB\times \cB$ has a canonical integral structure. If one takes a point $x\in \mathcal{I}g^b_{\infty,\tilde{Z}}$ and fixes a norm $|\cdot|: K(x)\to \mathbb R_{\geq 0}$ on the residue field, then the map taking $x,y\in X$ to the logarithm of the norm of the section of $\cP$ gives a positive definite symmetric hermitian form
\[
X\times X\to \R;
\]
changing the norm $|\cdot|$ changes this only by a scalar. We see that there is a canonical map
\[
|\mathcal{I}g^b_{\infty,\tilde{Z}}|\to \GL_{\cO_F}(X)\backslash X_r.
\]
In fact, working globally and Hecke-equivariantly, we note that $X_\Q$ is a well-defined $F$-vector space of rank $r$ over all of $\mathcal{I}g^b_{\infty,P}$. Fixing an isomorphism $X_\Q\cong F^r$, we get (for each $x\in \mathcal{I}g^b_{\infty,P}$) a point of $X_r$. Moreover, we have level structures: In the present situation, this is an $\cO_F$-linear symplectic isomorphism $V_f^p(\cA)\cong V\otimes_{\Q} \A_f^p$ matching the subspace $V_f^p(\cT)$ with $F^r\otimes_{\Q} \A_f^p$. This gives another isomorphism $X_\Q\otimes_{\Q} \A_f^p\cong F^r\otimes_{\Q} \A_f^p$, i.e.~an element of $\GL_r(F\otimes_\Q \A_f^p)$. Similarly, the Igusa level structure induces an isomorphism $X_\Q\otimes_{\Q} \Q_p\cong F^r\otimes_{\Q} \Q_p$, i.e.~an element of $\GL_r(F\otimes_\Q \Q_p)$. In total, we get a natural continuous map
\[
f: |\mathcal{I}g^b_{\infty,P}|\to \GL_r(F)\backslash (X_r\times \GL_r(\A_{F,f}))
\]
and it follows from the construction that it is $J_b(\Q_p)\times P(\A_f^p)$-equivariant.

There is a natural map
\[
R\Gamma(\mathfrak{Ig}^{b,*}_{\infty,P},i_P^\ast Rj_\ast \mathbb F_\ell)\to R\Gamma(\mathcal{I}g^b_{\infty,P},\mathbb F_\ell)
\]
(as the left-hand side can be computed as the cohomology of the perfect scheme obtained by taking the henselization of $\mathfrak{Ig}^{b,*}_\infty$ along $\mathfrak{Ig}^{b,*}_{\infty,P}$ and deleting the boundary, and there is a natural map from $\mathcal{I}g^b_{\infty,P}$ to that scheme) that is probably an isomorphism. However, as $\mathcal{I}g^b_{\infty,P}$ is in general highly non-quasicompact -- in fact, at least as noncompact as the locally symmetric space $\GL_r(F)\backslash (X_r\times \GL_r(\A_{F,f}))$ -- it is nontrivial to justify this. If the map were an isomorphism, we would now get a natural map
\[\begin{aligned}
R\Gamma(\GL_r(F)\backslash (X_r\times \GL_r(\A_{F,f})),\mathbb F_\ell)&\buildrel{f^\ast}\over\to R\Gamma(\mathcal{I}g^b_{\infty,P},\mathbb F_\ell)\\
&\cong R\Gamma(\mathfrak{Ig}^{b,*}_{\infty,P},i_P^\ast Rj_\ast \mathbb F_\ell).
\end{aligned}\]

For this reason, we make the following small circumlocutions. Recall that the natural map
\[
\varinjlim_{K\subset \GL_r(\A_{F,f})} R\Gamma(\GL_r(F)\backslash (X_r\times \GL_r(\A_{F,f})/K),\mathbb F_\ell)\to R\Gamma(\GL_r(F)\backslash (X_r\times \GL_r(\A_{F,f})),\mathbb F_\ell)
\]
is an isomorphism, where $K$ runs over compact open subgroups of $\GL_r(\A_{F,f})$; this follows for example from the Borel--Serre compactification.\footnote{Indeed, let $X_r^{\mathrm{BS}}$ be the Borel--Serre compactification of $X_r$. Then both $X_r$ and $X_r^{\mathrm{BS}}$ are contractible. It follows that for any paracompact Hausdorff space $S$, one has \[R\Gamma(X_r\times S,A) = R\Gamma(S,A) = R\Gamma(X_r^{\mathrm{BS}}\times S,A)\]
for any coefficient module $A$; here cohomology means \v{C}ech (or equivalently sheaf) cohomology. By descent along the $\GL_r(F)$-quotient, it follows that the natural map
\[
R\Gamma(\GL_r(F)\backslash (X_r^{\mathrm{BS}}\times \GL_r(\A_{F,f})/H),A)\to R\Gamma(\GL_r(F)\backslash (X_r\times \GL_r(\A_{F,f})/H),A)
\]
is an isomorphism, for any closed subgroup $H\subset \GL_r(\A_{F,f})$. On the other hand, the quotients $\GL_r(F)\backslash (X_r^{\mathrm{BS}}\times \GL_r(\A_{F,f})/H)$ are compact Hausdorff, and the natural maps
\[\GL_r(F)\backslash (X_r^{\mathrm{BS}}\times \GL_r(\A_{F,f})/H)\to \varprojlim_{K\supset H} \GL_r(F)\backslash (X_r^{\mathrm{BS}}\times \GL_r(\A_{F,f})/K)
\]
are homeomorphisms (as continuous bijections between compact Hausdorff spaces), where $K$ runs through open compact subgroups containing $H$. This implies that on \v{C}ech cohomology,
\[
\varinjlim_{K\supset H} R\Gamma(\GL_r(F)\backslash (X_r^{\mathrm{BS}}\times \GL_r(\A_{F,f})/K),A)\to R\Gamma(\GL_r(F)\backslash (X_r^{\mathrm{BS}}\times \GL_r(\A_{F,f})/H),A)
\]
is an isomorphism.}

Similarly, for each finite $m$ and $N\geq 3$ prime to $p$, let
\[
\mathrm{Ig}^{b,*}_{m,K(N),P}\subset \mathrm{Ig}^{b,*}_{m,K(N)}
\]
be the image of $\mathfrak{Ig}^{b,*}_{\infty,P}$, and let $\mathcal{I}g^{b}_{m,K(N),P}$ be formed similarly as the punctured formal completion, which is an analytic adic space. Then it follows from \cite[Corollary 3.5.14]{huber} applied at each of these finite levels, plus passage to the limit on the left-hand side, that
\[
R\Gamma(\mathfrak{Ig}^{b,*}_{\infty,P},i_P^\ast Rj_\ast \mathbb F_\ell)\cong \varinjlim_{m,N} R\Gamma(\mathcal{I}g^b_{m,K(N),P},\mathbb F_\ell).
\]
Moreover, the map
\[
f: |\mathcal{I}g^b_{\infty,P}|\to \GL_r(F)\backslash (X_r\times \GL_r(\A_{F,f}))
\]
is the limit of a map of pro-systems
\[
\{|\mathcal{I}g^b_{m,K(N),P}|\}_{m,N}\to \{\GL_r(F)\backslash (X_r\times \GL_r(\A_{F,f})/K)\}_K.
\]
Thus, passing to the colimit, we get the desired map
\[\begin{aligned}
R\Gamma(\GL_r(F)\backslash (X_r\times \GL_r(\A_{F,f})),\mathbb F_\ell)&\cong \varinjlim_{K\subset \GL_r(\A_{F,f})} R\Gamma(\GL_r(F)\backslash (X_r\times \GL_r(\A_{F,f})/K),\mathbb F_\ell)\\
&\to \varinjlim_{m,N} R\Gamma(\mathcal{I}g^b_{m,K(N),P},\mathbb F_\ell)\\
&\cong R\Gamma(\mathfrak{Ig}^{b,*}_{\infty,P},i_P^\ast Rj_\ast \mathbb F_\ell).
\end{aligned}\]
On the level of pro-systems (and passing to perfections where appropriate), everything carries natural $P_b(\Q_p)\times P(\A_f^p)$-actions for which the maps are equivariant, implying the desired equivariance.

\subsection{Local computation}

In this section, we do a local computation with toroidal boundary charts to finish the proof of Theorem~\ref{thm:boundary cohomology igusa}.

Note that
\[
\mathfrak{Ig}^{b,*}_{\infty,P}
\]
admits a further decomposition according to Igusa cusp labels above $P$. Concretely, this amounts to finite projective $\cO_F$-modules $X$ of rank $r$ with isomorphisms $X\otimes_{\Z} \widehat{\Z}^p\cong \cO_F^r\otimes_{\Z} \widehat{\Z}^p$ and $X\otimes_{\Z} \Z_p\cong \cO_F^r\otimes_{\Z} \Z_p$ (as we have fixed an isomorphism $\mathrm Z_{b,-2}\cong \Hom(\cO_F^r,\mu_{p^\infty})$). Noting that we are really only mapping to the limit over all levels $K$ of $K$-equivalence classes of cups labels, we get a decomposition according to
\[
\bigsqcup_{X/\cong} \overline{\GL_{\cO_F}(X)}\backslash \GL_{\cO_F}(X\otimes_{\Z}\widehat{\Z}).
\]
On the other hand, in the last section, we constructed a map from the punctured formal neighborhoods to
\[
\GL_r(F)\backslash (X_r\times \GL_r(\A_{F,f})),
\]
and in particular to $\overline{\GL_r(F)}\backslash \GL_r(\A_{F,f})$. Note that the natural map
\[
\bigsqcup_{X/\cong} \overline{\GL_{\cO_F}(X)}\backslash \GL_{\cO_F}(X\otimes_{\Z}\widehat{\Z})\to \overline{\GL_r(F)}\backslash \GL_r(\A_{F,f})
\]
is a bijection. It is readily seen that these two maps are compatible.

Thus, from now on we fix a finite projective $\cO_F$-module $X$ of rank $r$ with an isomorphism $X\otimes_{\Z} \widehat{\Z}\cong \cO_F^r\otimes_{\Z} \widehat{\Z}$, corresponding to an Igusa cusp label $\tilde{Z}$, and let
\[
\mathfrak{Ig}^{b,*}_{\infty,\tilde{Z}}\subset \mathfrak{Ig}^{b,*}_{\infty,P}
\]
be the corresponding closed subset, defined as the inverse limit of the corresponding strata at finite level. Note that the corresponding part of the cohomology of $\GL_r(F)\backslash(X_r\times \GL_r(\A_{F,f}))$ is given by
\[
\varinjlim_{\Gamma\subset \GL_{\cO_F}(X)} R\Gamma(\Gamma,\mathbb F_\ell),
\]
where $\Gamma$ runs through congruence subgroups of $\GL_{\cO_F}(X)$ (recall that $X_r$ is contractible). Thus, we have to prove the following statement.

\begin{prop} The map
\[
R\Gamma_c(\mathfrak{Ig}^{b_P}_\infty,\mathbb F_\ell)\otimes \varinjlim_{\Gamma\subset \GL_{\cO_F}(X)} R\Gamma(\Gamma,\mathbb F_\ell)\to R\Gamma_c(\mathfrak{Ig}^{b,*}_{\infty,\tilde{Z}},i_{\tilde{Z}}^\ast Rj_\ast \mathbb F_\ell)
\]
constructed in the previous section is an isomorphism.
\end{prop}

\begin{proof} Using the toroidal compactification to do the computation and the explicit description of the toroidal boundary charts from Theorem~\ref{thm:tor Igusa strata}, this is a by now standard computation due to Pink, \cite[Theorem 4.2.1]{pink}, see also~\cite[Theorem 4.3.10]{lan-stroh}. The key point is to notice that the inverse system $\left(\mathrm{Ig}^{b,\tor}_{K^p}\right)_{K^p}$ of partial toroidal compactifications of Igusa varieties has the axiomatic properties described in~\cite[Lemma 4.3.2]{lan-stroh}.
\end{proof}

This finishes the proof of Theorem~\ref{thm:boundary cohomology igusa}.

\begin{remark} Theorem~\ref{thm:boundary cohomology igusa} can be interpreted as a version of Pink's formula~\cite{pink} for Igusa varieties. A similar argument as in \S~\ref{sec:magic} in the setting of Shimura varieties can be used to prove that the original version of Pink's formula is Hecke-equivariant. This provides an alternative to the argument in~\cite[\S 4.8]{pink}. 
\end{remark}

\subsection{Applications}

We now use Theorem~\ref{thm:boundary cohomology igusa} to construct Galois representations associated to maximal ideals $\m\subset \mathbb{T}$ in the support of $R\Gamma_{(c-\partial)}(\mathrm{Ig}^{b}, \mathbb F_\ell)$.

From now on, assume that $F$ contains (properly) an imaginary quadratic field $F_0\subset F$ in which $p$ splits and fix the finite set $S$ of places of $\Q$ and a level $N\geq 3$ prime to $p$ as in the last section. We let $\mathbb T^S$ be the same unramified Hecke algebra, and fix a maximal ideal $\mathfrak m\subset \mathbb T^S$ containing $\ell$.

We will consider the Igusa varieties with implicit tame level $K^p(N)$.

\begin{thm}\label{thm:Galreps Igusa} Assume that for some $b\in B(G_{\Q_p},\mu^{-1})$ one of the cohomology groups
\[
H^i_{c-\partial}(\mathrm{Ig}^b,\mathbb F_\ell)_{\mathfrak m},H^i(\mathrm{Ig}^b,\mathbb F_\ell)_{\mathfrak m}
\]
is nonzero. Then there exists a continuous semisimple Galois representation
\[
\overline\rho_{\mathfrak m}: \mathrm{Gal}(\overline{F}/F)\to \GL_{2n}(\overline{\mathbb F}_\ell)
\]
such that for all primes $v$ dividing a rational prime $q\not\in S$ that splits in $F_0$, the characteristic polynomial of $\overline\rho_{\mathfrak m}(\Frob_v)$ is given by
\[
X^{2n} - T_{1,v}X^{2n-1}+\dots+ (-1)^iq_v^{i(i-1)/2}T_{i,v}X^{2n-i}+\dots + q_v^{n(2n-1)}T_{2n,v}
\]
with notation as in the introduction.

Moreover, if the map
\[
H^i_{c-\partial}(\mathrm{Ig}^b,\mathbb F_\ell)_{\mathfrak m}\to H^i(\mathrm{Ig}^b,\mathbb F_\ell)_{\mathfrak m}
\]
is not an isomorphism and $b$ is not ordinary, then $\overline{\rho}_{\mathfrak m}$ is of length at least $3$.
\end{thm}

Before proving the theorem, we record some preliminary notation and results. We denote by 
\[
R\Gamma_{\cont}(K^p(N), ): D^+_{\mathrm{sm}}(G(\A_f^p), \F_{\ell})\to D^+(\mathbb{T}^S)
\]
the derived functor of $K^p(N)$-invariants in the bounded below derived category of smooth $G(\A_f^p)$-representations with $\F_{\ell}$-coefficients. We also consider the forgetful functor $D^+(\mathbb{T}^S)\to D^+(\Z)$. After applying this forgetful functor to the left hand sides, we obtain $\mathbb{T}^S$-equivariant isomorphisms 
\begin{equation}\label{eq:HS1}
R\Gamma_{\cont}(K^p(N), R\Gamma_c(\mathfrak{Ig}^b_{\infty}, \F_{\ell}))\toisom R\Gamma_c(\mathrm{Ig}^b, \F_{\ell})
\end{equation}
and 
\begin{equation}\label{eq:HS2}
R\Gamma_{\cont}(K^p(N), R\Gamma_c(\mathfrak{Ig}^{b,*}_{\infty}, Rj_\ast\F_{\ell})) \toisom R\Gamma_c(\mathrm{Ig}^{b,*}, Rj_\ast\F_{\ell}). 
\end{equation}
in $D^+(\Z)$. 

We set $K^S_P:=K^S\cap P(\A^S)$ and consider the corresponding abstract Hecke algebra $\mathbb{T}^S_P$ for $P$. Recall from~\cite[\S 2.2.3]{newton-thorne} the ring
homomorphism $r_P: \mathbb{T}^S\to \mathbb{T}^S_P$ given by ``restriction of functions''. The following is a version of~\cite[Corollary 2.6]{newton-thorne} for smooth representations. 

\begin{lem}\label{lem:restriction of functions} There is a natural isomorphism of functors
\[
R\Gamma_{\cont}\left(K^S, \mathrm{Ind}_{P(\mathbb{A}^S)}^{G(\mathbb{A}^S)}\ \right) \toisom r_P^*R\Gamma_{\cont}(K^S_P, \ )
\]
from $D^+_{\mathrm{sm}}(P(\mathbb{A}^S), \F_{\ell})$ to $D^+(\mathbb{T}^S_P)$.
\end{lem}

\begin{proof} Since $K^S$ is hyperspecial for every place $v\not\in S$ of $F$, we have the 
Iwasawa decomposition $G(\A^S) = K^S \cdot P(\A^S)$. For non-derived functors, we obtain the commutative diagram
\[
\xymatrix{\mathrm{Mod}_{\mathrm{sm}}(G(\A^S), \F_{\ell})\ar[rr]^{\Gamma(K^S,\ )} &\ & \mathrm{Mod}(\mathbb{T}^S) \\ 
\mathrm{Mod}_{\mathrm{sm}}(P(\A^S), \F_{\ell})\ar[u]^{\mathrm{Ind}_{P(\A^S)}^{G(\A^S)}}
\ar[rr]^{\Gamma(K^S_P,\ )} &\ & \mathrm{Mod}(\mathbb{T}^S_P)\ar[u]^{r^*_P},
}
\]
where the smooth induction $\mathrm{Ind}_{P(\A^S)}^{G(\A^S)}$ is exact and preserves injectives, 
$r_P^*$ is exact, and $\Gamma(K^S,\ )$ and $\Gamma(K^S_P,\ )$ are left exact. The proof that the diagram commutes 
is identical to that of~\cite[Lemma 2.4 (iii)]{newton-thorne}. Taking derived functors, we obtain the lemma. 
\end{proof}

Let $M=\mathrm{Res}_{F/\Q} \GL_r\times G_{n-r,\Q}$ be the Levi and $N$ be the unipotent radical of $P$; we have a Levi decomposition
$P=M\rtimes N$. The compact open subgroup $K^S_P\subset P(\A^S)$
is decomposed in the sense of~\cite[\S 2.2.4]{newton-thorne} with respect to this Levi decomposition. 
We set $K^S_M:=K^S_P\cap M(\A^S)$ and $K^S_N:=K^S_P\cap N(\A^S)$. 
Recall also from~\cite[\S 2.2.4]{newton-thorne} the ring homomorphism 
$r_M: \mathbb{T}^S_P\to \mathbb{T}^S_M$ given by ``integration along 
unipotent fibers''. The following is an analogue of~\cite[Corollary 2.8]{newton-thorne} for smooth representations. 

\begin{lem}\label{lem:integration along unipotent fibers}
There is a natural isomorphism of functors
\[
r_M^* R\Gamma_{\cont}(K_M, ) \toisom R\Gamma_{\cont}(K^S_P, \mathrm{Inf}_{M(\mathbb{A}^S)}^{P(\mathbb{A}^S)}\ )
\]
from  $D^+_{\mathrm{sm}}(M(\mathbb{A}^S), \F_{\ell})$ to $D^+(\mathbb{T}^S_P)$.  
\end{lem}

\begin{proof}
On the level of non-derived functors, we have the commutative diagram 
\[
\xymatrix{\mathrm{Mod}_{\mathrm{sm}}(P(\A^S), \F_{\ell})\ar[rr]^{\Gamma(K^S_P,\ )} &\ & \mathrm{Mod}(\mathbb{T}^S_P) \\ 
\mathrm{Mod}_{\mathrm{sm}}(M(\A^S), \F_{\ell})\ar[u]^{\mathrm{Inf}_{M(\A^S)}^{P(\A^S)}}
\ar[rr]^{\Gamma(K^S_M,\ )} &\ & \mathrm{Mod}(\mathbb{T}^S_M)\ar[u]^{r^*_M},
}
\]
where the vertical functors are exact and the horizontal functors are left exact. 
We do not know that $\mathrm{Inf}_{M(\mathbb{A}^S)}^{P(\mathbb{A}^S)}$ preserves injectives,
but we get a natural transformation of functors 
\[
r_M^* R\Gamma_{\cont}(K_M, ) \to R\Gamma_{\cont}(K^S_P, \mathrm{Inf}_{M(\mathbb{A}^S)}^{P(\mathbb{A}^S)}\ ) 
\]
by~\cite[Lemma 2.1]{newton-thorne}. 

Now observe that the above natural transformation becomes an isomorphism on any object 
of $D^+_{\mathrm{sm}}(M(\mathbb{A}^S), \F_{\ell})$. It is enough to check this 
after applying functors that forget the Hecke action. We have 
$R\Gamma_{\cont}(K^S_P, \ ) = R\Gamma_{\cont}(K^S_M,\ )\circ \Gamma(K^S_N, \ )$, 
because $\Gamma(K^S_N,\ )$ preserves injectives and is exact on   
$D^+_{\mathrm{sm}}(P(\mathbb{A}^S), \F_{\ell})$. This last part follows
from the fact that $K^S_N$ is profinite with order prime to $\ell$.  
On the other hand, $\Gamma(K^S_N, \ )\circ \mathrm{Inf}_{K^S_M}^{K^S_P} = \mathrm{Id}$. 
\end{proof}

\begin{proof}[Proof of Theorem~\ref{thm:Galreps Igusa}] 
We argue by induction on $n$, so we may assume that the analogous result is known for Igusa varieties on smaller unitary groups. Assume first that the map
\[
H^i_{c-\partial}(\mathrm{Ig}^b,\mathbb F_\ell)_{\mathfrak m}\to H^i(\mathrm{Ig}^b,\mathbb F_\ell)_{\mathfrak m}
\]
is not an isomorphism. In particular, this happens at some finite Igusa level $p^m$, and then by Poincar\'e duality the map
\[
H_c^i(\mathrm{Ig}^b,\mathbb F_\ell)_{\mathfrak m^\vee}\to H_c^i(\mathrm{Ig}^{b,*},Rj_\ast \mathbb F_\ell)_{\mathfrak m^\vee}
\]
is not an isomorphism for the ``dual'' set of Hecke eigenvalues $\mathfrak m^\vee$. 

From the $\mathbb{T}^S$-equivariant isomorphisms~\eqref{eq:HS1} and~\eqref{eq:HS2}, and by considering the stratification of the boundary of $\mathfrak{Ig}^{b,*}_{\infty}$ 
in terms of the conjugacy classes $[P]$ of rational parabolic subgroups of $G$, 
we deduce that for some such $[P]$ corresponding to an integer $1\leq r\leq n$,
\[
R\Gamma_{\cont}(K^p(N),R\Gamma_c( \mathfrak{Ig}^{b,*}_{\infty,[P]}, i_{[P]}^\ast Rj_\ast \mathbb F_\ell))_{\mathfrak m^\vee}\neq 0.
\]
In particular,
\begin{equation}\label{eq:induction}
R\Gamma_{\cont}(K^S,R\Gamma_c( \mathfrak{Ig}^{b,*}_{\infty,[P]}, i_{[P]}^\ast Rj_\ast \mathbb F_\ell))_{\mathfrak m^\vee}\neq 0
\end{equation}
where $K^S = G(\cO_F\otimes_\Z \widehat{\Z}^S)$. 

We now wish to apply Theorem~\ref{thm:boundary cohomology igusa}
in order to show that $\mathfrak m^\vee$ must be pulled back under the unnormalized Satake transform $\mathbb{T}^S\to \mathbb{T}^S_M$
from a maximal ideal of the Hecke algebra $\mathbb{T}^S_M$. We will do this in two steps, going via the Hecke algebra $\mathbb{T}^S_P$ 
of $P$. By Theorem~\ref{thm:boundary cohomology igusa} and Lemma~\ref{lem:restriction of functions}, we can rewrite~\eqref{eq:induction} as the localization at $\mathfrak m^\vee$ of 
\begin{equation}\label{eq:inflation}
r_P^* R\Gamma_{\cont}\left(K_P^S, \mathrm{Inf}_{M(\mathbb{A}^S)}^{P(\mathbb{A}^S)} \mathrm{Ind}_{P_b(\Q_p)\times P(F\otimes_\Q \A_{S_f^p})}^{J_b(\Q_p)\times G(F\otimes_\Q \A_{S_f^p})}
 R\Gamma(\GL_r(F)\backslash (X_r\times \GL_r(\A_{F,f})),\mathbb F_\ell)\otimes R\Gamma_c(\mathfrak{Ig}^{b_P}_\infty,\mathbb F_\ell)\right).
\end{equation}
By Lemma~\ref{lem:integration along unipotent fibers}, we can rewrite~\eqref{eq:inflation} as the localization at $\mathfrak m^\vee$ of 
\begin{equation}\label{eq:satake}
r_M^*r_P^* R\Gamma_{\cont}\left(K_M^S, \mathrm{Ind}_{P_b(\Q_p)\times P(F\otimes_\Q \A_{S_f^p})}^{J_b(\Q_p)\times G(F\otimes_\Q \A_{S_f^p})}
 R\Gamma(\GL_r(F)\backslash (X_r\times \GL_r(\A_{F,f})),\mathbb F_\ell)\otimes R\Gamma_c(\mathfrak{Ig}^{b_P}_\infty,\mathbb F_\ell)\right).  
\end{equation}
After applying the forgetful functor to $D^+(\Z)$, and using~\eqref{eq:HS1} for $\mathrm{Ig}^{b_P}$ and the analogue for the locally symmetric spaces for $\GL_r$, we see that~\eqref{eq:satake} is $\mathbb{T}^S$-equivariantly isomorphic to 
\[
r_M^*r_P^* \mathrm{Ind}_{P_b(\Q_p)\times P(F\otimes_\Q \A_{S_f^p})}^{J_b(\Q_p)\times G(F\otimes_\Q \A_{S_f^p})} R\Gamma(\GL_r(F)\backslash(X_r\times \GL_r(\A_{F,f})/K^{S,\GL_r}),\mathbb F_\ell)\otimes R\Gamma_c(\mathrm{Ig}^{b_P}_{K^{S,G_{2(n-r)}}},\mathbb F_\ell), 
\]
where $K^{S,\GL_r} = \GL_r(\cO_F\otimes_{\Z} \widehat{\Z}^S)$ and similarly $K^{S,G_{2(n-r)}} = G_{2(n-r)}(\cO_F\otimes_\Z \widehat{\Z}^S)$
(satisfying $K^S_M = K^{S,\GL_r}\times K^{S, G_{2(n-r)}}$). 

Note that the composition $r_M\circ r_P: \mathbb{T}^S\to \mathbb{T}^S_M$ is the unnormalised Satake transform. 
Using the decomposition $M=\mathrm{Res}_{F/\Q}\GL_r\times G_{n-r,\Q}$, we find that there are maximal ideals $\mathfrak m_1$ of the Hecke algebra for $\mathrm{Res}_{\cO_F/\Z}\GL_r$ and $\mathfrak m_2$ of the Hecke algebra for $G_{2(n-r)}$ such that
\[
R\Gamma(\GL_r(F)\backslash(X_r\times \GL_r(\A_{F,f})/K^{S,\GL_r}),\mathbb F_\ell)_{\mathfrak m_1}\neq 0
\]
and
\[
R\Gamma_c(\mathrm{Ig}^{b_P}_{K^{S,G_{2(n-r)}}},\mathbb F_\ell)_{\mathfrak m_2}\neq 0,
\]
and $\mathfrak m^\vee$ maps into a maximal ideal containing $\mathfrak m_1\otimes \mathfrak m_2\subset \mathbb T_M$. Using \cite[Corollary 5.4.3]{scholze}, the induction hypothesis, and Poincar\'e duality, we get the desired Galois representation $\overline{\rho}_{\m}$ associated to $\m$. (Alternatively, one could appeal to~\cite[Theorem 2.3.5]{10Authors}, which relies on the unconditional base change result of Shin~\cite{shin-basechange},  instead of appealing to~\cite[Corollary 5.4.3]{scholze}.)

If $r<n$, $\overline{\rho}_{\m}$ will be a direct sum of three representations: Note that the representation from $\GL_r$ will contribute two summands. If $b$ is not ordinary, then $P=P_n\subset G$ does not contribute a stratum to the Igusa variety (as there will be no Igusa cusp labels where the rank of $X$ is $n$), justifying the final statement.

Thus, we are left with the case that
\[
H^i_{c-\partial}(\mathrm{Ig}^b,\mathbb F_\ell)_{\mathfrak m}\to H^i(\mathrm{Ig}^b,\mathbb F_\ell)_{\mathfrak m}
\]
is an isomorphism. Possibly changing $b$, we can then ensure, using Proposition~\ref{concentration usual cohomology} and Corollary~\ref{concentration compactly supported cohomology}, that these cohomology groups are concentrated in one degree. Now Theorem~\ref{thm:igusa computation} finishes the proof.
\end{proof}

\newpage

\bibliographystyle{amsalpha} 
\bibliography{Torsionvanishing} 

\end{document}